\documentclass[11pt,a4paper]{article}
\usepackage[T1]{fontenc}
\usepackage[utf8]{inputenc}
\usepackage{amsmath}
\usepackage{mathtools}
\usepackage{amsfonts}
\usepackage{amsthm}
\usepackage{amssymb}
\usepackage{bm}
\usepackage[colorlinks,unicode]{hyperref}
\usepackage{url}
\usepackage{color, xcolor}
\usepackage{psfrag, float}  

\usepackage{verbatim}   
\usepackage{booktabs}	
\usepackage{grffile}    

\newtheorem{ansatz}{Ansatz}
\newtheorem{lemma}{Lemma}[section]
\newtheorem{theorem}[lemma]{Theorem}
\newtheorem{corollary}[lemma]{Corollary}
\newtheorem{definition}[lemma]{Definition}

\newtheorem{notation}[lemma]{Notation}
\newtheorem{proposition}[lemma]{Proposition}

\theoremstyle{remark}
\newtheorem{remark}[lemma]{Remark}

\newcommand{\eps}{\varepsilon}
\newcommand{\ga}{\gamma}
\newcommand{\dps}{\displaystyle}
\newcommand{\RR}{\mathbb{R}}
\newcommand{\R}{\mathbb{R}}
\newcommand{\NN}{\mathbb{N}}
\newcommand{\CC}{\mathbb{C}}
\newcommand{\TT}{\mathbb{T}}
\newcommand{\ZZ}{\mathbb{Z}}
\newcommand{\MM}{\mathcal{M}}

\newcommand{\PP}{\mathcal{P}}
\newcommand{\GG}{\mathcal{G}}
\newcommand{\II}{\mathcal{I}}
\newcommand{\BB}{\mathcal{B}}
\newcommand{\LL}{\mathcal{L}}

\newcommand{\KK}{\mathcal{K}}
\newcommand{\SSS}{\mathcal{S}}
\newcommand{\DD}{\mathcal{D}}

\newcommand{\OO}{\mathcal{O}}
\newcommand{\FF}{\mathcal{F}}
\newcommand{\HH}{\mathcal{H}}

\newcommand{\RRR}{\mathcal{R}}

\newcommand{\UU}{\mathcal{U}}
\newcommand{\NNN}{\mathcal{N}}
\newcommand{\EE}{\mathcal{E}}
\newcommand{\JJ}{\mathcal{J}}
\newcommand{\E}{\mathbb{E}}
\newcommand{\bE}{\mathbf{E}}
\newcommand{\AAA}{\mathcal{A}}
\newcommand{\tA}{\mathtt{A}}
\newcommand{\tJ}{\mathbf{J}}
\newcommand{\bG}{\mathbf{G}}
\newcommand{\tD}{\mathtt{D}}

\newcommand{\CCC}{\mathcal{C}}
\newcommand{\Id}{\mathrm{Id}}

\newcommand{\bms}{\bm{\sigma}}
\newcommand{\bmmu}{\bm{\mu}}
\newcommand{\bbb}{\bm{b}}

\newcommand{\ii}{^{-1}}
\newcommand{\de}{\delta}
\newcommand{\pa}{\partial}
\newcommand{\ee}{e_0}
\newcommand{\la}{\lambda}

\newcommand{\inn}{\mathrm{in}}
\newcommand{\out}{\mathrm{out}}

\newcommand{\kk}{\kappa}
\newcommand{\rr}{\rho}

\newcommand{\ol}{\overline}

\newcommand{\al}{\alpha}
\newcommand{\ero}{\zeta}
\newcommand{\om}{\omega}

\renewcommand{\Re}{\mathrm{Re\, }}

\newcommand{\wt}{\widetilde}
\newcommand{\wh}{\widehat}

\newcommand{\g}{\hat g}
\newcommand{\Lb}{\Lambda}
\newcommand{\lb}{\lambda}
\newcommand{\ccirc}{\mathrm{circ}}
\newcommand{\eell}{\mathrm{ell}}

\newcommand{\SM}{\mathcal{SM}}

\newcommand{\Prob}{\mathrm{Prob}}
\newcommand{\iinn}{\mathrm{inn}}
\newcommand{\im}{\mathrm{i}}
\newcommand{\ecc}{\mathbf{e}}
\newcommand{\tm}{\mathtt{m}}

\newcommand{\Jminusnew}{-1.581}   
\newcommand{\Jplusnew}{-1.485} 

\usepackage{hyperref}
\usepackage{subcaption} 

\mathtoolsset{showonlyrefs}

\begin{document}
	
	\title{Stochastic behavior along  mean motion resonances in 
		the restricted planar 3-body problem}
	\author{
		M. Guardia\footnote{Universitat de Barcelona \& Centre de Recerca Matem\`atica,
			guardia@ub.edu}, \ \ 
		V. Kaloshin\footnote{Institute of Science and Technology, Austria 
			vadim.kaloshin@gmail.com},\ \  
		P. Mart\'in \footnote{Universitat Polit\`ecnica de Catalunya \& Centre de Recerca Matem\`atica, p.martin@upc.edu},
		\ \
		P. Roldan, \footnote{Universitat Polit\`ecnica de Catalunya, pablo.roldan@upc.edu}
	} 
	
	\maketitle

	\begin{abstract}
		
		One of the most remarkable instability zones in the Solar system are Kirkwood gaps in the asteroid belt.  In this paper we analyze instabilities in the famous Kirkwood gap $3:1$ in the regime of small eccentricity of Jupiter. Mathematically speaking, we study the evolution of asteroids under the influence of the Sun and Jupiter using the restricted planar elliptic 3-body problem (RPE3BP) for initial conditions near a mean motion resonance 3:1. The main result exhibits stochastic diffusing behavior of the eccentricity of the asteroid for a rich set of initial conditions. Roughly speaking, for small eccentricity $e_0$ of Jupiter, the evolution of the eccentricity of the asteroid $\ecc(t\cdot e_0^{-2})$  at the Kirkwood gap $3:1$ behaves like a diffusion process on the line, where the randomness comes from the initial conditions.
		
		\smallskip
		Along with KAM theory, we have mixed behavior  in the asteroid belt, that is coexistence of quasiperiodic (deterministic) and stochastic (diffusive) behavior.
		
		\smallskip
		
		Our proof consists of four main steps:
		\smallskip
		
		1. {\it (Separatrix map)} Compute the first and second orders of the perturbation theory of the separatrix map for the RPE3BP near 3:1 resonance with respect to small eccentricity of Jupiter.  
		
		\smallskip
		
		2. {\it (Normally Hyperbolic Invariant Lamination) } Establish the existence of a normally hyperbolic invariant lamination $\Lambda$ for the time one map of RPE3BP, which is a weakly invariant object locally given as the product of a Cantor set and a 2-dimensional cylinder. 
		
		\smallskip
		
		3. {\it (Skew-shift model)} The dynamics restricted to $\Lambda$, after a change of coordinates, is given by a skew-shift. In the simplified form we have  \newline \smallskip 
		$\qquad f:(\omega,I,\theta)\to 
		(\sigma \omega,I+e_0 A_\omega(I)
		\cos (\theta+\psi_\omega)
		+\OO(e^2_0),\theta+\Omega_\omega(I)+\OO(e_0)),$ 
		\newline
		\smallskip 
		where $ I\in [a,b],\  \theta\in \mathbb T,\ \omega\in\Sigma=\{0,1\}^\mathbb Z$,  the space of sequences of $0,1$’s, $\sigma:\Sigma \to \Sigma$ is the shift in this space, 
		$\Omega_\omega$ is the shear,\ $A_\omega$ is an amplitude, and 
		$e_0$ is the eccentricity of Jupiter, which is taken as a small parameter,
		\smallskip

		4. {\it (Martingale analysis)} Study the evolution of a probability measure $\mu$ supported on $\Lambda$ and given as a product of a Bernoulli measure on $\Sigma$ and either the Lebesgue measure or an atomic measure on $\mathbb T$.  A delicate martingale analysis shows that the dynamics restricted to $\Lambda$ exhibits diffusing behavior of the eccentricity of the orbit of the  asteroid. This martingale analysis strongly relies on decay of correlations of compact Lie group extensions of hyperbolic maps and a special Central Limit Theorem for sums of dependent random variables.
		
		\medskip

		
		%
		%
	\end{abstract}
	

	\tableofcontents
	
	\section{Introduction}\label{sec:introduction}
	
	The stability of the Solar System is a longstanding problem. 
	Over the centuries, mathematicians and astronomers have 
	spent an inordinate amount of energy proving stronger and 
	stronger stability theorems for dynamical systems closely 
	related to the Solar System, generally within the framework of 
	the Newtonian $N$-body problem:
	\[
	\ddot q_i = \sum_{j \ne i} m_j \frac{q_j-q_i}{\|q_j-q_i\|^3}, 
	\qquad q_i\in \RR^2,\quad m_i>0, \qquad  i = 0, 1,\dots , N - 1,
	\]
	and its planetary subproblem, where the mass $m_0$ (modeling 
	the Sun) is much larger than the other masses $m_i$ (the planets).

	Letting the masses $m_i$ tend to $0$ for $i=2,\dots ,N - 1$, we 
	obtain a collection of $N-2$ independent {\it restricted problems}:
	\[
	\ddot q_i = m_0 \frac{q_0-q_i}{\|q_0-q_i\|^3}+
	m_1 \frac{q_0-q_1}{\|q_0-q_1\|^3}, \qquad q_i\in \RR^2, \qquad  
	i = 2,\dots , N - 1.
	\]
	where the massless bodies are influenced by, without 
	themselves influencing, the primaries of masses $m_0$ and 
	$m_1$. For $N = 3$, this model is often used to approximate 
	the dynamics of the Sun--Jupiter--asteroid or other 
	Sun-planet-object problems, and it is the simplest 
	one conjectured to have a wide range of instabilities.
	
	\subsection{An example of relevance in astronomy: 
		The asteroid belt and the restricted planar 3-body problem} 
	One place in the Solar System where the dynamics is well
	approximated by the restricted 3-body problem is 
	the asteroid belt. The asteroid belt is located between the orbits of 
	Mars and Jupiter and consists of over one million objects ranging 
	from asteroids of 950 kilometers to one kilometer in diameter bodies. 
	The total mass of the asteroid belt is estimated to be 3\% that 
	of the Moon. Since the mass of Jupiter is approximately 2960 
	masses of Mars, away from close encounters with Mars, 
	one can neglect its influence on the asteroids and 
	focus on the influence of Jupiter. We also omit interactions 
	with the second biggest planet in  the Solar System, namely 
	Saturn, which actually is not so small\footnote{Indeed,  
		Saturn's mass is about a third of the mass of Jupiter and 
		its semi-major axis is about 1.83 times the semi-major axis 
		of Jupiter. This implies that the strength of interaction with 
		Saturn is around 10\% of strength of interaction with Jupiter. 
		However, instabilities discussed in this paper are fairly robust 
		and we believe that they are not destroyed by the interaction 
		with Saturn (or other celestial bodies), which to some degree 
		averages out.}. With these assumptions one can model 
	the motion of the objects in the asteroid belt by 
	{\it the restricted 3-body problem.} 
	
	Denote by 
	$\mu = m_1/(m_0 + m_1)$ the mass ratio, where $m_0$ is 
	the mass of the Sun and $m_1$ is the mass of Jupiter. For 
	$\mu = 0$ 
	bounded orbits of the asteroid are ellipses. Up to orientation, 
	the ellipses are characterized by their semi-major axis $a$ and 
	eccentricity $e$. The famous theorem of Lagrange asserts that, 
	for small $\mu > 0$ and away from collisions, the semimajor axis $a(t)$ 
	of an asteroid satisfies $|a(t)-a(0)| \lesssim \mu$ for all $|t| < 1/\mu$. 
	For very small $\mu$ the time of stability was considerably improved by 
	Nekhoroshev and Niederman  \cite{Nekhoroshev77,Niederman96}.
	
	
	Application of the KAM theory shows that there is a positive 
	measure set of invariant (KAM) tori, each having quasiperiodic 
	orbits \cite{Arnold:1963, Fejoz04}. The measure of these tori is fairly large in certain open sets 
	of the asteroid belt, nevertheless, if one looks at the distribution 
	of asteroids in terms of their semi-major axis, one encounters 
	several gaps, the so-called Kirkwood gaps. The largest of them is the $3:1$ gap associated to the mean motion resonance, i.e. 
	when periods of Jupiter and a hypothetical asteroid have $3:1$ 
	ratio. This gap is followed by $5:2$, $7:3$, $2:1$ and etc. gaps 
	(see Figure \ref{fig:distribution}). It is believed that in the asteroid belt
	there is {\it a mixed behavior}, namely, there are two sets of 
	positive measure with qualitatively different behavior. As 
	the aforementioned KAM theory indicates, one set is the set
	of tori with quasiperiodic motions. We conjecture that there 
	is {\it  another non-small set with a stochastic diffusive behavior 
		of eccentricity.} 
	
	\begin{figure} [h!]
		\begin{center}
			\includegraphics[scale=0.625]{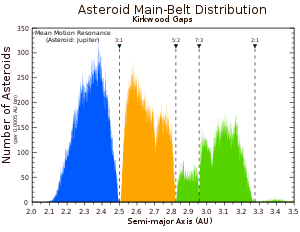}
			\caption{Distribution of asteroids in the asteroid belt}
			\label{fig:distribution}
		\end{center}
	\end{figure}

	The formation of these gaps can follow from a considerable variation 
	of the eccentricity. For example, by the Second Kepler law 
	for an asteroid being in $3:1$ Kirkwood gap means that its 
	semimajor axis $a$ is close to $3^{-2/3}$. If the orbit of this asteroid 
	has instant eccentricity $\ecc$, then the perihelion distance equals 
	$a(1-\ecc)$. Thus, for $\ecc=\ecc^*$ such that $a(1-\ecc^*)$ equals the semimajor 
	axis of Mars, an instant ellipse  of the asteroid will start crossing 
	the orbit of Mars. This would lead to a close encounter 
	which results either 
		in collision\footnote{See \url{https://en.wikipedia.org/wiki/List_of_craters_on_Mars}.},
		or capture\footnote{See the discussion of the origin of two moons of Mars in \url{https://en.wikipedia.org/wiki/Moons_of_Mars}. It is an open problem to prove 
			existence of solutions of the restricted planar
			four body problem Sun-Jupiter-Mars-asteroid with masses 
			$m(\text{Sun})\gg m(\text{Jupiter})\gg m(\text{Mars})>0=m(\text{asteroid})$
			such that they start in the Kirkwood 
			gap in the 3:1 resonant between Jupiter and asteroid and end 
			up being captured by Mars.},
		or ejection.   
	
	\subsection{The Wisdom mechanism}
	The first explanation of formation of Kirkwood gaps was
	proposed by Wisdom \cite{Wisdom82} and
	Neishtadt \cite{Neishtadt87}. Suppose that the ratio between the  mass ratio and
	the eccentricity of Jupiter satisfies
	\[
	\frac{\sqrt \mu}{e_0}\ \ll\  1.
	\]
	Then the dynamics in the mean motion resonances of the restricted
	planar 3-body problem can be described by a system with three time scales:
	fast, intermediate and slow. In reality, this ratio is approximately $\approx 0.632$ and
	not small.
	
	According to Wisdom in the main approximation, on a carefully chosen energy surface, 
	the secular evolution of the asteroid motion near 3:1 resonance with Jupiter is described 
	by a particular ``slow'' Hamiltonian system of two degrees of freedom. This system exhibits 
	a jump in eccentricity (see Figure \ref{fig:wisdom-3/1}).
	\begin{figure}[h]
		\begin{center}
			\includegraphics[width=5.25cm]{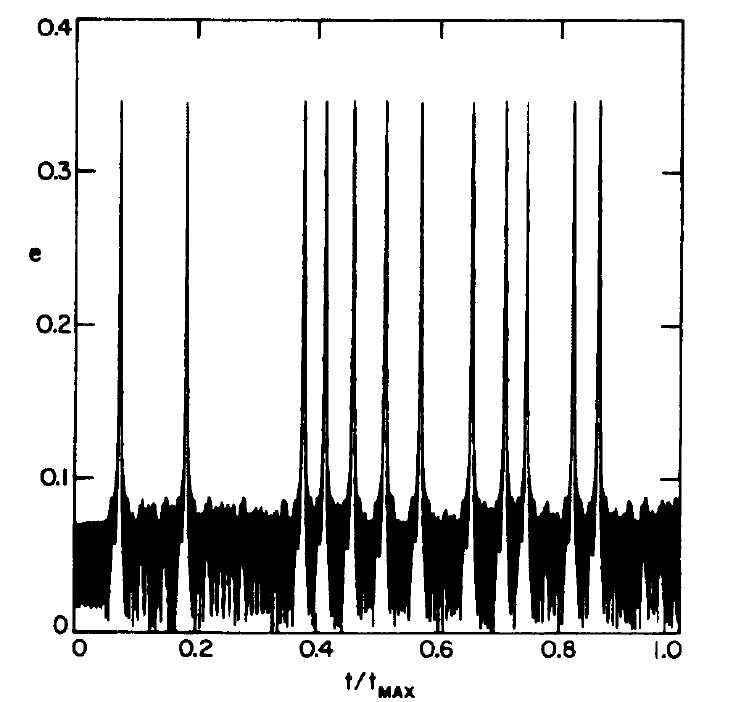}
		\end{center}
		\caption{Wisdom's plot of eccentricity evolution (see \cite{Wisdom82})}
		\label{fig:wisdom-3/1}
	\end{figure}
	This jump leads to a close passage next to Mars and is a potential explanation
	of the Kirkwood gap 3:1, under the assumption that $\frac{\sqrt \mu}{e_0}\ \ll\  1$.
	For more details about the Wisdom mechanism see Appendix \ref{wisdom}.

	The main goal of this paper is to describe possible dynamical instabilities
	in the $3:1$ Kirkwood gap in the regime opposite to Wisdom, that is
	\[
	\frac{e_0}{\sqrt \mu}\ll 1.
	\]
	In this regime, as we show, dynamics
	of the eccentricity can be approximated by a 1-dimensional stochastic
	diffusion process. This paper is a continuation of the investigations in
	\cite{FejozGKR15} where for small (unrealistic) eccentricity of Jupiter in the $3:1$
	resonance we exhibit a dynamical structure which leads to considerable
	fluctuations of eccentricity. We believe that a similar mechanism applies to
	the $5:2$ and $7:3$ Kirkwood gaps, but it requires verification of certain
	nondegeneracy conditions (Ansatz 1, 2 and 3 in 	\cite{FejozGKR15}). Now we turn to the mathematical
	model we study in this paper.
	
	\subsection{The restricted planar elliptic 3-body problem}
	Consider the restricted 3-body problem and assume that the massless 
	body moves in the same plane as the other bodies, 
	called primaries. We normalize the total mass to one, and we call 
	the three bodies the Sun (mass $1-\mu$),  Jupiter (mass $\mu$ 
	with $0 < \mu < 1$) and the asteroid (zero mass). If the energy of 
	the primaries is negative, their orbits describe two ellipses with 
	the same eccentricity, say $e_0 \in [0,1)$. For convenience, we denote 
	by $q_0(t)$ the normalized position of the primaries (or “fictitious body”)
	at time $t$, so that the Sun and Jupiter have respective positions 
	$-\mu q_0(t)$ and $(1-\mu)q_0(t)$. The Hamiltonian of 
	the asteroid is
	\begin{equation} \label{eq:RPE}
		\HH(p,q, t) =
		\frac{\|p\|^2}{2}- \frac{1-\mu}{\|q + \mu q_0(t)\|}-
		\frac{\mu}{\|q -(1- \mu) q_0(t)\|},
	\end{equation}
	where $q, p \in \R^2$ and $\|\cdot\|$ is the Euclidean distance. 
	Without loss of generality one can assume that $q_0(t)$ has 
	semi-major axis 1 and period $2\pi$. In what follows we 
	abbreviate the name of this problem as RPE3BP.
	
	For $e_0 \in[ 0,1)$, the RPE3BP has two and a half degrees of freedom. When $e_0 = 0$, 
	the primaries describe uniform circular motions around their center of mass. This system is often called 
	{\it the restricted planar circular 3-body problem} and, in what follows, is abbreviated RPC3BP. 
	
	
	Thus, in a frame rotating with the primaries, the system becomes autonomous and hence
	has only two degrees of freedom. Its energy in the rotating frame is a first integral, called 
	{\it the Jacobi integral or the Jacobi constant},  defined by
	\begin{equation} \label{eq:Jacobi}
		\tJ(q,p,t) =\frac{\|p\|^2}{2}- \frac{1-\mu}{\|q + \mu q_0(t)\|}-
		\frac{\mu}{\|q -(1- \mu) q_0(t)\|}-(q_1p_2-q_2p_1).
	\end{equation}
	The aforementioned KAM theory applies to both the circular 
	and the elliptic problem \cite{Arnold:1963, SiegelM95} and asserts that 
	if the mass of Jupiter is small enough, there is a set of initial 
	conditions of positive Lebesgue measure leading to quasiperiodic 
	motions, in the neighborhood of circular motions of the asteroid (see also \cite{Fejoz02b}).
	If Jupiter has a circular motion, since the system has only two 
	degrees of freedom, KAM invariant tori are $2$-dimensional 
	and separate the $3$-dimensional energy surfaces. But in 
	the elliptic problem, $3$-dimensional KAM tori do not prevent 
	orbits from wandering on a $5$-dimensional phase space. 
	In this paper we prove the existence of a wide enough set of 
	orbits in the RPE3BP whose eccentricity has {\it chaotic behavior 
		and can be approximated by a 1-dimensional stochastic 
		diffusion process. }

	Let us write the Hamiltonian \eqref{eq:RPE} as follows 
	\[
	\HH_\mu( p,q, t) = \HH_0( p,q) + \HH_1( p,q, t, \mu),
	\]
	where 
	\begin{equation} \label{eq:2BP}
		\begin{split}
			\HH_0( p,q)& =\frac{\|p\|^2}{2}-\frac{1}{\|q\|}\\
			\HH_1( p,q, t, \mu)& =\frac 1q-\frac{1-\mu}{\|q + \mu q_0(t)\|}-
			\frac{\mu}{\|q -(1- \mu) q_0(t)\|}.
		\end{split}
	\end{equation}

	The Keplerian part $\HH_0$ allows us to associate elliptical elements 
	to every point $( p,q)$ of the phase space of negative energy $\HH_0$. 
	We are interested in the fluctuations of the Jacobi constant $\tJ$ and 
	the osculating eccentricity $\ecc$ under the $t$-evolution. 
	
	To analyze this evolution,  in the following two sections, we introduce two main ingredients 
	that will be used in our main results, Theorems   \ref{thm:main-thm} and 
	\ref{thm:main-thm-bis} below. These are diffusion processes and the kind 
	of measures that we consider.
	
	\subsection{Diffusion processes}
	To state the main result about diffusive behavior, we need to recall some basic 
	probabilistic notions. A random process $\{B_t, t\ge 0\}$ is called {\it the Wiener process} or 
	{\it a Brownian motion} if the following  conditions hold: 
	
	\begin{itemize}
		\item $B_0=0$, $B_t$ is almost surely continuous and 
		$B_t-B_s\sim \mathcal N(0,t-s)$ for any $0\le s\le t$,
		where $\mathcal N(\bmmu,\bms^2)$ denotes the normal 
		distribution with expected value $\bmmu$ and variance $\bms^2$. 
		\item $B_t$ has independent increments. That is, if 
		$0 \leq s_1 \leq t_1 \leq s_2 \leq t_2$, then $B_{t_1}-B_{s_1}$ 
		and $B_{t_2}-B_{s_2}$ are independent random variables. 
	\end{itemize}
	A Brownian motion is a properly chosen limit of the standard 
	random walk. A generalization of a Brownian motion is 
	{\it a diffusion process} or {\it an It\^o (or It\^o-D\"oblin) 
	diffusion}. To define it, let $(\Omega,\Sigma,P)$ be a probability space. Let 
	$X_t:[0,+\infty) \times \Omega \to \R$. It is called an It\^o diffusion if
	it satisfies {\it a stochastic differential equation} of the form
	\begin{equation} \label{eq:diffusion}
		\mathrm{d} X_{t} = \bbb(X_{t}) \, \mathrm{d} t + 
		\bms (X_{t}) \, \mathrm{d} B_{t},
	\end{equation}
	where $B_t$ is a Brownian motion and $\bbb : \R \to \R$ 
	and $\bms : \R \to \R$ are the drift and the variance respectively. 
	For a point $x \in \R$, let $\mathbf{P}^x$ denote the law of $X_t$ 
	given initial data $X_0 = x$, and let $\mathbb{E}^x$ denote the 
	expectation with respect to $\mathbf{P}^x$.
	
	

	The {\it infinitesimal generator} of $X_t$ is the operator $A$, 
	which is defined to act  on suitable functions $f :\R\to \R$ by
	\[
	A f (x) = \lim_{t \downarrow 0} \dfrac{\E^{x} [f(X_{t})] - f(x)}{t}.
	\]
	The set of all functions $f$ for which this limit exists at 
	a point $x$ is denoted $D_A(x)$, while $D_A$ denotes 
	the set of $f$'s for which 
	the limit exists for all $x\in \R$. One can show that any 
	compactly-supported $\mathcal{C}^2$  function $f$ 
	lies in $D_A$ and that
	\begin{equation}\label{eq:diffusion-generator}
		Af(x)=\bbb(x) \dfrac{\partial f}{\partial x}+ \dfrac 12 \bms^2(x)
		\dfrac{\partial^2 f}{\partial x \partial x}.
	\end{equation}
	The distribution of a diffusion process is characterized  by 
	the drift $\bbb(x)$ and the variance $\bms(x)$.

	\subsection{Invariant measures supported on normally hyperbolic invariant laminations
	}\label{subsec:measures}
	The main results stated in the next section will show how the pushforward of a certain 
	class of measures along the flow associated to the Hamiltonian \eqref{eq:RPE} ``behave'' as 
	an It\^o process at certain time scales. Before giving precise statements, let us explain 
	how are the measures that we consider. 
	
The Hamiltonian \eqref{eq:RPE} fits what is usually called an 
	a priori chaotic setting in Arnold diffusion. A simplified model of such Hamiltonians is 
	a small perturbation of a perturbed pendulum like dynamics (i.e. a saddle with invariant 
	manifolds intersecting transversally) times an integrable twist map on a cylinder (see 
	Figure \ref{fig:NHIMCircular} below). 	For such models, one can build a normally hyperbolic (weakly) invariant lamination 
	(NHIL for short, see Definition \ref{def:NHIL} below). This NHIL is homeomorphic to 
	$\Sigma\times\TT\times \II$ where $\Sigma=\{0,1\}^{\ZZ}$ is the space of sequences 
	of two symbols and $\II\subset\RR$ is an interval.
	
	The measures whose pushforward we analyze are of two following  types:
	\begin{itemize}
		\item[1] We consider measures which are a Bernouilli measure on $\Sigma$ times an atomic measure at \emph{any given values} of angle $\theta\in\TT$ and $I\in\II$.
		\item[2] We consider measures of the form Bernouilli measure on $\Sigma$ times Lebesgue measure on $\TT$ times an atomic measure for any given value  $I\in\II$.
	\end{itemize}
	Roughly speaking Theorem \ref{thm:main-thm} below will show diffusive behavior for the pushforward of such measures projected onto the action coordinate $I$. Moreover, the diffusive behavior will follow the same It\^o process for either the type 1 measures (for any initial condition for the angle $\theta$) or the type 2 measures.
	
	\subsection{Stochastic diffusion along mean motion resonances:  Main Results}
	
	
	Consider the Hamiltonian $\HH_\mu$ in \eqref{eq:RPE}
	with $e_0\in [0,1)$. Denote by $\ecc=\ecc(p,q)$ the instant eccentricity of the asteroid 
	(see \eqref{def:ecc} for a precise formula) and by $\tJ=\tJ(p,q,t)$ its Jacobi constant (see \eqref{eq:Jacobi}).
	Let us consider an interval  $(\ecc_-,\ecc_+)\subset (0,1)$
	(respectively an interval $(\tJ_-, \tJ_+)$) 
	and denote by $X=(q,p,t)$  a point  in the whole phase space,  by $\Phi^tX$ the flow associated to the Hamiltonian $\HH_\mu$ with $X$
	as the initial condition, and by   $\Pi_\tJ(X)$ and   $\Pi_L(X)=L$ 
	the Jacobi constant at $X$ and the square root of the semimajor axis respectively.

	Let $\tJ^* \in [\tJ_-,\tJ_+]$ and $X^*=(p,q,t)$ be such that $\Pi_\tJ(X^*)=\tJ^*$
	Once 
	an orbit reaches the boundary of $[\tJ_-,\tJ_+]$ we stop the flow as follows
	\begin{equation}
		\label{eq:stop}
		{\small
			\widetilde \Phi^{t} X^*= 
			\begin{cases}
				\Phi^{t} X^* \qquad \qquad \ &  \text{if}\ 
				\Pi_\tJ(\Phi^{s} X^*) \in [\tJ_-,\tJ_+] \ \text{ for all } 
				\  0< s \le t. 
				\\
				\widetilde \Phi^{t} X^*=\Phi^{t^*} X^* \ & \text{if}\ \exists t^*(0,t)\,\,\text{such that }
				\Pi_\tJ(\Phi^{s} X^*) \in [\tJ_-,\tJ_+] \text{ for } 0< s < t^* \\ 
				&\text{and }  \ \Pi_\tJ(\Phi^{t^*} X^*) \in \partial[\tJ_-,\tJ_+]. 
		\end{cases}}
	\end{equation}
	%

	Let $\bbb(\tJ),\bms(\tJ):[\tJ_-,\tJ_+]\to \RR$ be smooth functions.
	Let $\tJ_s$ be the It\^o diffusion process with  drift $\bbb(\tJ)$ and 
	variance $\bms^2(\tJ)$, starting at $\tJ^*\in(\tJ_-, \tJ_+)$, i.e. 
	\begin{equation}\label{def:ItoIntegralForm}
	\tJ_s =\tJ^*+\int_0^s \bbb(\tJ_\tau)\, d\tau+\int_0^s \bms(\tJ_\tau)\, dB_\tau,
	\end{equation}
	where $B_s$ is the Brownian motion. 
	
	Given two functions $\bbb:[\tJ_-,\tJ_+]\to \mathbb R$ and $\bms:[\tJ_-,\tJ_+]\to \mathbb R_+$,
	we say that a family of probability measures $\nu_{\tJ^*,e_0},\ \tJ^*\in [\tJ-,\tJ_+]$, induces a stochastic 
	process with drift $\bbb(\tJ)$ and variance $\bms(\tJ)$ if, for any $s>0$, the distribution 
	of the pushforward of the Jacobi constant under the Hamiltonian flow $\Phi^{t}$ 
	in the time scale $t_{e_0}(s)=s\, e_0^{-2}$ (stopped if hits the boundary of $[\tJ_-,\tJ_+]$), 
	$\Pi_\tJ(\, \widetilde \Phi^{\,t_{e_0}(s)}_* \nu_{\tJ^*,e_0})$, converges weakly, as $e_0 \to 0$, 
	to the distribution of $\tJ_s$  given by \eqref{def:ItoIntegralForm}.

	\begin{theorem} \label{thm:main-thm} Fix $\mu=0.95387536\times10^{-3}$,  an interval $[\tJ_-,\tJ_+]$ and assume Ans\"atze \ref{ans:NHIMCircular:bis}
		and \ref{ansatz:Melnikov:1} hold. Then, for the Hamiltonian $\HH_\mu$ given by 
		\eqref{eq:RPE} there are smooth functions $\bbb(\tJ)$ and $\bms(\tJ),\ \tJ\in [\tJ_-,\tJ_+]$ such that: 
		for each $\tJ^*\in (\tJ_-,\tJ_+)$, there exist probability measures $\nu_{\tJ^*,e_0}$
		such that 
		\begin{itemize}
			\item they are supported inside the $3:1$ Kirkwood gap, i.e. 
			\[\Pi_L(\textup{supp }\nu_{\tJ^*,e_0}) \subset [3^{-1/3}-42\mu,3^{-1/3}+42\mu],\] 
			\item each initial condition in the support of $\nu_{\tJ^*,e_0}$ has 
			the Jacobi constant approximately equal to $\tJ^*$, i.e. 
			\[\Pi_\tJ (\textup{supp }\nu_{\tJ^*,e_0}) =\tJ^*+\OO(e_0).\]
		\end{itemize}
		Then,  for any $s>0$, the $J$-distribution of the pushforward under 
		 $\widetilde\Phi^{t}$ in the time scale 
		$t_{e_0}(s)=s\, e_0^{-2}$,
		i.e. $\Pi_\tJ (\, \widetilde \Phi^{\,t_{e_0}(s)}_* \nu_{\tJ,e_0})$ 
		converges weakly, as $e_0 \to 0$, to the distribution of the diffusion proces $\tJ_s$ by  given by \eqref{def:ItoIntegralForm} evaluated at the time $s$ (stopped  if hits the boundary of $[\tJ_-,\tJ_+]$).
	\end{theorem}

	Ans\"atze \ref{ans:NHIMCircular:bis} 
	and \ref{ansatz:Melnikov:1} are stated   in Sections \ref{sec:separatrixmap} and \ref{sec:RandomCylinder} respectively, since they need certain notions to be introduced. Roughly speaking they require the existence of certain hyperbolic invariant objects and that their associated invariant manifolds are in general position. They are explained heuristically in Section \ref{sec:planproof}.

	The measures provided by Theorem \ref{thm:main-thm} are of the two types introduced in 
	Section \ref{subsec:measures}. In particular, for the first type of measures 
	this implies that the limiting evolution of the Jacobi constant ``follows'' an It\^o process at certain time scales which  is independent of the initial condition on the angle. 
	

	We also have a slightly more involved version of Theorem \ref{thm:main-thm} for the evolution of 
	the eccentricity $\ecc$ of the asteroid. Roughly speaking, up to a deterministic  $\OO(\mu)$ correction and an arbitrarily small random correction, the eccentricity
	has also an It\^o diffusive  behavior. To this end, we define the function $\mathtt{E}$ which gives the equivalence between eccentricity and Jacobi constant at the $3:1$ mean motion resonance, that is
	 \begin{equation}\label{def:JtoE}
	 \mathbf{E}(\tJ)=\sqrt{1-\frac{\left(\frac{1}{L_0^2}+\tJ\right)^2}{L_0^2}},\qquad L_0=3^{-1/3}.
	 \end{equation}
	%

	\begin{theorem} \label{thm:main-thm-bis} 
 Fix $\mu=0.95387536\times10^{-3}$,  an interval $[\tJ_-,\tJ_+]$ and $\rho>0$ small enough. Assume Ans\"atze \ref{ans:NHIMCircular:bis} and \ref{ansatz:Melnikov:1} hold, fix $\tJ^*\in (\tJ_-,\tJ_+)$ and consider a  probability measure $\nu_{\tJ^*,e_0}$ given by Theorem \ref{thm:main-thm}.
		Then, there are  $C>1$ independent of $e_0$, an interval $(\ecc_-,\ecc_+)$ and  smooth functions
		$\wt \bbb(\ecc)$ and $\wt\bms(\ecc), \ecc \in (\ecc_-,\ecc_+)$, such that
		there exist 
		\begin{itemize}
			\item Correction functions $\delta \ecc(t,\ecc,i),\ \ecc \in (\ecc_-,\ecc_+),\ i=1,2,$ and $E(t,X)$, $X\in \textup{supp }\nu_{\tJ^*,e_0}$, satisfying $|E(t,X)|\leq \rho$,
			\item A sequence of times $\{t_j\}_{j \in \NN}$ and a sequence of indices 
			$\eta(t_j)\in \{1,2\}$ that depend only on position at time $t_j$ and such that, for each 
			$j \in \NN$,
			\[
			C^{-1}|\log \rho|< t_{j+1}-t_j <C |\log \rho|,\]
		\end{itemize}
	such that the  adjusted eccentricity defined piecewisely, for $t\in (t_j,t_{j+1})$, as 
			\[
			\tilde \ecc(X,t)= \ecc(\wt \Phi^t X)-\delta \ecc(t-t_j,\ecc(t_j),\eta(t_j))-E(\wt \Phi^tX)
			\] 
			satisfies the following.
		
		For any $s>0$, the distribution of the evolution of the adjusted eccentricity $\tilde \ecc$ at the  time scale $t_{e_0}(s)=s\, e_0^{-2}$
		(stopped if hits the boundary of $[\tJ_-,\tJ_+]$), i.e. $\Pi_{\tilde\ecc} (\, \widetilde \Phi^{\,t_{e_0}(s)}_* \nu_{\ecc,e_0})$
		converges weakly, as $e_0 \to 0$, to the distribution of $\, \wt \ecc_s$, where $\wt \ecc_\bullet$ is
		the diffusion process with the drift $\wt \bbb(\tilde \ecc)$ and the variance $\wt \bms(\tilde \ecc)$, starting at
		\[
		\ecc_{0}=\ecc^*=\mathbf{E}(\tJ^*)
		\]
		 evaluated at  time $s$.
	\end{theorem}
	
	In Section \ref{sec:maintheorem2}, we explain how to deduce this theorem from Theorem \ref{thm:main-thm}. In particular, we express the new drifts and variance in terms of the ones obtained in Theorem \ref{thm:main-thm}.

	\begin{remark}
	Relying on Ans\"atze 1 and 2, one can see that the correcting term $\delta \ecc(t-t_j,\ecc(t_j),\eta(t_j))$ satisfies
	\[
	|\delta \ecc(t-t_j,\ecc(t_j),\eta(t_j))|\leq 252 \mu,
	\]
(see also Section \ref{sec:maintheorem2}).
	\end{remark}

Theorems \ref{thm:main-thm} and \ref{thm:main-thm-bis} provide stochastic behavior in the energy and eccentricity drift at time scales $\sim e_0^{-2}$. Relying on the same analysis, one can construct orbits that drift at much faster pace.  This is explained in Appendix \ref{app:fastdrift} (see Theorem \ref{thm:fastdrift}).

		\begin{remark}\label{rmk:intervals}
		In the companion paper \cite{paperNHIL23} Ans\"atze \ref{ans:NHIMCircular:bis}
		and \ref{ansatz:Melnikov:1} are verified numerically in several intervals of Jacobi constant. In particular, we show that they are satisfied in the intervals
		\[
		\mathcal{I}_1=[-1.551, -1.485], \quad \mathcal{I}_2=[-1.535, -1.485].
		\]
		That is, the Jacobi constant has an ``stochastic It\^o  evolution'' as stated in Theorem \ref{thm:main-thm} in these intervals.
		In terms of the instant eccentricity of the asteroid $\ecc$ (see Theorem \ref{thm:main-thm-bis}), the interval $\mathcal{I}_1$ corresponds to $\EE_1=[0.676,0.77]$ and $\mathcal{I}_2$ corresponds to $\EE_2=[0.7,0.77]$. We expect that these numerical computations can be rigorously verified through a computer assisted proof.
	\end{remark}

	\begin{remark}[Comparison with \cite{CapinskiGidea23}]\label{rmk:CapinskiGide} The recent paper \cite{CapinskiGidea23} by 
		Capinski and Gidea develop tools to  analyze stochastic behavior 
		associated to Arnold diffusion in an a priori chaotic setting. Then, they 
		apply their tools to diffuse along the center manifold associated to the 
		Lagrange point $L_1$ in the  RPE3BP, modelling the  Neptune-Triton-Asteroid dynamics. Let us compare both papers. 
		\smallskip
		

		\begin{itemize}
			\item In the present paper we construct a NHIL and study Bernoulli probability measures supported on it. 
			As a result we obtain stochastic diffusion processes with variable drift and variance. Both drift and variance 
			can be dynamically computed  through Melnikov's type integrals. This is in contrast with   \cite{CapinskiGidea23}.
			The authors fix constant variance $\bms^2$ and constant drift $\bmmu$ and construct a probability measure 
			whose push forward converges to a stochastic process with such a drift and variance. 
				In both papers the measures do not have maximal Hausdorff dimension.
			\item We also would like to point out that our ``size'' of diffusion interval is over 0.06 (with respect to the Jacobi constant, which corresponds to 0.09 in eccentricity), while in \cite{CapinskiGidea23},
			size of diffusion interval of energy is $10^{-6}$. 
			\item In \cite{CapinskiGidea23} there is not second order analysis in the  perturbation parameter, as 
			usually happens when looking at diffusive behavior. Instead, they perform a backwards selection of initial conditions. In the present paper the second order in $e_0$ 
			plays a fundamental role in defining the drift of the It\^o process.
		\end{itemize}
	\end{remark}
	A  diffusive behavior phenomenon at  instabilities for nearly integrable systems was 
	first observed by Chirikov (see \cite{Chirikov:1979}).  A first example of unstable motion had been constructed 
	earlier by Arnold \cite{Arn64}, which inspired to Chirikov to coin  the name Arnold diffusion.

	\begin{remark}
		We believe that the stochastic diffusive mechanism that leads to Theorem \ref{thm:main-thm} is fairly general 
		and that it widely applies to \emph{a priori} chaotic models. Here are some of them,
		\begin{itemize}
			\item Mather's problem: Consider a geodesic flow plus a periodic time dependent potential. In the higher 
			energy regime, this model is a priori chaotic and one can obtain unbounded growth of energy 
			(see \cite{Mather96, DelshamsLS00, BolotinT99}). The time dependence can be generalized to 
			be quasiperiodic (see \cite{DelshamsLS06b}).
			\item Drift along secular resonances in the spatial 4 Body Problem: In the recent papers \cite{ClarkeFG22,ClarkeFG23}, 
			the authors show drift in eccentricity, semimajor axis and inclination for some of  bodies in 
			a 4 Body Problem. Such drift is achieved along secular resonances both in the planetary regime 
			(one body much larger than the other) and in the hierarchical regime (bodies increasingly separated). 
			In the hierarchical regime, this model is a priori chaotic.
			\item Arnold diffusion near $L_1$ and $L_2$: The Lagrange points $L_1$ and $L_2$ are saddle center 
			critical points and Arnold diffusion along their center manifolds for the RPE3BP fits the a priori chaotic 
			setting for small enough eccentricity of the primaries. Note that, as mentioned in Remark \ref{rmk:CapinskiGide}, Arnold diffusion 
			along the center manifold of $L_1$ has already been analyzed in \cite{CapinskiGidea23} (see also \cite{DelshamsGR26}), and along the center manifold of $L_2$  in \cite{Capinski17}.
			\item Another setting in the RPE3BP where one can achieve drift of angular momentum is relying on 
			an invariant cylinder at infinity and its associated invariant manifolds (the so-called parabolic motions). 
			Such setting was analyzed by 
			Delshams, de la Rosa, Kaloshin and Seara in \cite{DelshamsKRS19} 
			(see also \cite{GuardiaPS23}).
			
			\item In \cite{Gelfreich:2008}, Gelfreich and Turaev   consider Hamiltonian systems with slow 
			time dependence. The models they consider, which include not only the aformentioned Mather problem, 
			but also billiards with time dependent boundaries also fit an a priori chaotic setting. They show 
			the existence of unbounded growth of energy orbits.
			
			
			\item As far as the authors know the first results showing stochastic behavior related to Arnold 
			diffusion are by Marco and Sauzin \cite{MarcoSauzin04}. They showed existence of random 
			walks in nearly integrable Hamiltonian systems obtained by the so-called Herman method. 
			These Hamiltonians also fit the a priori chaotic setting.
			
		\end{itemize}
	\end{remark}
	
	\medskip 
	


	\subsection{Plan of the proof}\label{sec:planproof}

	In the restricted planar elliptic 3-body problem the primaries 
	have eccentricity $e_0\in (0,1)$. In the case $e_0=0$, one obtains 
	the circular problem. Denote by $q_0(t,e_0)$ the normalized 
	position of the primaries at time $t$ with eccentricity $e_0$. Then, 
	the Hamiltonian of the elliptic problem, defined in \eqref{eq:RPE},
	can be written 
	\[
	\HH_\mu( p,q, t) =
	\frac{\|p\|^2}{2}- \frac{1-\mu}{\|q + \mu q_0(t,e_0)\|}-
	\frac{\mu}{\|q -(1- \mu) q_0(t,e_0)\|}
	\]
	as a $e_0$-perturbation of the corresponding 
	Hamiltonian of the circular problem 
	\[
	\HH_{\text{circ},\mu}( p,q, t) =
	\frac{\|p\|^2}{2}- \frac{1-\mu}{\|q + \mu q_0(t,0)\|}-
	\frac{\mu}{\|q -(1- \mu) q_0(t,0)\|}.
	\]
	where $q_0(t,0)$ is just given by $q_0(t,0)=(\cos t, \sin t)$.

		This perturbative structure is very important in our analysis. 
		First, we shall construct  "horseshoes" (a hyperbolic invariant set)
		$\mathcal C_\tJ$ at level sets $\tJ$ of the Jacobi constant for the circular 
		problem and which depend smoothly on $\tJ\in (\tJ_-,\tJ_+)$. These invariant 
		sets gives rise to a normally hyperbolic invariant lamination 
		$\Lambda$  for the elliptic problem for $e_0>0$ small enough. 
		Over each point of $\mathcal C_\tJ$ we have a fiber given 
		by a 2-dimensional cylinder. 
	
	Let us be more precise on the steps of the proof.
	We start with the Hamiltonian of the 2-body problem $\HH_0$ and 
	the associated Jacobi constant $\tJ$, defined in \eqref{eq:Jacobi}. 
	Later we make two perturbations: 
	\begin{itemize}
		\item From the 2-body Hamiltonian $\HH_0$ to the circular Hamiltonian 
		$\HH_{\ccirc,\mu}$. 
		\item From the circular Hamiltonian $\HH_{\ccirc,\mu}$ to the elliptic Hamiltonian $\HH_\mu$. 
	\end{itemize}

	The more detailed scheme of the proof is the following 
	\begin{enumerate}
		\item 
		For the 2-body problem the Kirkwood gap $3:1$ is defined by 
		the semimajor axis $a=3^{-2/3}$, where $2a=\HH_0^{-1}$. For each 
		eccentricity $\ecc \in [0,1)$ there are one parameter families of 
		periodic orbits, each represented by an ellipse with semimajor 
		axis $a$ and eccentricity $\ecc$. It turns out that at the Kirkwood 
		gap $\ecc$ is an implicit function of $\tJ$.

		\item Consider the Hamiltonian of the circular problem $\HH_{\ccirc,\mu}$,
		which is an $\OO(\mu)$-pertur\-bation of $\HH_0$ (recall that we are taking $\mu=0.95387536\times10^{-3}$). Ansatz \ref{ans:NHIMCircular:bis} assumes the following.
		\begin{enumerate}
			\item There exists a family of saddle periodic orbits $p_\tJ$ of the RPC3BP at 
			 the $42\mu$-neighbor\-hood of the Kirkwood gap $3:1$ parametrized by the Jacobi constant $\tJ\in [\tJ_-,\tJ_+]$. 
			\item For values $\tJ\in [\tJ_-,\tJ_+]$,
			the associated invariant manifolds
			$W^s(p_\tJ)$ and $W^u(p_\tJ)$ intersect transversally (within the Jacobi constant level) at two distinct homoclinic orbits. Each transverse intersection 
			gives rise to a homoclinic channel. 
		\end{enumerate}
		Note that such behavior is typical in Hamiltonian systems. Thus Ansatz \ref{ans:NHIMCircular:bis} just assumes that this indeed happens for the RPC3BP at the $3:1$ mean motion resonance. This is verified numerically in the companion paper \cite{paperNHIL23}.
		
		\item For each homoclinic channel, we determine a collection of open 
		sets (see Figure \ref{fig:separatrix-map})
		where a separatrix (return) map $\SM$ is well defined and 
		consists of $n$ iterates.
		Then we compute $\SM$ (see Theorems \ref{thm:FormulasSMii:circ} and \ref{thm:FormulasSMii} below). 
		
		\item We prove that the separatrix map $\SM$ restricted to suitable open sets 
		is  partially hyperbolic and, using 
		an isolating block technique, we contruct a normally hyperbolic 
		lamination $\Lb$ (see Theorems \ref{thm:NHILCircular} and \ref{thm:NHILElliptic} below).

		\begin{figure} [h!]
			\begin{center}
				\includegraphics[scale=1.0]{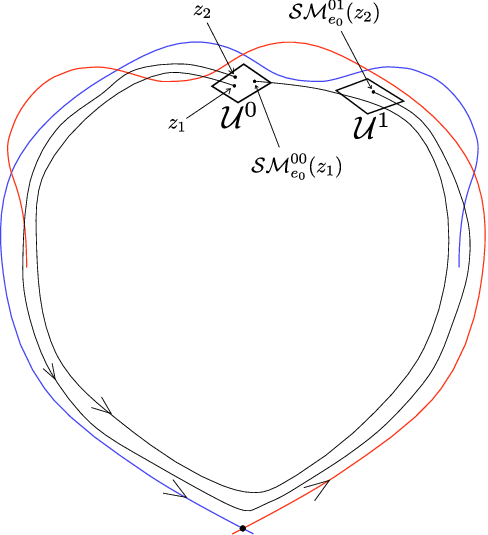}
				\caption{Separatrix map, which is defined on open sets in $\UU^0$ and $\UU^1$, close to the two homoclinic channels.}
				\label{fig:separatrix-map}
			\end{center}
		\end{figure}
		
		\item We derive a normal form (Proposition \ref{prop:Cylindermaps:normalform}) for the dynamics of $\SM$ restricted 
		to the lamination $\Lb$ such that it is conjugated to 
		a skew-shift  model 
		\[
		f:(\omega,I,\theta)\to 
		(\sigma \omega,I+e A_\omega(I)
		\cos (\theta+\psi_\omega)
		+\OO(e^2_0),\theta+\Omega_\omega(I)+\OO(e_0)),
		\]
		where $\theta\in \mathbb T,\  I\in [a,b],\  \Sigma=\{0,1\}^\mathbb Z$ is the space of sequences of $0,1$’s, $\sigma:\Sigma \to \Sigma$ is the shift in this space, 
		$\Omega_\omega$ is the shear,\ $A_\omega$ is the amplitude, and 
		$e_0$ is the eccentricity of Jupiter, which is taken as a small parameter,
		(see Lemma \ref{lemma:laminationmap} below).
		
		Ansatz \ref{ansatz:Melnikov:1} assumes that some of these functions are non-trivial in certain senses. In particular, on the one hand, we assume  that the shears $\Omega_\omega$ depend non-trivially on $\omega_0$ and have a non-resonance property. On the other hand, we assume that $A_\omega$ and $\psi_\omega$ (which are obtained through Melnikov-like integrals) are non-degenerate.
		These properties are also  verified numerically in the companion paper \cite{paperNHIL23}.
		
		\item To prove stochastic diffusive behavior for the above skew
		shift model, we analyze first a ``reduced model'': a Lie group (circle) extension of hyperbolic maps, for which one can prove decay of correlations and central limit theorems. For the measures of type 2 (see Section \ref{subsec:measures}) such properties are proven in \cite{FieldMT03} (see Theorem \ref{thm:Dima}). For type 1 measures, we prove both decay of correlations and a central limit theorem (see Theorem \ref{thm:Dima:FixedAngle}).
		
		\item Last step is to prove the convergence to the It\^o process through a martingale analysis  (see Lemma \ref{lemma:expectationlemsubstrip}). It is done in two steps. First we analyze it locally (on small Jacobi constant intervals) relying on the tools developed on the analysis of circle extensions of hyperbolic maps. Finally, we go from local to global by analyzing the visits of the orbits to the Jacobi constant intervals as a random walk (see Section \ref{sec:LocalToGlobal}). This completes the proof of Theorem \ref{thm:main-thm}.
	\end{enumerate}

\section*{Acknowledgements}
\addcontentsline{toc}{section}{\protect\numberline{}Acknowledgements}
We warmly thank Dmitry Dolgopyat and Ian Melbourne for many useful discussions concerning decay of correlations and central limit theorems for equivariant observables of compact Lie group extensions of hyperbolic maps, which are significant steps in the proof.

This work was partially supported by the grant PID-2021-122954NB-100 and PID2021-123968NB-100  funded by MCIN/AEI/10.13039/
501100011033 and “ERDF A way of making Europe. M. G. is
supported by the Catalan Institution for Research and Advanced Studies via an ICREA Academia Prize
2018 and 2023. V.K. is also partially supported by the ERC Advanced Grant SPERIG (\# 885707).
P. M. is also partially supported by PID2024-158570NB-I00.
This work is also supported by the Spanish State Research Agency through the Severo Ochoa and
Mar\'ia de Maeztu Program for Centers and Units of Excellence in R\&D(CEX2020-001084-M).

\section{A good system of coordinates and a time reparameterization}\label{sec:delaunay}

We start by writing the RPE3BP Hamiltonian $\HH_\mu$, defined in
\eqref{eq:RPE},  in the classical Delaunay coordinates $(L,\ell,G, \g,t)$. They are the action-angle coordinates for the Kepler problem and are defined as follows. Introducing $r$ and $\varphi$ by 
\[
 q=(r\cos \varphi,r\sin\varphi),
\]
we define
\begin{itemize}
 \item $L$ is the square of the semimajor axis of the ellipse. It satisfies
 \[
  -\frac{1}{2L^2}=\frac{\|p\|^2}{2}-\frac{1}{\|q\|}.
 \]
 \item $G$ is the angular momentum. That is $G=q\times p$. By $L$ and $G$, one can define the eccentricity of the ellipse as
 \begin{equation}\label{def:ecc}
  \ecc=\sqrt{1-\frac{G^2}{L^2}}.
 \end{equation}
 \item $\g$ is the argument of the pericenter.
 \item $\ell$ is the mean anomaly, which can be obtained using the true anomaly $v$ and the eccentric anomaly $u$. The true anomaly is the angle of the body with respect to the perihelion measured from the focus of the ellipse. That is, $\varphi=v+\g$. From the true anomaly, one can compute the  eccentric anomaly by
 \[
  \tan\frac{v}{2}=\sqrt{\frac{1+\ecc}{1-\ecc}}\tan\frac{u}{2}
 \]
and from it the mean anomaly by
\[
 \ell=u-\ecc\sin u.
\]
\end{itemize}

Then, in these coordinates, the Hamiltonian \eqref{eq:RPE} becomes
\begin{equation}\label{def:HamDelaunayNonRot}
\hat H(L,\ell,G, \g-t,t)=
-\frac{1}{2L^2}+\mu\Delta H_\ccirc(L,\ell,G,
  \g-t,\mu)
  + \mu e_0\Delta H_\eell(L,\ell,G, \g-t,t,\mu,e_0),
\end{equation}
where $e_0$ is the eccentricity of the primaries. 

Define the new angle $g=\g -t$ (the argument of the pericenter, measured in
the rotating frame) and a new variable $I$ conjugate to the time $t$.
Then, we have
\begin{equation}\label{def:HamDelaunayRot}
  H(L,\ell,G, g,I,t)=
  -\frac{1}{2L^2}-G+\mu\Delta H_\ccirc(L,\ell,G, g,\mu)
  +
  \mu e_0 \Delta H_\eell(L,\ell,G, g,t,\mu,e_0)+I.
\end{equation}
Without loss of generality we can restrict our study to $H=0$.

We consider the RPE3BP as a perturbation of the circular one, i.e.  $e_0=0$,
\begin{equation}\label{def:HamDelaunayCirc}
  H_\ccirc(L,\ell,G, g)=-\frac{1}{2L^2}-G+\mu \Delta H_\ccirc(L,\ell,G, g,\mu).
\end{equation}
Following \cite{FejozGKR15}, we  modify
our system so that $g$ is the new time. That is , we reparameterize
time and consider the equation
\begin{equation}\label{def:Reduced:ODE}
  \begin{array}{rlcrl}
    \frac{d}{ds} \ell=&\dps\frac{\pa_L H}{-1+\mu\pa_G\Delta H_\ccirc +\mu
e_0\pa_G\Delta H_\eell}&\text{   }&\frac{d}{ds} L=&\dps-\frac{\pa_\ell
H}{-1+\mu\pa_G\Delta H_\ccirc +\mu e_0\pa_G\Delta H_\eell}\\
    \frac{d}{ds} g=&1&\text{   }&\frac{d}{ds} G=&\dps-\frac{\pa_g
H}{-1+\mu\pa_G\Delta H_\ccirc +\mu e_0\pa_G\Delta H_\eell}\\
    \frac{d}{ds} t=&\dps\frac{1}{-1+\mu\pa_G\Delta H_\ccirc +\mu e_0\pa_G\Delta
H_\eell}&\text{   }&\frac{d}{ds} I=&\dps-\frac{\mu e_0\pa_t \Delta
H_\eell}{-1+\mu\pa_G\Delta H_\ccirc +\mu e_0\pa_G\Delta H_\eell}.
  \end{array}
\end{equation}
This system can be written as a non-autonomous Hamiltonian system
using the Poincar\'e-Cartan reduction. It is enough to consider
a  time-periodic Hamiltonian $K$ such that\footnote{Note that we change the order of the variables so that the new time $g$ is at the end.}
\[
   H(L,\ell,-K(L,\ell,I,t, g), g,I,t)=0.
\]
Note that it can be rewritten as
\begin{equation}\label{def:HamEll:Reparam}
 K(L,\ell,I,t,g)=-\frac{1}{2L^2}-I+\mu\Delta K_\mathrm{circ}(L,\ell, I,g)+\mu
e_0 \Delta K_\eell(L,\ell,I,t,g;e_0).
\end{equation}
Notice that $\Delta K_\mathrm{circ}$ and $\Delta K_\eell$
 depend on $\mu$. We omit
 this dependence, when inessential. 
It can be easily checked that the Hamiltonian equations coincide with
\eqref{def:Reduced:ODE}. We also define
\begin{equation}\label{def:HamCirc:Reparam}
 K_\mathrm{circ}(L,\ell, I,g)=-\frac{1}{2L^2}-I+\mu\Delta
K_\mathrm{circ}(L,\ell, I,g).
\end{equation}

\section{The normally hyperbolic invariant cylinder and the separatrix maps}\label{sec:separatrixmap}
In this section we discuss the first of the key ingredients of
our proof: the separatrix map. That is, a return map close to homoclinic channels. It was first introduced
and studied by Zaslavskii (see \cite{Treschev02a}) for  one and a half degrees of freedom
Hamiltonian systems with a separatrix loop and by Shilnikov for vector fields with a transverse separatrix crossing to an equilibrium (see the survey of Piftankin and Treschev
\cite{PiftankinT07}). Later, in  \cite{Treschev02a, GuardiaK15} it was constructed for a priori unstable nearly
integrable Hamiltonian systems. 

Relying on 
Ansatz \ref{ans:NHIMCircular:bis}, we compute the separatrix map associated to a normally hyperbolic invariany cylinder of 
the Hamiltonian \eqref{def:HamEll:Reparam}, 
which differs from the aforementioned models. 
Indeed, the RPE3BP with $\mu=10^{-3}$ and $0\le \ee\ll1$
is what is called an {\it a priori chaotic system:} for $\ee=0$, the system has
a first integral (the Jacobi constant) but it is not
integrable\footnote{It posesses a hyperbolic periodic orbit
with associated transverse homoclinic points at each level
of the first integral.}. 
Note that in a priori chaotic settings, the angle between the invariant manifolds of the normally hyperbolic invariant manifold is not small. This implies that one cannot study the
separatrix map in a whole fundamental domain as done in \cite{Treschev02a,
GuardiaK15} but  in small
neighborhoods of  transverse homoclinic channels to the cylinder.

We follow more closely \cite{Piftankin06}, which studies
the Mather problem which is also a priori chaotic. Even though there are similarities
with \cite{Piftankin06}, we need to adapt to the particular feature of the REP3BP.


We express the separatrix map in suitable variables, so that
its pullback to the invariant lamination (see Section \ref{sec:NHIL}) is integrable for the circular
problem.
Moreover, we use a suitable Poincar\'e map to simplify the numerical
computations needed to verify the ans\"atze of Theorem~\ref{thm:main-thm} (see \cite{paperNHIL23}).

%
%

\subsection{The invariant cylinder and the associated homoclinic channels}
Relying on numerical computations done in \cite{paperNHIL23}, 
we assume the following ansatz.

%

\begin{ansatz}\label{ans:NHIMCircular:bis}
  Consider the Hamiltonian \eqref{def:HamDelaunayCirc} with
  $\mu=0.95387536\times10^{-3}$.  In every energy level $\tJ\in [\tJ_-, \tJ_+]$:
  \begin{itemize}
  \item   There exists a hyperbolic periodic orbit
  $\lb_\tJ=(L_\tJ(t), \ell_\tJ(t), G_\tJ(t), g_\tJ(t))$ of period $T_\tJ$ with
  \begin{equation}\label{def:periodestimate:intro}
 9\mu< \left| T_\tJ-2\pi\right|<15 \mu,
  \end{equation}
  such that
  \[
  \left| L_\tJ(t)-3^{-1/3}\right|< 19\mu
  \]
  for all $t\in [0,T_\tJ]$.  The periodic orbit and its period
  depend analytically on $\tJ$.

\item The stable and unstable invariant manifolds of every $\lb_\tJ$, $W^{s}(\lb_\tJ)$
  and $W^{u}(\lb_\tJ)$, intersect transversally at \emph{two primary homoclinic points}. Moreover, these homoclinic points depend analytically on $J$ and their orbits are confined in the interval 
  \[
  \left| L-3^{-1/3}\right|< 42\mu.
  \]
\end{itemize}
\end{ansatz}
Note that the the exact lower and upper bounds in \eqref{def:periodestimate:intro} are not important. What is crucial is that 
\[
0< \left| T_\tJ-2\pi\right|<\frac{\pi}{2},
\]
to avoid certain resonances.
\begin{figure} [h!]
\begin{center}
\includegraphics[scale=0.59]{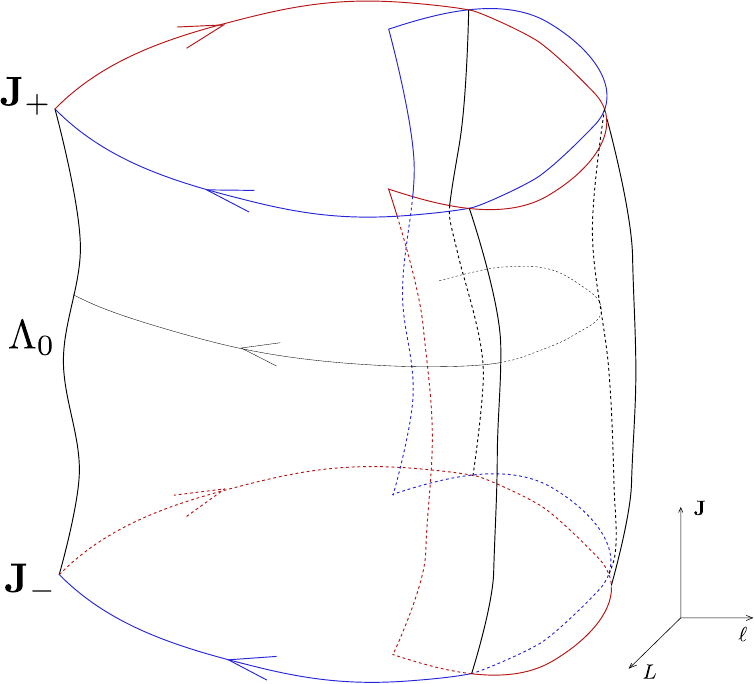}
\caption{The periodic orbits provided by Ansatz \ref{ans:NHIMCircular:bis} give rise to a normally hyperbolic invariant cylinder, as stated in Corollary \ref{coro:NHIMCircular}. It posseses two transverse homoclinic channels.}\label{fig:NHIMCircular}
\end{center}
\end{figure}
In \cite{paperNHIL23} this ansatz is verified for the interval  $[\tJ_-, \tJ_+]=[\Jminusnew, \Jplusnew]$. Note that in \cite{FejozGKR15} the authors considered a similar ansatz. Indeed, Item 1 is the
same. However, in Item 2, they only assumed the existence of  one
transverse homoclinic orbit for each $J$ instead of two. The analyticity with
respect to the period is just a consequence of the analyticity of the Hamiltonian 
\eqref{def:HamDelaunayCirc}.
Later we impose an additional ansatz (see Ansatz \ref{ansatz:Melnikov:1}),
a non-degeneracy condition for dynamics along
the 
homoclinic orbits.
\smallskip

\begin{remark}
 In \cite{FejozGKR15} we showed numerically the existence of homoclinic
tangencies. To overcome that problem, we considered two families of homoclinic
points.  In the present paper, we
only consider intervals of Jacobi constants 
 between two of the
homoclinic tangencies detected in \cite{FejozGKR15}. In this way, in these intervals there are two transverse homoclinic channels.
\end{remark}

After a time reparameterization and recalling that the right hand side of the
following system is $t$ independent, Ansatz \ref{ans:NHIMCircular:bis} implies
that the system defined by the $(\ell, L, g, G)$ components of
\begin{equation}\label{def:Reduced:ODE:circular}
  \begin{array}{rlcrl}
    \frac{d}{ds} \ell=&\dps\frac{\pa_L H_\ccirc}{-1+\mu\pa_G\Delta H_\ccirc}&\text{
}&\frac{d}{ds} L=&\dps-\frac{\pa_\ell
H_\ccirc}{-1+\mu\pa_G\Delta H_\ccirc }\\
    \frac{d}{ds} g=&1&\text{   }&\frac{d}{ds} G=&\dps-\frac{\pa_g
H_\ccirc}{-1+\mu\pa_G\Delta H_\ccirc}\\
    \frac{d}{ds} t=&\dps\frac{1}{-1+\mu\pa_G\Delta H_\ccirc}&\text{
}&\frac{d}{ds} I=&\dps0
  \end{array}
\end{equation}
where $H_\ccirc$ is the Hamiltonian introduced in \eqref{def:HamDelaunayCirc}, has a periodic orbit for each
\[
I_0\in [I_-,I_+]=[-\tJ_+,-\tJ_-].
\]
Since the
equation is $t$-independent,  such periodic orbits can be parameterized as
\begin{equation}\label{eq:PeriodicOrbitsReparam}
 \begin{split}
 \ell&=\ell_0(I,g)\\
 L&=L_0(I,g),
 \end{split}
\end{equation}
for analytic functions $(\ell_0, L_0)$.

\begin{corollary}\label{coro:NHIMCircular}
  Assume Ansatz~\ref{ans:NHIMCircular:bis} holds. The system
\eqref{def:Reduced:ODE:circular} with $\mu=0.95387536\times10^{-3}$  has an
analytic normally hyperbolic\footnote{See Definition \ref{def:NHIM} below
for the precise definition of Normally Hyperbolic Invariant Manifold.} 
invariant $3$-dimensional cylinder $\Lambda_0$, which is foliated 
by $2$-dimensional invariant tori.

The cylinder $\Lambda_0$ has 
stable and
unstable invariant manifolds, denoted $W^{s}(\Lambda_0)$ and
$W^{u}(\Lambda_0)$. In the  
invariant  planes $I=I_0$,  for every $I_0\in [I_-,I_+]$, 
these invariant manifolds 
intersect transversally at two homoclinic channels $\CCC^1_{0}$
and $\CCC^2_{0}$.
\end{corollary}

Recall that the equation \eqref{def:Reduced:ODE:circular} coincides with the Hamiltonian equations of motion  associated to $K_\mathrm{circ}$ in \eqref{def:HamCirc:Reparam}. In the next two lemmas, we apply time dependent (i.e. $g$ dependent) symplectic changes of 
coordinates to Hamiltonian $K_\mathrm{circ}$. 

\begin{lemma}\label{lemma:ChangePO}
There exists an analytic $g$--time dependent symplectic change of coordinates
\[(\wh L,\wh\ell, I,\wh t, g)=\Psi(L,\ell,t,I,g),\] periodic in $g$  and of the form
\[
 \begin{split}
 \wh L&=L-L_0(I,g)\\
 \wh \ell&=\ell-\ell_0(I,g)
 \end{split}\qquad \begin{split}
 I&=I\\
 \wh t&=t+T(L,\ell,I,g)
  \end{split}
\]
such that the periodic orbit
\eqref{eq:PeriodicOrbitsReparam} is translated to
$(\wh L,\wh\ell, I)=(0,0,I)$ and the function $T$ satisfies
 \begin{equation}\label{def:0onPO}
  T(L_0(I,g),\ell_0(I,g),I,g)=0.
 \end{equation}
\end{lemma}
\begin{proof}
The proof of this lemma is straightforward. Indeed, it is
enough to define a change of coordinates as a time-one map
of the Hamiltonian
\[
 \Gamma(L,\ell,I,g)=L_0(I,g)\ell-\ell_0(I,g)L+Q(I,g),
\]
for a suitable $Q$ so that \eqref{def:0onPO} is satisfied.
\end{proof}

Note that we are not requiring that this change of
coordinates keeps the $2\pi$-periodicity in $\wh\ell$.
This is not important, since the study in the
$(\wh L,\wh\ell)$--plane is local close to the saddle and
the associated local stable and local unstable
invariant manifolds.

Now that we have placed the periodic orbit at the origin,
we can apply the Moser normal form \cite{Moser56}. Note that $\wh I$ is
a first integral and, therefore, it can be treated as a parameter.
That is, we apply the Moser normal form for each value
$I_0\in [I_-,I_+]$.

\begin{lemma}\label{lemma:MoserNF}
Consider the Hamiltonian $K_\mathrm{circ}$  in
\eqref{def:HamCirc:Reparam} and the change of
coordinates $\Psi$ introduced in Lemma \ref{lemma:ChangePO}.
Then, there exists an  analytic  $g$--time dependent  symplectic change of 
coordinates
$(I,s,p,q, g)=\Upsilon(\wh L,\wh\ell,\wh I,\wh t, g)$ of
the form 
\[
 \begin{split}
 I&=I\\
 s&=\wh t+B(\wh L,\wh \ell,\wh  I,g)\\
 p&=P(\wh L,\wh \ell,\wh  I,g)\\
q&=Q(\wh L,\wh \ell,\wh  I,g),
  \end{split}
\]
such that
\[
 B(0,0,\wh  I,g)=0
\]
and
the
Hamiltonian $K_\mathrm{circ}\circ \Psi^{-1}\circ\Upsilon^{-1}$ is of the form
\begin{equation}\label{def:HamCircFinal}
\KK_0(I,p,q)= K_\mathrm{circ}\circ \Psi^{-1}\circ\Upsilon^{-1}(I,p,q)=E(I)+F(I,pq),
\end{equation}
where the frequency
\begin{equation}\label{def:nu}
\nu(I)=\pa_IE(I)
\end{equation}
 satisfies
\begin{equation}\label{def:Twist}
 \left|\nu(I)-1\right|<\frac{15}{2\pi}\cdot \mu
\end{equation}
and $F$ satisfies $F(I,pq)=\la(I)pq+\OO_2(pq)$, where $\la(I)$ is the 
strictly positive Floquet exponent of the  periodic orbit
$(p,q)=(0,0)$ given by Ansatz \ref{ans:NHIMCircular:bis}.
\end{lemma}
\begin{proof}
This proof is a direct consequence of the Moser normal form. Indeed,
$K_\mathrm{circ}\circ \Psi^{-1}$ is $\wh t$ independent. Therefore, for fixed $\wh
I$, the Moser
normal form gives a change of coordinates $(p,q)=\Upsilon^I(\wh L,\wh\ell)$ which
transforms $K_\mathrm{circ}\circ \Psi^{-1}$ to $\KK_0$ and is
symplectic with respect to $(\wh L,\wh\ell)$ (it preserves
the symplectic form $d\wh L\wedge d\wh \ell$ treating $I$ as a constant).
To have a full symplectic change of coordinates
for varying $(I,\wh t)$, it is enough to choose a suitable $B$ to
define $s$.

Estimate \eqref{def:Twist} is a direct consequence of Ansatz
\ref{ans:NHIMCircular:bis}, since on the periodic orbits we have
$\dot t=\dot s=\pa_I E(I)=\nu(I)$ and, therefore, the period is
just  $2\pi\nu(I)$.
\end{proof}

%

\begin{corollary}\label{coro:NHIMCircular:MoserCoord}
Assume Ansatz~\ref{ans:NHIMCircular:bis} holds. Then the Hamiltonian \eqref{def:HamCircFinal} with $\mu=0.95387536\times10^{-3}$  has an analytic normally 
hyperbolic invariant $3$-dimensional cylinder
\begin{equation}\label{def:NHIMMoser}
\Lambda_0=\left\{(p,q)=(0,0),\ I\in [I_-,I_+],\ (s,g)\in\TT^2\right\},
\end{equation}
 which is foliated by $2$-dimensional invariant tori.

Moreover, in the constant invariant
planes $I=I_0$,  for every $I_0\in [I_-,I_+]$, 
the invariant manifolds of the cylinder,  $W^{s}( \Lambda_0)$ and $W^{u}( \Lambda_0)$, 
intersect transversally at two homoclinic channels $\CCC^1_{0}$ and
$\CCC^2_{0}$. Moreover, these channels intersected with $\{g=0\}$ are 
parameterized as
\begin{equation}\label{def:HomoChannels}
 (I,s,p,q)=( I,s,0,q^1(I)) \qquad \text{ and }\qquad  (I,s,p,q)=(I,s,0,q^2(I))
\end{equation}
for some smooth functions $q^i:[I_-,I_+]\to\RR$ and $s\in\TT$, respectively.
\end{corollary}

Corollaries \ref{coro:NHIMCircular} and \ref{coro:NHIMCircular:MoserCoord}
give the existence of a normally hyperbolic invariant cylinder with
transverse homoclinic channels. Classical perturbation theory ensures the
persistence of such structure for $e_0>0$ small enough.
To analyze the perturbed problem, we start by expressing the elliptic  Hamiltonian
perturbation $\Delta K_\eell$ in \eqref{def:HamEll:Reparam} in the new coordinates. We define
\begin{equation}\label{def:EllipticPerturbMoser}
 \Delta \KK_\eell=\Delta K_\eell\circ \Psi^{-1}\circ\Upsilon^{-1}
\end{equation}
where $\Psi$ and $\Upsilon$ are the symplectic changes of coordinates introduced in Lemmas
\ref{lemma:ChangePO} and 
\ref{lemma:MoserNF} respectively.
We
 define  the transformed elliptic Hamiltonian
\begin{equation}\label{def:EllipticMoser}
\KK=K\circ \Psi^{-1}\circ\Upsilon^{-1}=\KK_0+\mu
e_0 \Delta \KK_\eell
\end{equation}
(see \eqref{def:HamCircFinal}). Classical perturbative arguments imply that this Hamiltonian  has a normally
hyperbolic invariant
cylinder $\Lambda_{e_0}$, which is $e_0$-close to the invariant cylinder
$\Lambda_0$ in \eqref{def:NHIMMoser}.

Analogously, we can define the cylinder for the Poincar{\'e} map
$\PP_{e_0}$ associated to this Hamiltonian and the section $\{g=0\}$. It is defined by
$\wt\Lambda_{e_0}=\Lambda_{e_0}\cap \{g=0\}$. (Recall that changes performed
in Lemmas \ref{lemma:ChangePO} and \ref{lemma:MoserNF} have not modified $g$ and
thus the section is defined independently of the system of coordinates).


By Corollary \ref{coro:NHIMCircular}, for $I\in [I_-,I_+]$,  the Poincar\'e map $\PP_0$ has transverse homoclinic channels 
\begin{equation}\label{def:channelsPoincarecirc}
\wt\CCC_{0}^i=\CCC_{0}^i\cap\{g=0\}, \qquad\ i=1,2.
\end{equation}
By classical perturbative arguments, for
 $e_0>0$ small enough,
the elliptic problem also possesses 
transverse homoclinic channels
$\wt\CCC^i_{\ee},\ i=1,2$ 
which are $\OO(e_0)$-close to $\wt\CCC_{0}^i$.

\begin{theorem}\label{th:Elliptic:NHIM}
Assume Ansatz~\ref{ans:NHIMCircular:bis} holds.
Let $\PP_{e_0}$ be the Poincar{\'e} map associated to 
the Hamiltonian $\KK$ in \eqref{def:EllipticMoser} and the section $\{g=0\}$.
For any $\de>0$, there exists $e_0^\ast>0$ such that for 
$e_0\in (0,e_0^\ast)$ the map $\PP_{e_0}$  has a  normally  hyperbolic  
weakly invariant manifold  $\wt\Lambda_{e_0}$, which is $e_0$-close in 
the $C^1$-topology to the unperturbed cylinder  
$\wt\Lambda_0=\Lambda_0\cap \{g=0\}$ of  the circular problem and 
can be written as a graph over it. Namely, there exists a function 
$\GG_{e_0}: [I_-+\de,I_+-\de]\times\TT\rightarrow \RR^3\times\TT$ such that
  \[
  \wt\Lambda_{e_0}=\left\{ \GG_{e_0}(I,s):(I,s)\in[I_-+\de,I_+-\de]\times
\TT\right\}.
  \]

Moreover, in the region $I\in [I_-+\de,I_+-\de]$,  the invariant
manifolds $W^{u}(\wt\Lambda_{e_0})$ and 
 $W^{s}(\wt\Lambda_{e_0})$ intersect transversally
along homoclinic channels $\wt\CCC^i_{\ee},\ i=1,2$,
which are $\ee$-close to the homoclinic channels
$\wt\CCC^i_{0},\ i=1,2$ respectively.
\end{theorem}


\subsection{The separatrix map for the circular and   elliptic problems}

We now define the separatrix maps associated to this normally hyperbolic
invariant manifold and the attached homoclinic channels.
Consider open domains $\mathcal U^i\subset\{g=0\}$ close the  homoclinic channels $\wt\CCC^i_{\ee}$. Then we define the return time functions $N^{ij}:\UU^i\to \ZZ$ as
\begin{equation}\label{def:ReturnTime}
 N^{ij}(z)=\min\left\{n>0: \PP_{e_0}^n(z) \in \UU^j \right\}.
\end{equation}
Note that $N^{ij}$ may be infinite for some points in $\UU^i$. We define the  open  sets 
\[
\UU^{ij}=\left\{z\in\UU^i: N^{ij}(z)<\infty\right\},\qquad i,j\in\{1,2\},
\]
and the separatrix maps
\begin{equation}\label{def:SMij} \SM_{e_0}^{ij}:\UU^{ij}\to \UU^j\qquad \text{
as }\qquad  \SM_{e_0}^{ij}(z)=\PP_{e_0}^{N^{ij}(z)}(z).
\end{equation}
Note that $N^{ij}$ is locally constant. Thus, in suitable open subsets of $\UU^{ij}$ these maps are smooth.

To derive formulas for the separatrix maps in Theorems
\ref{thm:FormulasSMii:circ} and \ref{thm:FormulasSMii} below, we restrict them
to strictly smaller domains $\UU_\rho^{ij}\subset\UU^{ij}$ defined as follows.
Fix $\de>0$ small and  $\kk>1$,  we define the subsets
\begin{equation}\label{def:SubDomainUij}
 \UU_\rho^{ij}\subset \left\{
(I,s,p,q)\in  \UU^{ij} :   I\in [I_-+\de,I_+-\de],\ 
s\in\TT,\
\rr^\kk\leq  p\leq \rr,\
|q-q^i(I)|\leq \rr\right\},\,\, i,j=1,2,
\end{equation}
 where $q^i$ are the homoclinic channels introduced in
\eqref{def:HomoChannels}. 


%
%
%
%
We give the formulas for the separatrix map $\SM^{ij}_{e_0}$ in two steps.
First, in Theorem \ref{thm:FormulasSMii:circ}, we provide them for
$\SM_{e_0}^{ij}$ with $e_0 = 0$ (that is for the RPC3BP). Then, in Theorem
\ref{thm:FormulasSMii}, we give formulas for $\SM_{e_0}^{ij}$ for $0<e_0\ll 1$.

Before, stating these theorems we introduce some notation.

\begin{notation}\label{not:harmonics} For every smooth function $f$ that is
$2\pi$-periodic in $t$, we define $\NNN (f )$ as the set of integers $k\in\ZZ$
such that the $k$-th harmonic of $f$ (possibly depending on other variables) is
non-zero.
\end{notation}

\begin{theorem}\label{thm:FormulasSMii:circ}
Assume  Ansatz \ref{ans:NHIMCircular:bis} and fix $\de>0$, $\kk>1$. There 
exists $0<\rr\ll 1$
such that, on the domain $ \UU_\rho^{ij}$ in \eqref{def:SubDomainUij},
%
%
the separatrix map $\SM_0^{ij}:  \UU_\rho^{ij}\to \UU^j$  in
 \eqref{def:SMij}
is well defined and its components 
$\SM_0^{ij}=(\FF^{ij}_{I,0},\FF^{ij}_{s,0},\FF^{ij}_{p,0},\FF^{ij}_{q,0})$ 
have the following form.
%
\begin{itemize}
\item The $(I,s)$ components are given by
\[
\begin{pmatrix}
\FF_{I,0}^{ij}\\\FF_{s,0}^{ij}\end{pmatrix}=\begin{pmatrix}I\\
s+\al_i(I)+(\nu(I)+\pa_I F(I,pq))N^{ij}(I,p,q)+ D^i(I,p,q)\end{pmatrix},
\]
where  $F$ and $\nu$ are the functions introduced in
Lemma \ref{lemma:MoserNF}, $N^{ij}$ is the return time introduced in 
\eqref{def:ReturnTime}, which is locally constant, and satisfies 
\[
 N^{ij}(I,p,q)\sim |\log\rr|
\]
and $D^i$ is of the form
\[
D^i(I,p,q)=\, a^i_{10}(I)p+
a_{01}^i(I)\left(q-q^i(I)\right)+\OO_2\left(p,q-q^i(I)\right).
 \]
\item The $(p,q)$ components are of the form
\[
\begin{split}
& \begin{pmatrix}\FF_{p,0}^{ij}\\\FF_{q,0}^{ij}\end{pmatrix}=
\\
& \,\begin{pmatrix}
  e^{-\Theta^i(I,p,q)N^{ij}} &0\\
 0 &e^{\Theta^i(I,p,q)N^{ij}}
 \end{pmatrix}\left[\begin{pmatrix}p_0^i(I)\\0\end{pmatrix}+ \begin{pmatrix}
 b_{10}^i(I) &b_{01}^i(I)\\
 c_{10}^i(I) &c_{01}^i(I)
 \end{pmatrix}
\begin{pmatrix}
  p\\ q-q^i(I)
 \end{pmatrix}+\OO_2\left(p,q-q^i(I)\right)\right],
 \end{split}
\]
with 
\begin{equation}\label{def:Theta}
 \Theta^i(I,p,q)=\pa_2 F\left(I^i_*(I,p,q), 
p_*^i(I,p,q)q_*^i(I,p,q)\right),\end{equation}
where $I_*^i, p_*^i, q_*^i$ are the images of the gluing map introduced in \cite{paperNHIL23} (see
Lemma~7.5)  and $F$ is given by Lemma 
\ref{lemma:MoserNF}.
\end{itemize}
Moreover,
%
the functions $a_*^i$, $b_*^i$, $c_*^i$ and $p_0^i$  satisfy
\begin{equation}\label{eq:GluingSymplectic}
 b_{10}^ic_{01}^i-b_{01}^ic_{10}^i=1,\quad a_{10}^i=c_{10}^i\pa_Ip_0^i,\quad
a_{01}^i=c_{01}^i \pa_Ip_0^i
\end{equation}
and
\begin{equation}\label{eq:GluingTransversality}
 c_{01}^i\neq 0.
\end{equation}
\end{theorem}

\begin{theorem}\label{thm:FormulasSMii}
Assume Ansatz \ref{ans:NHIMCircular:bis} and fix $\de>0$, $\kk>1$,  $M>0$. 
There exists $0<\rr\ll 1$ and $e_0^*>0$
such that for any $e_0\in (0,e_0^ *)$, there exists a system of
coordinates
$(\wh I,\wh s,\wh p,\wh
q)$ on
\[
 \mathcal D=\left\{(I,s,p,q,g): I\in [I_-+\de,I_+-\de], (s,g)\in\TT^2,
|p|,|q|\leq M, |pq|\leq\rr\right\},
\]
satisfying
\[
(\wh I,\wh s,\wh p,\wh q)= (I,s,p,q)+\OO(e_0),
\]
such that, on the domain  $\UU_\rho^{ij}$ introduced in \eqref{def:SubDomainUij},
which satisfies $\UU_\rho^{ij}\subset \mathcal D$,
%
%
the separatrix map $\SM_{e_0}^{ij}:  \UU_\rho^{ij}\to \UU^j$  in
 \eqref{def:SMij}
is well defined.
Moreover, dropping the hats in the coordinates,
 the separatrix map
$\SM_{e_0}^{ij}=(\FF^{ij}_{I,e_0},\FF^{ij}_{s,e_0},\FF^{ij}_{p,e_0},\FF^{ij}
_{q,e_0})$ is as follows.
%
\begin{itemize}
\item The $(I,s)$ components are given by\footnote{In the estimates of the errors below, note that the power in the logarithms could be more accurate. Indeed, the $C^0$ norm of the errors is $\OO(e_0^3\log \rr)$ and the $C^1$ norm is $\OO(e_0^3\log^2 \rr)$. Since, we consider $\rr$ fixed and take $e_0$ arbitrarily small, these logarithms do not play any role.}
\[
\begin{pmatrix}\FF^{ij}_{I,e_0}\\
\FF^{ij}_{s,e_0}\end{pmatrix}=\begin{pmatrix}\FF^{ij}_{I,0}+e_0
\FF^{ij}_{I,1}+e_0^2
\FF^{ij}_{I,2}+\OO_{C^2}\left(e_0^3\log^3\rr\right)\\
\FF^{ij}_{s,0}+e_0
\FF^{ij}_{s,1}+\OO_{C^2}\left(e_0^2\log^3\rr\right)
\end{pmatrix},
\]
where $\FF^{ij}_{I,0}$, $\FF^{ij}_{s,0}$ are the functions introduced in
Theorem \ref{thm:FormulasSMii:circ} and
\[
\begin{split}
\FF^{ij}_{I,1}&=\MM^{I,1}_i(I,s)+\MM^{I,1}_{i,pq}(I,s,p,q)\\
\FF^{ij}_{I,2}&=\MM^{I,2}_i(I , s)+\MM^{I,2}_{i,pq}(I, s,p,q)\\
\FF^{ij}_{s,1}&=\MM^{s,2}_i(I,s)+\MM^{s,1}_{i,pq}(I,s,p,q),
\end{split}
\]
where
$\MM^{*,*}_i$ are $C^2$ functions which satisfy
\[
\MM^{I,l}_i=\OO_{C^2}(1),\quad  \MM^{s,1}_i=\OO_{C^2}(\log\rr), \quad
\MM^{I,l}_{i,pq}=pq\OO_{C^2}(1),\quad  
\MM^{s,1}_{i,pq}=pq\OO_{C^2}(\log\rr),
\]
and
\[
\NNN(\MM^{z,1}_{i})=\NNN(\MM^{z,1}_{i,pq})=\{\pm 1\},\,\,\,z=I,s\quad
\text{and}\quad \NNN(\MM^{I,2}_i)=\NNN(\MM^{I,2}_{i,pq})=\{0,\pm 2\}.
\]
\item The $(p,q)$ components are of the form
%
\[
\begin{pmatrix}\FF^{ij}_{p,e_0}\\
\FF^{ij}_{q,e_0}\end{pmatrix}=\begin{pmatrix}\FF^{ij}_{p,0}+e_0
\FF^{ij}_{p,1}+\OO_{C^2}\left(e_0^2\log\rr\right)\\
                                           \FF^{ij}_{q,0}+e_0
\FF^{ij}_{q,1}+\OO_{C^2}\left(e_0^2\log\rr\right)
                                          \end{pmatrix},
\]
where $\FF^{ij}_{p,0}$, $\FF^{ij}_{q,0}$ are the functions introduced in
Theorem \ref{thm:FormulasSMii:circ}
and the functions $\FF^{ij}_{p,1}$, $\FF^{ij}_{q,1}$ satisfy
$\FF^{ij}_{p,1},\FF^{ij}_{q,1}=\OO_{C^2}(\log\rr)$ and
$\NNN(\FF^{ij}_{p,1})=\NNN(\FF^{ij}_{q,1})=\{\pm 1\}$.
\end{itemize}
\end{theorem}

The proofs of these theorems can be found in \cite{paperNHIL23}.

For our
analysis we need explicit formulas for the functions $\MM_i^{I,1}$ and
$\al_i(I)$. Moreover, we need to express these functions in terms of the
original system
\eqref{def:Reduced:ODE}. Indeed, they will be defined through the functions
$\Delta H_\ccirc$ and $\Delta H_\eell$ introduced in
\eqref{def:HamDelaunayNonRot}. It can be checked (see \cite{FejozGKR15}) that
they satisfy
\begin{equation}\label{def:HarmonicsHams}
 \NNN\left(\Delta H_\ccirc\right)=\{0\}\qquad\text{ and }\qquad\NNN\left(\Delta
H_\eell |_{e_0=0} \right)=\{\pm 1\}.
\end{equation}

We need to introduce some more notation. Consider system
\eqref{def:Reduced:ODE:circular}. Since the right hand side does
not depend on $t$, one can consider it for the coordinates $(\ell, L, g, G)$
 treating $I$ as a constant. Denote by $\Phi^\ccirc_0$ its flow. Ansatz
\ref{ans:NHIMCircular:bis} implies that, for each
value $I\in [I_-, I_+]$,  this flow has a periodic orbit with two transverse
homoclinic orbits. We denote the time-parameterization of this periodic orbit
by $\la_I(s)$ and those of the two
homoclinic orbits by $\ga_I^i(s)$ with $i=1,2$. Moreover, since $\dot t$ in
\eqref{def:Reduced:ODE:circular} only depends on $(\ell, L, g, G, I)$, it can be
integrated on the periodic orbit to obtain
\[
 t=t_0+\wt\la_I(s),\qquad \wt\la_I(s)=\int_0^s\frac{1}{-1+\mu\pa_G\Delta
H_\ccirc(\la_I(\sigma))}d\sigma
\]
and analogously on the homoclinic orbits to obtain  $t=t_0+\wt\ga^i_I(s)$.

\begin{proposition}\label{prop:Melnikov}
The functions $\alpha_i(I)$ and $\MM_i^{I,1}(I,s)$, introduced in  Theorems \ref{thm:FormulasSMii:circ} and 
\ref{thm:FormulasSMii} respectively, satisfy the following.
\begin{itemize}
\item $\alpha_i(I)$ can be written as $ \alpha_i(I)= \alpha^+_i(I)-
\alpha^-_i(I)$
with
\begin{equation}\label{def:omega}
\alpha^\pm_i(I)
=\mu\lim_{N\to\pm\infty}\left(\int_0^{2\pi N}\frac{\left(\pa_G
\Delta H_\ccirc\right)\circ \ga_I^i(\sigma)}{-1+\mu \left(\pa_G
\Delta H_\ccirc\right)\circ \ga_I^i(\sigma)}d\sigma+2\pi N\nu(I)\right)
\end{equation}
where $\nu$ is the function defined in \eqref{def:nu}.
\item The function  $\MM_i^{I,1}(I,s)$ satisfies
$\NNN(\MM_i^{I,1}(I,s))=\{\pm 1\}$ and can be written as
\[
\MM_i^{I,1}(I,s)=B_i(I)e^{\im s}+\ol{B_i}(I)e^{-\im s}
\]
with $B_i=B_i^\inn+B_i^\out$ where
%
\begin{align}\label{def:Melnikovs}
 B_i^\inn(I)=&\im\mu\frac{1-e^{\im\alpha_i
(I)}}{1-e^{\im 2\pi  \nu(I)}}\int_0^{2\pi}
\frac{\Delta H_\eell^{1,+}\circ\la_I(\sigma)}{-1+\mu \pa_G\Delta
H_\ccirc\circ\la_I(\sigma)}
e^{\im\wt\la_I(\sigma)}d\sigma.\\
 B_i^\out(I)=&- \im\mu\lim_{T\rightarrow+\infty}\int_0^T\left(\frac{\Delta H^{1,+}_\eell\circ\gamma_I^i(\sigma)}{-1+\mu\pa_G\Delta H_\ccirc\circ\gamma_I^i(\sigma)}e^{ \im\wt\gamma_I^i(\sigma)}\right.\nonumber\\
    &\qquad\qquad\qquad\left.-\frac{\Delta
H_\eell^{1,+}\circ\la_I(\sigma)}{-1+\mu\pa_G\Delta
H_\ccirc\circ\la_I(\sigma)}e^{
\im\left(\wt\la_I(\sigma)+\alpha^+_i(I)\right)}\right) d\sigma\\
    & +\im\mu\lim_{T\rightarrow-\infty}\int_0^T\left(\frac{\Delta H_\eell^{1,+}\circ\gamma_I^i(\sigma)}{-1+\mu\pa_G\Delta H_\ccirc\circ\gamma_I^i(\sigma)}e^{ \im\wt\gamma_I^i(\sigma)}\right.\nonumber\\
    &\qquad\qquad\qquad\left.-\frac{\Delta
H_\eell^{1,+}\circ\la_I(\sigma)}{-1+\mu\pa_G\Delta
H_\ccirc\circ\la_I(\sigma)}e^{
\im\left(\wt\la_I(\sigma)+\alpha^-_i(I)\right)}\right)d\sigma,\nonumber
\end{align}
where $\Delta H_\eell^{1,+}$ is defined by $\Delta
H_\eell^{1}(L, \ell,G, g,t)=\Delta H_\eell^{1,+}(L, \ell,G, g)e^{\im t}+\Delta H_\eell^{1,-}(L, \ell,G, g)e^{-\im t}$,  with $\Delta
H_\eell^{1}=\Delta
H_\eell|_{e_0=0}$
(see \eqref{def:HarmonicsHams}).
\end{itemize}
\end{proposition}
This proposition is proved in \cite{paperNHIL23}.

\section{The normally hyperbolic invariant lamination}\label{sec:NHIL}

In this section we construct the Normally Hyperbolic Invariant Lamination in two steps. 
First we analyze the Poincar\'e map associated to the circular problem \eqref{def:HamDelaunayCirc}. 
This system is autonomous and hence $I$ is a constant of motion. Therefore building the lamination 
is reduced to proving the existence of a horseshoe at each level $I=\text{constant}$ and showing its 
regularity with respect to $I$. We prove its existence by an isolating block technique.
As a second step, we show the existence of the NHIL for  (the Poincar\'e map of) the elliptic problem 
\eqref{def:HamDelaunayRot} with $0<e_0\ll 1$. This is a regular perturbation problem and, therefore, 
the existence of the lamination is a consequence of normal hyperbolicity.

A normally hyperbolic invariant lamination generalizes the concepts of hyperbolic sets and normally 
hyperbolic invariant manifolds (see~\cite{HirschPS77}). Here we  consider a less general definition than the one  in~\cite{HirschPS77}, 
which is be enough for our purposes. It follows the ideas in~\cite{KaloshinZ15}, where normally 
hyperbolic invariant laminations for a priori unstable nearly integrable Hamiltonian systems are constructed. 

First, we define a normally hyperbolic invariant manifold.

\begin{definition}\label{def:NHIM}
Given a $C^1$ map $F$ on a manifold $M$, we say that  that a submanifold 
$N\subset M$ is normally hyperbolic for the map $F$ if it is $F$--invariant and 
one can split $T_NM$ into three invariant subbundles,
\[
 T_xM=T_xN\oplus E_x^s\oplus E_x^u,\qquad \text{for all }\quad x\in N
\]
which satisfy
\[
 \begin{split}
  \|DF^n(x)|_{E^s}\|\leq C \la^n\qquad \text{ for }n\geq 0\\
  \|DF^n(x)|_{E^u}\|\leq C \la^{|n|}\qquad \text{ for }n\leq 0\\
  \|DF^n(x)|_{TN}\|\leq C \eta^{|n|}\qquad \text{ for }n\in\NN
 \end{split}
\]
for some $C>0$, $\la<1$ and $\eta\geq 1$ such that $\eta\la<1$.
\end{definition}

The concept of normally hyperbolic invariant lamination is a generalization of 
this definition. To this end, we define first the following.

\begin{definition}
\label{def:Shift}
Let $\Sigma = \{0,1\}^{\ZZ}$. We define the Bernoulli  
shift $\sigma:\Sigma \to \Sigma$ as $(\sigma\omega)_k=\omega_{k+1}$.
\end{definition}

\begin{definition}\label{def:NHIL}
Consider a $C^1$ map $F$ defined on a manifold $M$. We say that 
a closed set $\Xi$ is a normally hyperbolic invariant lamination if 
\[
 \Xi=\bigcup_{\omega\in\Sigma}\Xi_\omega
\]
where $\Xi_\omega$ are $C^1$ manifolds which satisfy 
$F(\Xi_\omega)=\Xi_{\sigma\omega}$ and is normally hyperbolic in the following 
sense. There exist constants  $C>0$, $0<\la<1$ and $\eta\geq 1$ satisfying 
$\eta\la<1$ and subbundles $E^s$ and $E^u$ such that for any $\omega\in\Sigma$ 
and $x\in\Xi_\omega$,
\[
 T_xM=T_x\Xi_\omega\oplus E_x^s\oplus E_x^u
\]
which are invariant and satisfy 
\[
 \begin{split}
  \|DF^n(x)|_{E^s}\|\leq C \la^n\qquad \text{ for }n\geq 0\\
  \|DF^n(x)|_{E^u}\|\leq C \la^{|n|}\qquad \text{ for }n\leq 0\\
  \|DF^n(x)|_{T\Xi_\omega}\|\leq C \eta^{|n|}\qquad \text{ for }n\in\NN.
 \end{split}
\]
Finally, we say that $\Xi$ is a  normally hyperbolic lamination weakly invariant under $F$ if there exist a neighborhood $U$ of $\Xi$ such that, if there exist $x\in\Xi$ such that $F(x)\not\in \Xi$ implies $F(x)\not\in U$.
\end{definition}

The next two theorems prove the existence of such objects. First, Theorem 
\ref{thm:NHILCircular} for the RPC3BP (Hamiltonian 
\eqref{def:HamDelaunayCirc}) and then Theorem \ref{thm:NHILElliptic} for the 
RPE3BP with $0<e_0\ll 1$ (Hamiltonian \eqref{def:HamDelaunayRot}).

\begin{theorem}\label{thm:NHILCircular}
Fix $0<\rr\ll 1$ and consider $\rr$-neighborhoods $B^i_\rr$ of the homoclinic 
channels $\{\wt\CCC^i_0\}_{I\in \II}$ (see \eqref{def:channelsPoincarecirc}) of  the Poincar\'e map of the circular problem 
\eqref{def:HamDelaunayCirc}. Then, the  separatrix map $\SM_0$ associated to this 
problem given in  Theorem \ref{thm:FormulasSMii:circ} have a NHIL 
denoted by $\LL_0$ satisfying $\LL_0\subset B^1_\rr\cup B^2_\rr$. That is,
\begin{itemize}
\item
The set $\LL_0$ is invariant: there exists an embedding
\[
 L_0:\Sigma\times\TT\times \II\to B^1_\rr\cup B^2_\rr\qquad L_0(\omega, I, 
s)=(I,s,P_0(\omega,I), Q_0 (\omega,I))
\]
and a function $\beta (\omega,I)$, which are $C^3$ in $I$ and $\vartheta$-H\"older in 
$\omega$, for some $\vartheta\in (0,1)$ independent of $\rho$, such that
\[
 \SM_0\left(I,s,P_0(\omega,I), Q_0(\omega,I)\right)=\left(I,s+\beta(\omega,I),P_0({\sigma\omega},I), Q_0 ({\sigma\omega},I)\right).
\]
\item The set $\LL_0$ is normally hyperbolic.
\end{itemize}
Moreover, there exists $\ero'>\ero_*>0$ such that
\[
 \rr^{\ero'}\lesssim  \inf_{z\in \LL_0, \ z'\in\widetilde{\CCC}^i_0,\ i=1,2}\mathrm{ dist}(z,z')\leq  \sup_{z\in \LL_0, \ z'\in\widetilde{\CCC}^i_0,\ i=1,2}\mathrm{ dist}(z,z')\lesssim  \rr^{\ero_*}.
\]
\end{theorem}
This theorem is proven in \cite{paperNHIL23}.

For $e_0>0$ small enough the NHIL persists for the Hamiltonian system associated to  \eqref{def:HamDelaunayRot} and it is smooth with respect to $e_0$.
\begin{theorem}\label{thm:NHILElliptic}
 Fix $0<\rr\ll 1$ and consider $\rr$-neighborhoods $B^j_\rr$ of the homoclinic 
channels $\{\wt\CCC^i_{e_0}\}_{I\in \II}$ (see Theorem \ref{thm:NHILElliptic}) of  the Poincar\'e map of the elliptic problem 
\eqref{def:HamDelaunayRot}. Then, the  separatrix map $\SM_{e_0}$ associated to this 
problem given in  Theorem \ref{thm:FormulasSMii} has a NHIL denoted by 
$\LL_{e_0}$ which is $e_0$-close to the NHIL $\LL_0$ obtained in Theorem 
\ref{thm:NHILCircular} and  satisfies $\LL_{e_0}\subset B^1_\rr\cup B^2_\rr$. That 
is,
\begin{itemize}
\item
The set $\LL_{e_0}$ is invariant: there exists an embedding
\[
 L_{e_0}:\Sigma\times\TT\times \II\to B^1_\rr\cup B^2_\rr\qquad L_{e_0}(\omega, I, s)=(I,s,P_{e_0}(\omega,I,s), Q_{e_0} (\omega,I,s))
\]
and functions $S_{e_0}(\omega,I,s)$ and $J_{e_0}(\omega,I,s)$, which are 
$C^3$ in $(I,s,e_0)$ and $\vartheta$-H\"older in $\omega$, for some $\vartheta\in (0,1)$ independent of $\rho$,  such that
\begin{multline}
 \SM_{e_0}\left(I,s, P_{e_0}(\omega,I,s), Q_{e_0}(\omega,I,s)\right) \\
 \label{eq:inveq_lamination_elliptic}
 =\left(\II_{e_0}(\omega,I,s),\SSS_{e_0}(\omega,I,s),P_{e_0}\circ\FF_{\iinn}(\omega,I,s), Q_{e_0}\circ\FF_{\iinn}(\omega,I,s)\right)
\end{multline}
with
\begin{equation}
\label{def:inner_map_elliptic}
\FF_{\iinn}(\omega,I,s)=(\sigma\omega,  \II_{e_0}(\omega,I,s), \SSS_{e_0}(\omega,I,s))
\end{equation}
and
\[
\begin{split}
 \II_{e_0}(\omega,I,s)&=I+J_{e_0}(\omega,I,s)\\
\SSS_{e_0}(\omega,I,s)&=s+\beta(\omega,I)+S_{e_0}(\omega,I,s).
\end{split}
\]
\item The set $\LL_{e_0}$ is normally hyperbolic.
\end{itemize}
Moreover,
\begin{itemize}
\item There exists $\ero'>\ero_*>0$ such that
\[
 \rr^{\ero'}\lesssim   \inf_{z\in \LL_0,\ z'\in \wt\CCC^i_{e_0},\ i=1,2}\mathrm{ dist}(z,z')\leq\sup_{z\in \LL_0,\ z'\in \wt\CCC^i_{e_0},\ i=1,2}\mathrm{ dist}(z,z')\lesssim\rr^{\ero_*}.
\]
\item The functions $P_{e_0}$, $Q_{e_0}$ are of the form
\[
(P_{e_0}, Q_{e_0})=(P_0, Q_0)+e_0(P_1, Q_1)+\OO(e_0^2)\quad \text{with}\quad \NNN(P_1)=\NNN(Q_1)=\{\pm 1\},
\]
where $(P_0, Q_0)$ are the functions obtained in Theorem \ref{thm:NHILCircular}.
\item The functions $S_{e_0}$ and $J_{e_0}$ are of the form
\[
\begin{split}
J_{e_0}=&\,e_0J_1+e_0^2J_2+\OO(e_0^3)\qquad \text{with}\qquad \NNN(J_1)=\{\pm 1\},\,\, \NNN(J_2)=\{0,\pm 2\},\\
S_{e_0}=&\, e_0S_1+\OO(e_0^2)\qquad \text{with}\qquad \NNN(S_1)=\{\pm 1\}.
\end{split}
\]
\end{itemize}
\end{theorem}

\section{Dynamics on the NHIL}\label{sec:RandomCylinder}
We analyze the dynamics on the invariant lamination constructed  in Theorem \ref{thm:NHILElliptic}. The separatrix map, obtained in Theorem \ref{thm:FormulasSMii}, induces a map on the lamination, given by
\begin{equation}\label{def:laminationmap}
\begin{split}
 \FF:\,\, \Sigma\times [I_-,I_+]\times\TT &\longrightarrow  \Sigma\times \RR\times\TT\\
(\omega, I,s)&\ \mapsto \ \left(\sigma\omega, \FF_\omega (I,s)\right).
 \end{split}
 \end{equation}
 In \cite{KaloshinZ15}, the authors construct laminations in a priori unstable nearly integrable Hamiltonian systems, which are $\eps$-close to a pendulum plus a rotator. In their setting, the lamination is $\eps$-close to the 
homoclinic channels and therefore the leaves are $\eps$-close to each other. 
This implies that the lamination map has a very explicit first order which is  
essentially given by the Melnikov potential and the first symbol $\omega_0$. 
That is
\[
 \FF_\omega (I,s)=\FF^0 (\omega_0,I,s)+\text{h.o.t in $\eps$ depending on all $\omega$}.
\]
That is, the dynamics on the lamination can be ``essentially'' reduced to the random iteration of two cylinder maps. These higher order terms need to be $\OO(\eps^{2+a})$ in the $I$-component to carry out the martingale analysis.

Instead, we consider different laminations, whose distance to the homoclinic channels is small but independent of $e_0$ (see Section \ref{sec:NHIL}). Indeed, the leaves, which  are cylinders with
boundaries,  are  localized  $\rr^{\ero_*}$-close to the homoclinic channels
where $\ero_*=\min_{I\in [I_-,I_+]}\ero(I)$ (see Theorem \ref{thm:NHILElliptic}). Note that this is possible because
this model is \emph{a priori chaotic} (instead of \emph{a priori unstable}). This implies that the dynamics cannot be reduced to two
cylinder maps plus small errors.
However, since both $e_0$ and $\rr$ are taken small, $e_0\ll\rr$, we have ``good''
expansions on the separatrix maps restricted into the weakly invariant lamination.
%

Next lemma gives a finer asymptotic expansion for the induced map compared to that in Theorem \ref{thm:NHILElliptic}. It is a direct consequence of Theorems \ref{thm:FormulasSMii} and \ref{thm:NHILElliptic}.

\begin{lemma}\label{lemma:laminationmap}
The separatrix map, obtained in Theorem \ref{thm:FormulasSMii}, induces a map 
$\FF$ of the form \eqref{def:laminationmap}
on the lamination, which is derived in Theorem \ref{thm:NHILElliptic},
\begin{equation}\label{def:stochasticmap}
 \FF_\omega: \begin{pmatrix} I\\ s\end{pmatrix}\to \begin{pmatrix} I+
e_0\AAA_\omega(I,s)+e_0^2\BB_\omega(I,s)+\OO\left(e_0^{2+a}\right)\\
s+\beta_\omega(I)+e_0\DD_\omega(I,s)+\OO\left(e_0^{1+a}\right)
\end{pmatrix}
\end{equation}
where $\beta_\omega$ satisfies
\begin{equation}\label{def:Omega}
\beta_\omega(I)=\beta^0_{\omega_0}(I)+\eta_{\omega}(I),\qquad
\beta^0_{\omega_0}(I)=\nu(I)\ol g+\alpha_{\omega_0}(I)
\end{equation}
and the other functions satisfy $\NNN( \AAA_\omega)=\NNN(\DD_\omega)=\{\pm 1\}$ and $\NNN(\BB_\omega)=\{0,\pm 2\}$,
\begin{equation}\label{def:ExpansionCylinderMap}
 \begin{split}
  \AAA_\omega(I,s)&=\MM^{I,1}_{\omega_0}(I,s)+\RRR^{I,1}_{\omega}(I,s)\\
  \BB_\omega(I,s)&=\MM^{I,2}_{\omega_0}(I,s)+\RRR^{I,2}_{\omega}(I,s)\\
  \DD_\omega(I,s)&=\wt\MM^{s}_{\omega_0}(I,s,\ol
g)+\RRR^{s}_{\omega}(I,s).
 \end{split}
\end{equation}
The functions $\MM^\ast_{\omega_0}$ and $\alpha_{\omega_0}$ are those
introduced in Theorem \ref{thm:FormulasSMii} (see also Proposition
\ref{prop:Melnikov}), $\nu$ is defined in \eqref{def:nu} and the time $\ol g$ is independent of $e_0$ and  satisfies $\ol g\sim |\log \rr|$.
Moreover,  the remainders in formulas \eqref{def:Omega} and
\eqref{def:ExpansionCylinderMap} satisfy the following estimates.
\[
\begin{split}
\left\|\RRR^{I,j}_{\omega}\right\|_{\Sigma\times\TT\times[I_-,I_+]} \ \ \ \ \ \ 
&\lesssim \ \ \  \rr^{\ero_*}\\
\left\|\eta_{\omega}\right\|_{\Sigma\times\TT\times[I_-,I_+]},
\left\|\RRR^{s}_{\omega}\right\|_{\Sigma\times\TT\times[I_-,I_+]}
\ \ \  & \lesssim \ \ \ \rr^{\ero_*}|\log\rr|.
\end{split}
\]
Moreover, there exists $\vartheta\in (0,1)$, such that the functions $\RRR^{I,j}_{\omega}$, $\RRR^{s}_{\omega}$, $\eta_{\omega}$ are $\vartheta$-H\"older with respect to $\omega$.
\end{lemma}




We study the stochastic behavior of the iteration of the map 
\eqref{def:stochasticmap} in time scales
$n\sim e_0^{-2}$.  Note that when $e_0=0$, $I$ is a constant of motion and therefore this map fits into the framework of compact Lie 
group extensions of hyperbolic maps. Indeed, is a skew-shift that can be seen as a random iteration of rigid rotations on a circle.  Such models are considered in 
\cite{FieldMT03} and analyzed in Section \ref{sec:LiegroupExtensions} below. 

Note that, since $\rr\ll 1$, the formula for $\beta_{\omega}$  in \eqref{def:Omega}
implies that each of these circle  maps (for fixed $\omega$) encounter many resonances, which are
$|\log \rr|\ii$-close. 
Nevertheless, they will not play any role in
the random map analysis since, thanks to Ansatz \ref{ansatz:Melnikov:1} below, the  maps have
different twist
condition and therefore they are not simultaneously resonant. 


\begin{ansatz}\label{ansatz:Melnikov:1}
The following is satisfied.
\begin{itemize}
\item The functions $\alpha_i, \ i=1,2$ in \eqref{def:Omega} (defined in Proposition \ref{prop:Melnikov}), satisfy
\[
 |\alpha_i(I)|\leq 100\mu\quad \textup{ for all}\,\,I\in [I_-,I_+].
\]
\item For all $I \in
[I_-,I_+]$, we have 
\[
\alpha_1(I)-\alpha_2(I) \neq 0.
\]
\item 
In \eqref{def:DriftVarFirstOrder} we define the first order of a certain variance, which we denote by  $\bms_0^2(I)$. It satisfies
\[
\bms_0^2(I)\  {\bf \neq }\ 0\quad \textup{for all}\, I\in [I_-,I_+].
\]
\end{itemize}
\end{ansatz}

These ansatz is verified  numerically in the companion paper \cite{paperNHIL23} and it is satisfied (as well as Ansatz \ref{ans:NHIMCircular:bis}) in the intervals given in Remark \ref{rmk:intervals}. Note that the additional assumptions in Ansatz \ref{ansatz:Melnikov:1} imply that one has to narrow the range of $I$'s compared to those where Ansatz \ref{ans:NHIMCircular:bis} is satisfied. 

\subsection{Compact Lie group extensions of hyperbolic maps}\label{sec:LiegroupExtensions}
The paper \cite{FieldMT03} studies compact Lie group extensions of hyperbolic maps and provides, under certain hypotheses, exponential decay of correlations and a central limit theorem for measures of type 2 (see Section \ref{subsec:measures}). In this section we state their results and generalize them to measures of type 1. We do not state their results in full generality but applied to maps of the form \eqref{def:stochasticmap}. In Section~\ref{app:LiegroupExtensions}, we show how to deduce the results for type 2 meaures presented in this section from \cite{FieldMT03} and we prove those for type 1 measures.

We introduce first some notation. We 
consider a map  $g:\Sigma \times\TT\to\Sigma\times\TT$ of the form
\begin{equation}\label{def:LieExtension}
 g(\omega, s)=\left(\sigma\omega, s+\beta_\omega\right),
\end{equation}
where $\sigma$ is an aperiodic subshift of finite type, $\beta_\omega$ is a phase shift which depends on $\omega$ and is independent of $s$.

For functions defined on $\Sigma$, we define the H\"older norm as follows. For a continuous function $f:\Sigma\to \CC^n$, we define, for each $N$,
the $N$--th variation as $\mathrm{var}_N f=\sup\{|f(\omega)-f(\omega')|\}$,
where the supremum is taken over all $\omega,\omega'\in\Sigma$ with 
$\omega_i=\omega'_i$ for $|i|<N$. Then, for $\vartheta\in (0,1)$, we define 
$|f|_\vartheta=\sup_{N\geq 1}\vartheta^{-N}\mathrm{var}_N f$ and the H\"older norm
\begin{equation}\label{def:HolderNorm}
 \|f\|_\vartheta=\|f\|_{C^0}+|f|_\vartheta.
\end{equation}
Note that this norm satisfies
\begin{equation}\label{def:HolderAlgebra}
	\|fg\|_\vartheta\leq \|f\|_\vartheta\ \|g\|_\vartheta.
\end{equation}
Assume that the function $\beta_\omega$ is $\vartheta$-H\"older and consider the following class of H\"older observables for the map $g$
\begin{equation}\label{def:observables}
\EE_{\vartheta,m}=\left\{F:\Sigma\times\TT\to \RR: F_\omega(s)=A_\omega e^{\im ms}+\ol{A_\omega} e^{-\im ms},\, \|A_\omega\|_\vartheta<\infty\right\}, \quad m\in\NN.
\end{equation}
We define  the H\"older norm of $F$, as the norm of the corresponding Fourier 
coefficient, that is $\|F_\omega\|_\vartheta=\|A_\omega\|_\vartheta$.

We assume the following condition on the map  \eqref{def:LieExtension}
\begin{equation}\label{def:WeakMixing}
\begin{gathered} 
\text{Fix $\vartheta\in (0,1)$. The equation }e^{\im m\beta_\omega} A_{\sigma\omega}=\nu A_\omega,\,m=1,2,4\\\quad \text{can only have one solution in $\EE_{\vartheta,m}$:} \text{ the trivial solution }\nu=1, A_\omega=\text{constant}.
\end{gathered}
\end{equation}
Note that \cite{FieldMT03} considers maps \eqref{def:LieExtension} which are
weak mixing, which implies this condition. The maps considered in the present
paper are not  (a priori) weak mixing since the functions $\beta$ could be simultaneously  rational. Nevertheless, the results in \cite{FieldMT03} do not
need weak mixing if one restricts the class of observables. Since in the present
paper, all observables are trigonometric polynomials of degree at most 2,
the condition of weak mixing can be weakened to impose \eqref{def:WeakMixing},
which is equivalent to say that the transfer operator does not have non-trivial
eigenfunctions within the classes of observables in $\EE_{\vartheta,m}$, $m=1,2,4$ (the reason of asking also weak mixing for $m=4$ is needed  to prove decay of correlations for $m=1,2$ in Theorem \ref{thm:Dima}).


Next theorem gives decay of correlations and a central limit theorem for maps of the form \eqref{def:LieExtension} and observables of the form \eqref{def:observables} for the product measure Bernouilli (for $\Sigma$) times Haar (for $\TT$). To emphasize what is the measure, we denote the expectation by $\E_{\omega, s}$.

\begin{theorem}\label{thm:Dima}
Fix $\vartheta\in (0,1)$. Consider a map of the form \eqref{def:LieExtension} with  $\beta_\omega$ $\vartheta$-H\"older and satisfying the condition \eqref{def:WeakMixing}. Denote by  $s_n$ the projection of the $n$-th iterate of
the map into the circle. Namely,
\[
s_n=s_0+\sum_{j=0}^{n-1}\beta_{\sigma^j\omega}.
\]
Then, there exist constants $r\in (0,1)$ and
$C>0$ such that the following is satisfied.
\begin{enumerate}
\item Take $m=1,2$ and   $F,G\in \EE_{\vartheta,m}$. Then, for  any $s_0\in\TT$,
\begin{equation}\label{def:DecayCorrelations:0}
\left|\E_{\omega, s}\left(G(\sigma^n\omega, s_n)F(\sigma^k\omega,
s_k)\right)\right|\leq C\,\|G\|_{\vartheta^{1/2}}\,\|F\|_{\vartheta^{1/2}}\ r^{n-k} \qquad \text{ for
}n>k.
\end{equation}
\item Any  $G(\omega,s)=A_\omega e^{\im s}+\ol{A_\omega} e^{-\im s}\in \EE_{\vartheta,1}$ satisfies that
$G_N=\frac{1}{\sqrt{N}}\sum_{n=0}^{N-1}G\left(\sigma^n\omega, s_n\right)$ converges weakly to a normal distribution $\mathcal{N}(0,\bms^2)$ with
\[
 \bms^2=2\E_{\omega, s} |A_\omega|^2+2\Re\left[\E_{\omega, s}\left(\sum_{n=1}^{+\infty}A_{\sigma^n\omega}\ol{A_{\omega}}e^{\im \sum_{j=0}^{n-1}\beta_{\sigma^j\omega}}\right)\right].
\]
\end{enumerate}
\end{theorem}

In Section \ref{app:LiegroupExtensions} we explain how to deduce this theorem from the results in \cite{FieldMT03}. It implies the following corollary.

\begin{corollary}\label{coro:decaycorrelations}
	Fix $\vartheta\in (0,1)$.	Consider a map of the form \eqref{def:LieExtension} with  $\beta_\omega$ $\vartheta$-H\"older and satisfying the condition \eqref{def:WeakMixing}. 
 Then, there exist constants $r\in (0,1)$ and
	$C>0$ such that the following is satisfied.
	
	For any $m=1,2$,  $n, k\in\NN$, $n>k$, any two functions $P,Q:\Sigma\to \CC$ which are $\vartheta$-H\"older, satisfy
	\[
	\left|\E_\omega \left( P\left(\sigma^n\omega\right)Q\left(\sigma^k\omega\right)e^{\im \sum_{j=k}^{n-1}m\beta_{\sigma^j\omega}}\right)\right|\leq C\,\|P\|_{\vartheta^{1/2}}\,\|Q\|_{\vartheta^{1/2}}\ r^{n-k}.
	\]
\end{corollary}

Theorem \ref{thm:Dima} analyzes observables for the measure given by the product of Bernouilli times Haar, which is invariant under the map \eqref{def:LieExtension}. However, in Section \ref{sec:Martingale}, we want to analyze the evolution of observables for fixed $s\in\TT$ and considering the  Bernouilli   measure for $\omega\in\Sigma$. We denote the corresponding expectation by $\E_\omega$.


\begin{theorem}\label{thm:Dima:FixedAngle}
Fix $\vartheta\in (0,1)$. Consider a map of the form \eqref{def:LieExtension}  with  $\beta_\omega$ $\vartheta$-H\"older and satisfying the condition \eqref{def:WeakMixing}. 
 Then, there exist constants $r\in (0,1)$ and
$C>0$ such that the following is satisfied.
\begin{enumerate}
\item Take $m=1,2$ and any  $F,G\in \EE_{\vartheta,m}$. Then, for any $s_0\in\TT$,
\begin{equation}\label{def:DecayCorrelations:1}
\left|\E_{\omega}\left(G(\sigma^n\omega, s_n)F(\sigma^k\omega,
s_k)\right)\right|\leq C\,\|G\|_{\vartheta^{1/2}}\,\|F\|_{\vartheta^{1/2}}\ r^{n-k} \qquad \text{ for
}n>k.
\end{equation}
\item	For any $m=1,2$, $n, k\in\NN$, $n>k$, any two functions $P,Q:\Sigma\to \CC$ which are $\vartheta$-H\"older, satisfy
\begin{equation}\label{def:DecayCorrelations:3}
\left|\E_\omega \left( P\left(\sigma^n\omega\right)e^{\im \sum_{j=0}^{n-1}m\beta_{\sigma^j\omega}} Q\left(\sigma^k\omega\right)e^{\im \sum_{j=0}^{k-1}m\beta_{\sigma^j\omega}}\right)\right|\leq C\,\|P\|_{\vartheta^{1/2}}\,\|Q\|_{\vartheta^{1/2}}\ r^{n}.
\end{equation}
\item Consider  $G(\omega,s)=A_\omega e^{\im s}+\ol{A_\omega} e^{-\im s}\in \EE_{\vartheta,1}$ and assume that\footnote{Note that Item 2 of this theorem ensures that the sum below is exponentially decreasing. The same happens in the infinite sum in the definition of the variance below.} 
\[
\sum_{n=0}^{\infty}\E_{\omega}G\left(\sigma^n\omega, s_n\right)=0.
\]
for any $s_0\in\TT$. Then, the normalized sum 
$G_N=\frac{1}{\sqrt{N}}\sum_{n=0}^{N-1}G\left(\sigma^n\omega, s_n\right)$ converges weakly to a normal distribution $\mathcal{N}(0,\bms^2)$ with
\[
 \bms^2=2\E_{\omega, s} |A_\omega|^2+2\Re\left[\E_{\omega, s}\left(\sum_{n=1}^{+\infty}A_{\sigma^n\omega}\ol{A_{\omega}}e^{\im \sum_{j=0}^{n-1}\beta_{\sigma^j\omega}}\right)\right].
\]
\end{enumerate}
\end{theorem}

This theorem is proven in Section \ref{app:LiegroupExtensions}. In Section \ref{sec:Martingale} below we will also need the following  slight modification of the second statement of Theorem \ref{thm:Dima:FixedAngle}. Its proof is also provided in Section~\ref{app:LiegroupExtensions}.

\begin{lemma}\label{lemma:tripledecay}
Fix $\vartheta\in (0,1)$ and assume $\beta_\omega$ $\vartheta$-H\"older. Then, there exist constants $r\in (0,1)$ and
$C>0$ such that, for any $m=1,2$, $n, k,i\in\NN$, $n>k>i$ and  any three functions $P,Q,S:\Sigma\to \CC$ which are $\vartheta$-H\"older, the following is satisfed,
\[
\left|\E_\omega \left( P\left(\sigma^n\omega\right) Q\left(\sigma^k\omega\right) S\left(\sigma^i\omega\right)e^{\im \sum_{j=0}^{n-1}m\beta_{\sigma^j\omega}}e^{\im \sum_{j=0}^{i -1}m\beta_{\sigma^j\omega}}\right)\right|\leq C\,\|P\|_{\vartheta^{1/2}}\,\|Q\|_{\vartheta^{1/2}}\,\|S\|_{\vartheta^{1/2}}\ r^{n}.
\]
\end{lemma}

\subsection{A normal form for the map $\FF_\omega$}\label{sec:Stochastic:NormalForm}
The first step in the martingale analysis is to perform a normal form procedure  to the map
\eqref{def:stochasticmap}
to kill
the $e_0$-fluctuations of the actions. 
To this end, we must show that the circle extensions of the shift considered by setting $e_0=0$ in \eqref{def:stochasticmap} satisfy condition \eqref{def:WeakMixing} and that, therefore, one can apply to them Theorems \ref{thm:Dima} and \ref{thm:Dima:FixedAngle}. More precisely, note that taking $e_0=0$, \eqref{def:stochasticmap}  gives a family of circle maps parameterized by the action $I\in[I_-,I_+]$. We show that the non-resonance conditions in Ansatz \ref{ansatz:Melnikov:1}  imply \eqref{def:WeakMixing}.
Indeed, by Item 1 and 2   of Ansatz \ref{ansatz:Melnikov:1}, there exists $\delta>0$ independent of $\rho$ and $e_0$ such that  $0<\delta\leq|\alpha_2(I)-\alpha_1(I_0)|\leq 1/5$ for all $I\in[I_-,I_+]$.
Therefore, 
if we define $B_\de(\ZZ)$ as the set of points in $\RR$ which are $\de$-close 
to $\ZZ$, we have that
\begin{equation}\label{ansatz:omegadifferent}
4(\alpha_2(I)-\alpha_1(I))\not \in B_\de(\ZZ) \qquad\text{ for all
}I\in
[I_-, I_+].
\end{equation}


\begin{lemma}\label{lemma:WeakMixing}
The map $g:\Sigma\times\TT\to\Sigma\times\TT$ induced by \eqref{def:stochasticmap} with  $e_0=0$ and $I\in
[I_-, I_+]$ satisfies condition \eqref{def:WeakMixing}.
\end{lemma}

\begin{proof}
Let us assume that we have a non-trivial solution $A$, $\nu$ of $e^{\im m\beta_\omega} A_{\sigma\omega}=\nu A_\omega$ and we reach a contradiction. Since $A_\omega\neq 0$ and periodic orbits are dense in $\Sigma$ there exists a periodic point $\omega^+\in\Sigma$ such that $A_{\omega^+}\neq 0$. Without loss of generality, we can assume that its period $2P+1$ satisfies $P\geq 20\log\rr$ (if not, just take a multiple of the period). 
Let us define
$\omega^-$ as the $(2P+1)$-periodic point such that in the period $[-P,P]$,
\[
 \omega^-_j=\omega^+_j \quad\text{ for }j=-P,\ldots P-1,\qquad  \omega^-_P\neq \omega^+_P.
\]
Since  $A$ is $\vartheta$-H\"older, taking $P$ large enough, $A_{\omega^-}\neq 0$.

Define the  operator $U_m$ by $U_mA_\omega=e^{\im m\beta_\omega} A_{\sigma\omega}$. Then, if $A_\omega$ is an eigenfunction of this operator of eigenvalue $\nu$, it also satisfies  that $U_m^{2P+1}A_\omega=\nu^{2P+1}A_\omega$. Evaluating this equality at $\omega=\omega^+,\omega^-$ and using $\sigma^{2P+1}\omega^+=\omega^+$, $\sigma^{2P+1}\omega^-=\omega^-$, we have
\[
 e^{\im m\sum_{j=0}^{2P}\beta_{\sigma^j\omega^+}}A_{\omega^+}=\nu^{2P+1}A_{\omega^+}\qquad \text{ and }\qquad  e^{\im m\sum_{j=0}^{2P}\beta_{\sigma^j\omega^-}}A_{\omega^-} =\nu^{2P+1}A_{\omega^-}.
\]
Since $A_{\omega^+}, A_{\omega^-}\neq 0$,
\[
 e^{\im m\sum_{j=0}^{2P}\beta_{\sigma^j\omega^+}}=\nu^{2P+1}=e^{\im m\sum_{j=0}^{2P}\beta_{\sigma^j\omega^-}},
\]
which is equivalent to
\begin{equation}\label{def:ContradWeakMixing}
 e^{\im m\sum_{j=0}^{2P}
 \left(\beta_{\sigma^j\omega^+}-\beta_{\sigma^j\omega^-}\right)}=1.
\end{equation}
We show that this is not possible. We split the sum in the exponent into two sums. The first one for $j\in [P-|\log\rho|, P+|\log\rho|]$ and the second one for the complement. For the first one, note that the number of terms in the sum is of order $\log\rho$. Then, using the expansion for $\beta_\omega$ given in Lemma \ref{lemma:laminationmap}, we have
\[
 \sum_{\substack{j=0\\j\in [P-|\log\rho|, P+|\log\rho|]}}^{2P}\left(\beta_{\sigma^j\omega^+}-\beta_{\sigma^j\omega^-}\right)= \sum_{\substack{j=0\\j\in [P-|\log\rho|, P+|\log\rho|]}}^{2P}\left(\beta^0_{\omega_j^+}-\beta^0_{\omega^-_j}\right)+\OO\left(\rr\log^2\rr\right).
\]
Since the two periodic orbits $\omega^+$, $\omega^-$ only differ in the $P$-symbol in the interval $[-P,P]$, one has
\begin{equation}\label{def:WeakMixingSum1}
\begin{split}
 \sum_{\substack{j=0\\j\in [P-|\log\rho|, P+|\log\rho|]}}^{2P}\left(\beta_{\sigma^j\omega^+}-\beta_{\sigma^j\omega^-}\right)&=\beta^0_{1}-\beta^0_{2}+\OO\left(\rr\log^2\rr\right)\\
 &=\al_1-\al_2+\OO\left(\rr\log^2\rr\right).
\end{split}
\end{equation}
Now we consider the second sum $j=0,\ldots ,2P$ with $j\not\in [P-|\log\rho|, P+|\log\rho|]$. Since $\omega^+$ and $\omega^-$ in the period $[-P,P]$ only differ by the symbol $P$, one has that
\[
\| \sigma^j\omega^+- \sigma^j\omega^-\|\lesssim \kk^{|P-j|},
\]
for some $\kk\in (0,1)$ and $j=0\ldots 2P-1$. Note that here the metric on $\Sigma$ is the one induced by the horseshoe obtained in Theorem \ref{thm:NHILCircular}. Then, since $\beta_\omega$ is H\"older continuous. One has that for  $j=0\ldots 2P-1$,
\[
| \beta_{\sigma^j\omega^+}-\beta_{ \sigma^j\omega^-}|\lesssim \kk^{|P-j|}.
\]
Thus, the second sum satisfies
\begin{equation}\label{def:WeakMixingSum2}
 \sum_{\substack{j=0\\j\not\in [P-|\log\rho|, P+|\log\rho|]}}^{2P}\left(\beta_{\sigma^j\omega^+}-\beta_{\sigma^j\omega^-}\right)\lesssim \kk^{|\log\rho|}\lesssim\rr^\zeta\qquad \text{ for some
}\quad  \zeta>0.
\end{equation}
Then, applying the estimates obtained in \eqref{def:WeakMixingSum1} and \eqref{def:WeakMixingSum2} to \eqref{def:ContradWeakMixing}, we obtain
\[
 1=e^{\im m\sum_{j=0}^{2P}
 \left(\beta_{\sigma^j\omega^+}-
 \beta_{\sigma^j\omega^-}\right)}=e^{\im m(\al_1-\al_2)}+\OO(\rho^\zeta).
\]
Taking $\rr>0$ small enough and using  condition \eqref{ansatz:omegadifferent} we reach a contradiction.
\end{proof}
%
%
%
We rely on Theorem \ref{thm:Dima} to perform a global normal form for
the lamination map $\FF_\omega$ in \eqref{def:stochasticmap}.
We
call circular iterates to the iterates of the map
\eqref{def:stochasticmap} with $e_0=0$, that is, to the dynamics of the
lamination of the RPC3BP. We denote them by
\begin{equation}\label{def:CircularOrbit}
(I_0,s^\ccirc_n)=\left.\FF_{\sigma^{n-1}\omega}
\circ\ldots\circ\FF_\omega\right|_{e_0=0}(I_0,s_0)=\left(I_0,s_0+\sum_{j=0}^{n-1}\beta_{
\sigma^j\omega}(I_0)\right)
\end{equation}
(recall that $I$ is a first integral for the circular problem). Even if not stated explicitly $s^\ccirc_n$ depends on $s_0$.

\begin{proposition}\label{prop:Cylindermaps:normalform}
Fix $\de>0$ arbitrarily small. There exists a symplectic change of coordinates
$\Phi:  [I_--\de, 
I_++\de]\times\TT\to
[I_-, 
I_+]\times\TT$, $(I,s)=\Phi(J,\theta)$, satisfying
\[
 \|\Phi-\Id\|_{C^1}=\OO(e_0),
\]
such  that the map
\[
 \wt \FF_\omega=\Phi^{-1}\circ \FF_\omega\circ\Phi
\]
is of the form
\begin{equation}\label{def:ModelCylinderMapAfterNF}
\wt \FF_\omega: \begin{pmatrix} J\\ \theta
\end{pmatrix}\to \begin{pmatrix}J+
e_0
A_{\omega}(J,\theta)+e_0^2
B_{\omega}(J,\theta)+\OO\left(e_0^{2+a}\right)\\
\theta+\beta_{\omega}(J)+e_0
D_{\omega}(J,\theta)+\OO\left(e_0^{1+a}\right)\end{pmatrix}.
\end{equation}
where the functions $A_{\omega}$ and $B_\omega$ satisfy $\NNN(A_{\omega})={\pm
1}$ and  $\NNN(B_{\omega})={0,\pm
2}$.

Furthermore,   $A_{\omega}$ satisfies
\begin{equation}\label{def:ZeroExpectation}
\sum_{n=0}^{\infty}\E_\omega\left(A_{\sigma^n\omega}\left(J,\theta_n^\ccirc\right)\right)=0,
\end{equation}
for any orbit of the form \eqref{def:CircularOrbit} and
is of the form
\[
A_{\omega}(J,\theta)=A_{\omega}^1(J) e^{\im \theta}+\ol A_{\omega}^1(J)
e^{-\im \theta}
\]
with
\[
A_{\omega}^1(J)=\MM^{I,1}_{\omega_0}(J)-\E_{\omega_0}\left(\MM^{I,1}_{\omega_0}(J)\right)\frac{1-e^{
\im \beta^0_{\omega_0}(J)}}{1-\E_{\omega_0}\left(e^{
\im \beta^0_{\omega_0}(J)}\right)}+\OO(\sqrt{\rr}).
\]
%
\end{proposition}

\begin{proof}
We perform a symplectic change of coordinates defined through a generating
function $\mathcal S(J,s)= sJ+e_0 S(J,s)$.
As usual, it induces a symplectic change of coordinates
$(I,s)=\Phi(J,\theta)$ which satisfies
\begin{eqnarray}\label{eq:NF:expansion}
\begin{aligned}
 \Phi(J,\theta)=
\begin{pmatrix}\theta-e_0\partial_{J}S(J,
\theta)+e_0^2\partial_\theta\partial_
{J}S(J,\theta)\partial_{J}S(J,\theta)+\mathcal{O}(e_0^3)\\
 J+e_0\partial_\theta
S(J,\theta)-e_0\partial_s^2S(J,\theta)\partial_{J}S(J,\theta)+\mathcal{O
}(e_0^3)
\end{pmatrix},
\end{aligned}\end{eqnarray}
and its inverse is given by
\begin{eqnarray}\label{Phiinverse}
\begin{aligned}
\Phi^{-1}(I,s)=
\begin{pmatrix}s+e_0\partial_{J}S(I,s)-e_0^2\partial^2_{J}S(I,
s)\partial_s
S(I,s)+\mathcal{O}(e_0^3)\\
I-e_0\partial_s S(I,s)+e_0^2\partial_s\partial_{J}S(I,s)\partial_s
S(I,s)+\mathcal{O}(e_0^3)
\end{pmatrix}.
\end{aligned}
\end{eqnarray}
Applying the change of coordinates to the maps \eqref{def:stochasticmap},
 we obtain the map
\[
\wt \FF_\omega: \begin{pmatrix} J\\ \theta\end{pmatrix}\to \begin{pmatrix} J+
e_0
A_{\omega}(J,\theta)+e_0^2 B_{\omega}(J,\theta)+\OO\left(e_0^{2+a}
\right)\\
\theta+
\beta_{\omega}(J)+e_0 D_{\omega}(J,\theta)+\OO\left(e_0^{1+a}
\right)\end{pmatrix},
\]
where
\[
 \begin{split}
 A_{\omega}(J,\theta)=&
\AAA_{\omega}(J,\theta) +\pa_s S(J,\theta)-\pa_s
S\left(J,\theta+\beta_{\omega}(J)\right)\\
B_{\omega}(J,\theta)=& \left(
\pa_s^2S\left(J,\theta+\beta_{\omega}(J)\right)-\pa_s^2S(J,
s)-\pa_s A_{\omega}(J,\theta)\right)\pa_J S(J,\theta)\\
&+\pa_s\pa_J S\left(J,\theta+\beta_{\omega}(J)\right)\left(
\pa_s S\left(J,s+\beta_{\omega}(J)\right)-\pa_s S(J,
\theta)- A_{\omega}(J,\theta)\right)\\
&+\pa_JA_{\omega}(J,\theta)\pa_s
S(J,s)+\mathcal B_{\omega}(J,\theta)\\
&+\pa_s^2S\left(J,\theta+\beta_{\omega}(J)
\right)\left(\pa_J\beta_{\omega}(J)\pa_s
S(J,\theta)+\mathcal D_{\omega}(J,\theta)\right).
 \end{split}
\]
First step to prove Proposition \ref{prop:Cylindermaps:normalform} is to choose
a suitable $S$ so that
\begin{equation}\label{cond:NormalForm}
\E_\omega \sum_{n=0}^{\infty} A_{\sigma^n\omega}(J,\theta_n^\ccirc)=\E_\omega
\sum_{n=0}^{\infty} \left(\mathcal  A_{\sigma^n\omega}(J,\theta_n^\ccirc)+\pa_s
S(J,\theta_n^\ccirc)-\pa_s
S(J,\theta_n^\ccirc+\beta_{\sigma^n\omega}(J))\right)=0
\end{equation}
This condition can be written as
\[
 \begin{split}
0=& \lim_{N\to+\infty}\left[
\sum_{n=0}^{N}  \E_\omega \mathcal
A_{\sigma^n\omega}(J,\theta_n^\ccirc)+\sum_{n=0}^{N} \E_\omega \pa_s
S(J,\theta_n^\ccirc)-\sum_{n=0}^{N} \E_\omega \pa_s
S(J,\theta_n^\ccirc+\beta_{\sigma^n\omega}(J))\right]\\
=& \lim_{N\to+\infty}\left[
\sum_{n=0}^{N}  \E_\omega \mathcal
A_{\sigma^n\omega}(J,\theta_n^\ccirc)+\sum_{n=0}^{N}  \E_\omega\pa_s
S(J,\theta_n^\ccirc)-\sum_{n=0}^{N}  \E_\omega\pa_s
S(J,\theta_{n+1}^\ccirc)\right]\\
=& \lim_{N\to+\infty}\left[
\sum_{n=0}^{N}  \E_\omega \mathcal
A_{\sigma^n\omega}(J,\theta_n^\ccirc)+ \pa_s
S(J,\theta_0)- \E_\omega\pa_s
S(J,\theta_{N+1}^\ccirc)\right].
\end{split}
\]
By  Theorem \ref{thm:Dima:FixedAngle}, which can be applied since we are assuming \eqref{ansatz:omegadifferent} (see Lemma \ref{lemma:WeakMixing}), the sum in the last line is convergent. Take $S$ such that
\[
\pa_s
S(J,\theta_0)=-\sum_{n=0}^{+\infty}\E_\omega
\mathcal{A}_{\sigma^n\omega}(J,\theta_n^\ccirc).
\]
Since $\NNN(\mathcal A_{\omega})=\{\pm 1\}$ (see Lemma \ref{lemma:laminationmap}), one has that $\mathcal A_{\omega}\in \EE_{1,\vartheta}$ and therefore the same happens for $\pa_sS$. Therefore, one can easily integrate to obtain $S$ which,
applying again Theorem \ref{thm:Dima:FixedAngle}, satisfies
\[
  \lim_{N\to+\infty}  \E_\omega\ \pa_s
S(J,\theta_{N+1}^\ccirc)=0.
\]
Therefore,  this choice of $S$ solves \eqref{cond:NormalForm}.

Since we can write $\mathcal A_{\omega}$ as
$\mathcal A_{\omega}(J,\theta)=\mathcal  A^1_{\omega}(J)e^{\im\theta}+
 \ol{\mathcal{A}^1_{\omega}}(J)e^{-\im\theta}$, the generating function  $S$ is just given by
 \[
S(J,\theta)=S^1(J)e^{\im\theta}+\ol{S^1}(J)e^{-\im\theta}
\]
with
\begin{equation}\label{def:S1}
 S^1(J)=\sum_{n=0}^{+\infty}\E_\omega\left[
\mathcal
A_{\sigma^n\omega}(J)e^{\im\sum_{k=0}^{n-1}\beta_{\sigma^k\omega}(J)}\right].
\end{equation}

%
With this choice of $S_1$, the transformed map satisfies the property
\eqref{def:ZeroExpectation} and the fact that $\NNN(B_\omega)=\{0,\pm 2\}$.

Now it only remains to compute the first order in $\rho$ of the
function $A_\omega$. By Corollary \ref{coro:decaycorrelations}, the terms in the sum in
\eqref{def:S1} decay exponentially. Then, one can truncate the sum at
$N^*=C\log\rr$ with a suitable $C$ such that
\[
  S^1(J)=\sum_{n=0}^{N^*}\E_\omega\left[
\mathcal
A_{\sigma^n\omega}(J)e^{i\sum_{k=0}^{n-1}\beta_{\sigma^k\omega}(J)}
\right]+\OO(\sqrt { \rr})
\]
Then, using the expansions of $\mathcal A_{\omega}$ and $\beta_{\omega}$ in Lemma
\ref{lemma:laminationmap}, one obtains
\[
  S^1(J)=\sum_{n=0}^{N^*}\E_\omega \left[
\MM^{I,1}_{\omega_n}(J)e^{\im\sum_{k=0}^{n-1}\beta^0_{\omega_k}(J)}\right]
+\OO(\sqrt {
\rr}).
\]
Note that now each term involved in the expectation depends only on a
different symbol $\omega_k$ in $\omega$, which are independent. Therefore, by Lemma \ref{lemma:laminationmap} and using
that \eqref{ansatz:omegadifferent} is satisfied uniformly in $\rr$,
\[
\begin{split}
 S^1(J)&=\E_{\omega_0} \left[\MM^{I,1}_{\omega_0}(J)\right]\sum_{n=0}^{N^*} \E_{\omega_0}\left(e^{
\im\beta^0_{\omega_0}(J)}\right)^n
+\OO(\sqrt{\rr})\\
&=\E_{\omega_0}
\left[\MM^{I,1}_{\omega_0}(J)\right]\frac{1}{1-e^{\im\nu(J)\ol g}\E_{\omega_0}\left(e^{\im\alpha_{\omega_0}
(J)}\right)} +\OO(\sqrt{\rr}).
\end{split}
\]
This completes the proof of Proposition \ref{prop:Cylindermaps:normalform}.
\end{proof}

\subsection{The martingale analysis}\label{sec:Martingale}
We analyze the stochastic limit of the random iteration of the map \eqref{def:ModelCylinderMapAfterNF} at time
scales $n\sim  e_0^{-2}$. To this end, we use a discrete version of the scheme 
by Friedlin and Wentzell \cite{FreidlinW12}, which gives a sufficient condition for  weak convergence to a diffusion 
process at scales $e_0^{-2}$ as $e_0\to 0$ in terms of the associated martingale problem.  Namely, $\JJ_s$ 
satisfies a diffusion process with drift $\bbb(\JJ_s)$ and variance $\bms^2(\JJ_s)$ provided that for any $\kk>0$, 
any time $n\leq \kk e_0^{-2}$, any $C^3$ function $f$ (with $C^3$
norm with an $e_0$-independent upper bound) and   any $(J_0,\theta_0)$, one has that
\begin{equation}\label{def:ExpectLemma}
\E_\omega\left(f(J_n)-e_0^2\sum_{k=0}^{n-1}\left(\bbb(J_k)f'(J_k)+\frac{\bms^2(J_k)}{2}
f''(J_k) \right)\right) -f(J_0)\xrightarrow[e_0\to 0]{} 0.
\end{equation}
To prove such statement we split the cylinder $[I_--\de,
I_++\de]\times\TT$ in smaller
strips. Take
\begin{equation}\label{def:gamma}
 \ga\in \left(\frac{7}{8},1\right).
\end{equation}
We consider an $e_0^\ga$  grid in the interval  $[ I_--\de,
I_++\de]$. Then, we denote by
$I_\ga$ any of the segments whose endpoints are consecutive points in the grid.

We describe the behavior of the random iteration of the  map \eqref{def:ModelCylinderMapAfterNF} in each of these
strips $I_\gamma\times\TT$. We consider the (random) exit times from each strip
\[
 0=n_0<n_1<n_2<\ldots <n_m<n\leq \kk e_0^{-2}
\]
for some random $m=m(\omega)$. Almost surely $m(\omega)$ is finite. Then, the
first step to prove \eqref{def:ExpectLemma}, is to analyze such exit times and
to descompose \eqref{def:ExpectLemma} into the summands
\begin{equation}\label{def:ExpectLemma:SmallStrip}
\E_\omega\left(\sum_{\ell=0}^m\left[f(J_{n_{\ell+1}})-f(J_{n_{\ell}})-e_0^2\sum_{
k=n_\ell }^{n_{\ell+1}-1} \left(\bbb(J_k)f'(J_k)+\frac{\bms^2(J_k)}{2}f''(J_k)
\right)\right]\right)\xrightarrow[e_0\to 0]{} 0
\end{equation}
and analyze each of them.
%
%
%
Before going into this analysis, we introduce the  drift and variance associated to the diffusion process, which are defined in terms of the functions appearing in the map \eqref{def:ModelCylinderMapAfterNF}.

	
	\begin{lemma}\label{lemma:DriftVar}
		The functions
		\begin{equation}\label{def:drift}
			\bbb(J)=\bbb_1(J)+\bbb_2(J)+\bbb_3(J)
		\end{equation}
		with
		\[
		\begin{split}
			\bbb_1(J)=&\,\E_{\omega,\theta} B_{\omega}\\
			\bbb_2(J)=&\, 2\sum_{k=1}^{+\infty}\Re\E_\omega\left[\pa_J A^+_{\sigma^k\omega}\ol
			A^+_\omega e^{\im\sum_{j=0}^k \beta_{\sigma^j\omega}(J)}\right]\\
			\bbb_3(J)=&\, 2\sum_{k=1}^{+\infty}\Re\E_\omega\left[\pa_\theta A^+_{\sigma^k\omega}\ol
			D^+_\omega
			e^{\im\sum_{j=0}^k \beta_{\sigma^j\omega}(J)}\right]\\
			&\, + 2\Re\E_\omega\left[\sum_{i=2}^{+\infty}\sum_{j=1}^{i-1}\pa_\theta
			A^+_{\sigma^i\omega}(J)\beta'_{\sigma^j\omega}(J)\ol A^-_{\omega}(J)
			e^{\im\sum_{m=0}^i\beta_{\sigma^m\omega}(J)}\right]
		\end{split}
		\]
		and
		\begin{equation}\label{def:variance}
			\bms^2(J)=\bms_1^2(J)+\bms_2^2(J),
		\end{equation}
		with
		\[
		\begin{split}
			\bms_1^2(J)&=\E_{\omega,\theta}
			\left(A_{\omega}\right)^2\\
			\bms_2^2(J)& =4\sum_{k=1}^{+\infty}\Re\E_\omega\left[A^+_{\sigma^k\omega}\ol
			A^+_\omega
			e^{\im\sum_{j=0}^k \beta_{\sigma^j\omega}(J)}\right]
		\end{split}
		\]
		where $A_\omega^+, \ol A_\omega^+$ are the $\pm 1$ Fourier coefficients of the function
		$A_\omega$ (and the same for $D_\omega$), are well defined.
		
		Moreover, they satisfy
		\begin{equation}\label{def:DriftVarFirstOrder}
			\begin{split}
				\bbb(J)&=\E_{\omega,\theta} B^0_{\omega_0}(J)+\OO\left(\sqrt{\rr}\right)\\
				\bms^2(J)&=\bms_0^2(J)+\OO\left(\sqrt{\rr}\right),
			\end{split}
		\end{equation}
		with
		\begin{equation}\label{def:DriftVarFirstOrder:1}
			\bms_0^2(J)=2\E_\omega\left|\MM^{I,1}_{\omega_0}(J)-\E_\omega\MM^{I,1}_{\omega_0}(J)\frac{1-e^{
					\im\beta^0_{\omega_0}}(J)}{1-\E_\omega\left(e^{
					\im\beta^0_{\omega_0}(J)}\right)}\right|^2.
		\end{equation}
	\end{lemma}
	Note that in the formula of $b_3$ we are abusing notation by writing $\pa_\theta A^+_{\sigma^k\omega}$, which denotes the Fourier $1$ coefficient of the function $\pa_\theta A_{\sigma^k\omega}$. We keep this notation since makes easier to follow the proof of Lemma \ref{lemma:expectationlemsubstrip}. 
	
	Recall that by Ansatz \ref{ansatz:Melnikov:1}, the function $\bms_0^2(J)$ 
	introduced in this lemma satisfies that for $J\in [I_--\de, 
	I_++\de]$,
	\begin{equation}\label{def:variance:bound}
		\inf_{J\in[I_--\de, 
			I_++\de]}\bms_0^2(J)>0,
	\end{equation}
	and, therefore, the same happens for $\bms^2(J)$  provided $\rr$ is small enough,
	\begin{proof}[Proof of Lemma \ref{lemma:DriftVar}]
		We show that the series involved in the definition of $\bbb_3$ are absolutely convergent. One can proceed for the other sums analogously.
		It is enough to use Corollary \ref{coro:decaycorrelations}, which can be applied by Lemma \ref{lemma:WeakMixing}. Indeed, for the first sum
		\[
		\left|\sum_{k=1}^{+\infty}\Re\E_\omega\left[\pa_\theta A^+_{\sigma^k\omega}\ol
		D^+_\omega
		e^{\im\sum_{j=0}^k \beta_{\sigma^j\omega}}\right]\right|\lesssim \sum_{k=1}^{+\infty}\|\pa_\theta A^+_\omega\|_{\vartheta^{1/2}}\|D^+_\omega\|_{\vartheta^{1/2}} r^k<\infty.
		\]
		For the second sum, by Lemma \ref{lemma:tripledecay}, one can proceed analogously,
		\[
		\begin{aligned}
			\MoveEqLeft\left|\Re\sum_{i=2}^{+\infty}\E_\omega\left[\sum_{j=1}^{i-1}\pa_\theta
			A^+_{\sigma^i\omega}\pa_J\beta_{\sigma^j\omega}\ol A^-_{\omega}
			e^{\im\sum_{m=0}^i\beta_{\sigma^m\omega}(J)}\right]\right|\\
			&\lesssim
			\sum_{i=2}^{+\infty}\left\|\pa_\theta
			A^+_\omega\right\|_{\vartheta^{1/2}} \|\pa_J\beta_\omega\|_{\vartheta^{1/2}}\|\ol A^-_{\omega}\|_{\vartheta^{1/2}}  i r^i<\infty.
		\end{aligned}
		\]
		The formulas for the expansion in $\rr$ can be obtained as in the proof of Proposition \ref{prop:Cylindermaps:normalform}.
	\end{proof}

\subsubsection{Martingale analysis in the strips
$I_\ga\times\TT$}\label{sec:LocalMartingale}

Consider the $n$-iteration of the map  \eqref{def:ModelCylinderMapAfterNF}. Taking into account  \eqref{def:CircularOrbit}, it is of the form
\begin{equation}\label{def:RandomMap:niterates}
\begin{split}
 J_n&=J_0+e_0 \sum_{k=0}^{n-1}
A_{\sigma^k\omega}(J_k,\theta_k)+e_0^2\sum_{k=0}^{n-1}
B_{\sigma^k\omega}(J_k,\theta_k)+\OO(e_0^{2+a}n)\\
 \theta_n&=\theta_n^\ccirc+\OO(e_0 n^2).
\end{split}
\end{equation}
We fix a strip $I_\ga\times\TT$ and let $(J_0,\theta_0)\in I_\ga\times\TT$. We define $n_\ga\leq
n\leq \kk e_0^{-2}$, which is either the exit time from the strip $I_\ga$, that
is the smallest number such that $(J_{n_\ga+1},\theta_{n_\ga+1})\not\in
I_\ga\times\TT$, or
$n_\ga=n$ is the final time.

\begin{proposition}\label{lemma:exittime:Igamma}
Fix $\ga\in (7/8,1)$. Then, there
exists a constant  $C>0$ such that,
\begin{itemize}
\item For any $\de\in (0, 1-\ga)$ and $e_0>0$ small
enough,
\[
\Prob\{n_\ga<e_0^{-2(1-\ga)+\delta}\}\leq
\exp\left(-\frac{C}{e_0^{\delta}}\right).
\]
\item For any   $\de>0$ and $e_0>0$ small
enough,
\[
\Prob\{e_0^{-2(1-\ga)-\delta}<n_\ga<\kk e_0^{-2}\}\leq
\exp\left(-\frac{C}{e_0^{\delta}}\right).
\]
\end{itemize}
\end{proposition}

This proposition is proven in Appendix \ref{app:exittime}.
We use these estimates on the probability of the exit time to prove
\eqref{def:ExpectLemma:SmallStrip} for the piece  of the orbit
belonging to the strip $I_\ga$.
%
%
Recall that we use the drift $\bbb$ and variance $\bms^2$ introduced in Lemma \ref{lemma:DriftVar}.

\begin{lemma}\label{lemma:expectationlemsubstrip}
Fix $\ga\in (7/8,1)$ and let
$f:\mathbb{R}\rightarrow\mathbb{R}$ be any $\mathcal{C}^3$  with $\|f\|_{C^3}\leq C$ for some constant $C>0$ independent of $e_0$.
Then, there exists $\zeta>0$ such that
\begin{equation*}
\E_\omega\Bigg(f(J_{n_\ga})- e_0^2
\sum_{k=0}^{n_\ga-1}\left(\bbb(J_k)f'(J_k)+\frac{\bms^2(J_k)}{2}
f''(J_k)\right) \Bigg)- f(J_0)=\mathcal{O}(e_0^{2\ga+\zeta}).
\end{equation*}
\end{lemma}

\begin{proof}
 Let us denote
\begin{equation}\label{def:eta}
\eta=f(J_{n_\ga})-e_0^2
\sum_{k=0}^{n_\ga-1}\left(\bbb(J_k)f'(J_k)+\frac{\bms^2(J_k)}{2}
f''(J_k)\right).
\end{equation}
Writing
\[f(J_{n_\ga})=f(J_0)+
\sum_{k=0}^{n_\ga-1}\left(f(J_{k+1})-f(J_k)\right)\]
and doing the Taylor expansion in each term inside the sum. one obtains
\begin{equation*}
f(J_{n_\ga})=f(J_0)+\sum_{k=0}^{n_\ga-1}\Big[
f'(J_k)(J_{k+1}-J_k)
+\frac{1}{2}f''(J_k)(J_{k+1}-J_k)^2+\mathcal{O}(e_0^3)\Big].
\end{equation*}
Substituting this in \eqref{def:eta}, we obtain
\begin{equation}\label{eta-version2}
 \begin{split}
 \eta=&f(J_0)+\sum_{k=0}^{n_\ga-1}
\Big[f'(J_k)(J_{k+1}-J_k)+
\frac{1}{2}f''(J_k)(J_{k+1}-J_k)^2\Big]\\
&-e_0^2\sum_{k=0}^{n_\ga-1}
\left[\bbb(J_k)f'(J_k)+\frac{\bms^2(J_k)}{2}
f''(J_k)\right ] +\sum_{k=0}^{n_\ga-1}\mathcal{O}\left(e_0^3\right).
\end{split}
\end{equation}
Now, using \eqref{def:RandomMap:niterates}, we can write
\[
\begin{split}
J_{k+1}-J_k&=e_0 A_{\sigma^k\omega}(J_k,\theta_k)+e_0^2
B_{\sigma^k\omega}(J_k,\theta_k)
+\mathcal{O}(e_0^{2+a})\\
(J_{k+1}-J_k)^2&=e_0^2
A^2_{\sigma^k\omega}(J_k,\theta_k)+\mathcal{O}(e_0^3).
\end{split}
\]
Thus, using the formulas for the drift and variance in \eqref{def:drift} and
\eqref{def:variance},  \eqref{eta-version2} can be written as
\begin{equation}\label{eta-version3}
  \eta=f(J_0)+P_1+P_2
 \end{equation}
with
\[
\begin{split}
P_1= &e_0\sum_{k=0}^{n_\ga-1}\left[
f'(J_k)\left( A_{\sigma^k\omega}(J_k,\theta_k)-e_0 \bbb_2(J_k)-e_0
\bbb_3(J_k)\right)-e_0 f''(J_k)\bms_2^2(J_k)\right]\\
 P_2 =&e_0^2\sum_{k=0}^{n_\ga-1}
 f'(J_k)\left[B_{\sigma^k\omega}(J_k,\theta_k)-\bbb_1(J_k)\right]
 +\frac{e_0^2}{2}\sum_{k=0}^{n_\ga-1}f''(J_k)
 \left[ A^2_{\sigma^k\omega}(J_k,\theta_k)-\bms_1^2(J_k)\right]\\
  &+\sum_{k=0}^{n_\ga-1}\mathcal{O}(e_0^{2+a}).
\end{split}
\]
Now we estimate the expectations of $P_1$ and $P_2$. We start with $P_2$. Using
\eqref{def:RandomMap:niterates}, we obtain that $P_2=P_{21}+P_{22}$ with
\begin{equation}\label{def:P21P22}
\begin{split}
P_{21} =&\,e_0^2\sum_{k=0}^{n_\ga-1}
f'(J_0)\left[B_{\sigma^k\omega}(J_0,\theta_k^\ccirc)-\bbb_1(J_0)\right]
+\frac{e_0^2}{2}\sum_{k=0}^{n_\ga-1}f''(J_0)
\left[ A^2_{\sigma^k\omega}(J_0,\theta_k^\ccirc)-\bms_1^2(J_0)\right]\\
P_{22}=
&\,\mathcal{O}\left(e_0^{2+a}n_\ga\right)+\mathcal{O}\left(e_0^{3}
n_\ga^3\right).
\end{split}
\end{equation}
For $P_{22}$, we condition the expectation as
\begin{equation}\label{def:Expectation:condition}
\begin{split}
\E_\omega\left(P_{22}\right)\,=&\,\E_\omega\left(P_{22}\,|\,e_0^{-2(1-\ga)+\delta}\leq
n_\ga\leq e_0^{-2(1-\ga)-\delta}\right)
\Prob\left\{e_0^{-2(1-\ga)+\delta}\leq n_\ga\leq
e_0^{-2(1-\ga)-\delta}\right\}\\
&\,+\E_\omega\left(P_{22}\,|\,n_\ga< e_0^{-2(1-\ga)+\delta}\right)
\Prob\left\{n_\ga<e_0^{-2(1-\ga)+\delta}\right\}\\
&\,+\E_\omega\left(P_{22}\,|\,e_0^{-2(1-\ga)-\delta}<n_\ga\leq \kk e_0^{-2}\right)
\Prob\left\{e_0^{-2(1-\ga)-\delta}<n_\ga\leq \kk e_0^{-2}\right\}.
\end{split}
\end{equation}
In the first row, we use that by \eqref{def:gamma},
\[
\left| \E_\omega\left(P_{22}\,|\,e_0^{-2(1-\ga)+\delta}\leq
n_\ga\leq e_0^{-2(1-\ga)-\delta}\right)\right|\leq
\OO\left(e_0^{2\ga+\zeta}\right)
\]
for any $\zeta\leq \min\{a-\de, 8\ga-3-3\de\}>0$ (recall that $\ga>7/8$, see \eqref{def:gamma}).

For the third row, since $n_\ga\leq \kk e_0^{-2}$, one has
\[
 \left|\E_\omega\left(P_{22}\,|\,e_0^{-2(1-\ga)-\delta}<n_\ga\leq
\kk e_0^{-2}\right)\right|\leq\OO\left(e_0^{-3}\right).
\]
Now, using the second statement of Lemma \ref{lemma:exittime:Igamma}, we obtain
\[
\left|\E_\omega\left(P_{22}\,|\,e_0^{-2(1-\ga)-\delta}<n_\ga\leq
\kk e_0^{-2}\right)\right|
\Prob\left\{e_0^{-2(1-\ga)-\delta}<n_\ga\leq \kk e_0^{-2}\right\}\leq
\OO\left(e_0^{2\ga+\zeta}\right).
\]
Proceeding analogously, one can bound the second row in
\eqref{def:Expectation:condition}.

Now we bound the expectation of $P_{21}$ in \eqref{def:P21P22}. We estimate the
first sum. The second one can be done analogously.

Recall that $\NNN(B_{\omega})=\{ 0,\pm 2\}$. We can compute the
expectation for
each harmonic in $\theta_0$. For the harmonic zero, it is enough to use the
definition of the
drift in \eqref{def:drift} and the fact that
\[
 \E_\omega\left(\langle B_{\sigma^k\omega}\rangle_\theta(J_0)\right)= \E_\omega\left(\langle
B_{\omega}\rangle_\theta(J_0)\right)=\bbb_1(J_0).
\]
For the other harmonics we use Corollary \ref{coro:decaycorrelations}. Indeed, the function
\[B_{\sigma^k\omega}(J_0,\theta)-\langle
B_{\sigma^k\omega}\rangle_\theta(J_0)\]
has zero $\theta$-average and has only harmonics $\pm 2$. Namely,
\[
\E_\omega \left(B_{\sigma^k\omega}(J_0,\theta_k^\ccirc)-\bbb_1(J_0)\right)= \E_\omega \left(B^2_{\sigma^k\omega}e^{2\im\sum_{j=0}^{k-1}\beta_{\sigma^j\omega}}\right)e^{2\im\theta_0}+ \E_\omega \left(B^{-2}_{\sigma^k\omega}e^{-2\im\sum_{j=0}^{k-1}\beta_{\sigma^j\omega}}\right)e^{-2\im\theta_0},
\]
where the superindex $\pm 2$ denotes the second Fourier coefficient of the function $B$. By Corollary~\ref{coro:decaycorrelations}, we have
\[
\left|\E_\omega\left( B^{\pm 2}_{\sigma^k\omega}e^{\pm 2\im\sum_{j=0}^{k-1}\beta_{\sigma^j\omega}}\right)\right|
\lesssim \|B\|_{\vartheta^{1/2}} r^k.
\]
Thus,
\[
\left|e_0^2\E_\omega\left(\sum_{k=0}^{n_\ga-1}
f'(J_0)\left[B_{\sigma^k\omega}(J_0,\theta_k^\ccirc)-\langle
B_{\sigma^k\omega}\rangle(J_0)\right]\right)\right|\lesssim
e_0^2\|f\|_{C^1}\|B\|_{\vartheta^{1/2}}\sum_{k=0}^{n_\ga-1}r^k\lesssim
e_0^{2\ga+\zeta}
\]
taking $0<\zeta< 2-2\ga$.

Therefore, it only remains to bound the expectation of $P_1$. To this end, we
expand $P_1$ in powers of $e_0$ up to order 2. To do so, we need more
accurate expansions of the $n$-iterates than those in
\eqref{def:RandomMap:niterates}. Expanding the map  \eqref{def:ModelCylinderMapAfterNF} in powers of $e_0$ (see also \eqref{def:CircularOrbit}), one can easily see that 
\[
\begin{split}
 J_n=&\,J_0+e_0
\sum_{k=0}^{n-1}A_{\sigma^k\omega}\left(J_0,
\theta_k^\ccirc\right)+\OO\left(e_0^2n^3\right)\\
\theta_n=&\,\theta_n^\ccirc+e_0\sum_{k=0}^{n-1}\left[\beta_{\sigma^k\omega}
'(J_0)\sum_{j=0}^{k-1}A_{\sigma^j\omega}\left(J_0,
\theta_j^\ccirc\right)+D_{\sigma^k\omega}\left(J_0,
\theta_k^\ccirc\right)\right ]+\OO\left(e_0^2n^4\right).
\end{split}
\]
Therefore, $P_1=P_{11}+P_{12}+P_{13}+P_{14}$ as
\begin{equation}\label{def:P1Sum}
\begin{split}
P_{11}=&\,e_0\sum_{k=0}^{n_\ga-1}
f'(J_0) A_{\sigma^k\omega}(J_0,\theta_k^\ccirc)\\
P_{12}=&\,e^2_0\sum_{k=0}^{n_\ga-1}f'(J_0) \left[
\pa_J A_{\sigma^k\omega}(J_0,\theta_k^\ccirc)\left(\sum_{j=0}^{
k-1 }A_{\sigma^j\omega}\left(J_0,
\theta_j^\ccirc\right)\right)-\bbb_2(J_0)\right]\\
P_{13}=&\,e^2_0\sum_{k=0}^{n_\ga-1}\frac{1}{2}f''(J) \left[2
A_{\sigma^k\omega}(J_0,\theta_k^\ccirc)\left(\sum_{j=0}^{k-1
}A_{\sigma^j\omega}\left(J_0,
\theta_j^\ccirc\right)\right)-\bms_2^2(J_0)\right]\\
P_{14}=&\,e^2_0\sum_{k=0}^{n_\ga-1}f'(J_0) \Bigg[
\pa_\theta
A_{\sigma^k\omega}(J_0,\theta_k^\ccirc)\left(\sum_{j=0}^{k-1}\left[\beta_{
\sigma^j\omega}
'(J_0)\sum_{i=0}^{j-1}A_{\sigma^i\omega}\left(J_0,
\theta_i^\ccirc\right)+D_{\sigma^j\omega}\left(J_0,
\theta_j^\ccirc\right)\right ]\right)\\
&\qquad \qquad \qquad \quad -\bbb_3(J_0)\Bigg]\\
P_{15}=&\,\OO\left(e_0^3n_\ga^5\right).
\end{split}
\end{equation}
We bound the expectation of each of these terms. Proceeding as before we can
condition the expectation to $e_0^{-2(1-\ga)+\delta}\leq
n_\ga\leq e_0^{-2(1-\ga)-\delta}$. Indeed the probability of the complement is
exponentially small and therefore, the corresponding terms in the conditioned
expectation are smaller than $\OO(e_0^{2\ga+\zeta})$.

Therefore, it only remains to bound the expectation of each row in
\eqref{def:P1Sum} assuming that
\begin{equation}\label{def:timeestimate}
e_0^{-2(1-\ga)+\delta}\leq
n_\ga\leq e_0^{-2(1-\ga)-\delta}.
\end{equation}

For $P_{11}$ in \eqref{def:P1Sum}, we use that,
by Proposition \ref{prop:Cylindermaps:normalform}, $A_\omega$ satisfies
\eqref{def:ZeroExpectation}. Then, by the first statement of Theorem \ref{thm:Dima:FixedAngle}, the
expectation of the first row satisfies
\[
 \left|\E_\omega P_{11}\right|\leq\sum_{k=n_\ga}^{+\infty}
\left|f'(J_0)\E_\omega\left(
A_{\sigma^k\omega}(\theta_k^\ccirc,J_0)\right)\right|\lesssim r^{n_\ga}
\]
for some $r$ independent of $e_0$. Since $n_\ga$ satisfies
\eqref{def:timeestimate}, this expectation is exponentially small with respect
to $e_0$.

For $P_{15}$ it is enough to use \eqref{def:timeestimate} and
\eqref{def:gamma} to obtain that, taking $\de>0$ small enough,
\[
 \left|\E_\omega P_{15}\right|\lesssim e_0^{-7+10\ga-5\de}\lesssim
e_0^{2\ga+\zeta}
\]
taking $\zeta<8\ga-7-5\de$. Note that taking $\de$ small enough, $8\ga-7-5\de>0$ by \eqref{def:gamma}.

Now it only remains to estimate the expectation of $P_{12}$, $P_{13}$ and
$P_{14}$. All these bounds can be done analogously. We estimate just the first
one.  Since $\NNN(A_\omega)=\{\pm 1\}$, $P_{12}$ satisfies
$\NNN(P_{12})=\{0,\pm 2\}$. We treat each harmonic $P_{12}^l$, $l=0,\pm 2$,
separately.

For $l=0$,
\[
\begin{split}
 P_{12}^0=&\,e^2_0f'(J_0)\sum_{k=0}^{n_\ga-1}\left(\sum_{j=0}^{k-1}\Re \left[
\pa_J A^{+}_{\sigma^k\omega}(J_0)e^{\im\theta_k^\ccirc}\ol
A^{+}_{\sigma^j\omega}(J_0)
e^{-\im\theta_j^\ccirc}\right]-\bbb_2(J_0)\right)\\
=&\,e^2_0f'(J_0)\sum_{k=0}^{n_\ga-1}\left(\sum_{j=0}^{k-1}\Re \left[
\pa_J A^{+}_{\sigma^k\omega}(J_0)\ol
A^{+}_{\sigma^j\omega}(J_0)
e^{\im\sum_{m=j}^{k-1}\beta_{\sigma^m\omega}(J_0)}\right]-\bbb_2(J_0)\right).
\end{split}
\]
We bound the expectation of each of the summands on $k$
\[
 \PP_k=\sum_{j=0}^{k-1}\Re \left[
\pa_J A^{+}_{\sigma^k\omega}(J_0)\ol
A^{+}_{\sigma^j\omega}(J_0)
e^{\im\sum_{m=j}^{k-1}\beta_{\sigma^m\omega}(J_0)}\right]-\bbb_2(J_0).
\]
Note that, since the
expectation is over $\omega$, it can be replaced by $\sigma^{-j}\omega$. Then,
\[
 \begin{split}
\E_\omega\PP_k&= \E_\omega \left(\sum_{j=0}^{k-1}\Re \left[
\pa_J A^{+}_{\sigma^{k-j}\omega}(J_0)\ol
A^{+}_{\omega}(J_0)
e^{\im\sum_{m=0}^{k-j-1}\beta_{\sigma^m\omega}(J_0)}\right]-\bbb_2(J_0)\right)\\
&=  \sum_{j=1}^{k} \Re \E_\omega\left[
\pa_J A^{+}_{\sigma^{j}\omega}(J_0)\ol
A^{+}_{\omega}(J_0)
e^{\im\sum_{m=0}^{j-1}\beta_{\sigma^m\omega}(J_0)}\right]-\bbb_2(J_0).
 \end{split}
\]
Now, by the definition of $\bbb_2$ in \eqref{def:drift} and proceeding as in the proof of Lemma \ref{lemma:DriftVar}, we have that
\[
 \left|\E_\omega\PP_k\right|\lesssim r^k,
\]
for some $r\in (0,1)$ independent of $e_0$. Thus, we can conclude that
\[
 \left|\E_\omega P_{12}^0\right|\lesssim e_0^2\lesssim e_0^{2\ga+\zeta}.
\]
Now we estimate the expectation of $P_{12}^2$, the second harmonic of $P_{12}$.
Proceeding analgously one can estimate the harmonic $-2$. $P_{12}^2$ can be
written as
\[
 P_{12}^2=e^2_0f'(J_0)\sum_{k=0}^{n_\ga-1}\left(\sum_{j=0}^{k-1}\Re \left[
\pa_J A^{+}_{\sigma^k\omega}(J_0)\ol
A^{+}_{\sigma^j\omega}(J_0)
e^{2\im\sum_{m=j}^{k-1}\beta_{\sigma^m\omega}(J_0)}e^{\im\sum_{m=0}^{j-1}\beta_{
\sigma^m\omega}(J_0)}\right]\right)
\]
Then, applying the second statement in  Theorem  \ref{thm:Dima:FixedAngle},  the expectation of each term in the sum satisfies
\[
\left|\E_\omega \left( \pa_J A^{+}_{\sigma^k\omega}(J_0)\ol
A^{+}_{\sigma^j\omega}(J_0)
e^{2\im\sum_{m=j}^{k-1}\beta_{\sigma^m\omega}(J_0)}e^{\im\sum_{m=0}^{j-1}\beta_{
\sigma^m\omega}(J_0)}\right)\right|
\leq \left\| \pa_J A^{+}_{\omega}\right\|_\vartheta \left\|
 A^{+}_{\omega} \right\|_\vartheta r^{k}.
\]
Thus, taking $\zeta<2-2\ga$,
\[
\left| \E_\omega\left( P_{12}^2\right)\right|\lesssim
e^2_0\sum_{k=0}^{n_\ga-1}\sum_{j=0}^{k-1} r^k\lesssim e_0^2\lesssim
e_0^{2\ga+\zeta}.
\]
Therefore, one can conclude that
\[
 |\E_\omega P_{12}|\lesssim e_0^{2\ga+\zeta}.
\]
Proceeding analgously, one obtains the same bounds for the expectations of
$P_{13}$. The estimates of  the expectation of $P_{14}$ can be obtained in the same way using Lemma \ref{lemma:tripledecay}. This completes the proof of the lemma.


\end{proof}

\subsubsection{From local to global: End of the proof of Theorem \ref{thm:main-thm}}\label{sec:LocalToGlobal}
The last step is to use the local analysis of the martingale done in Lemma
\ref{lemma:expectationlemsubstrip} to prove its global version
\eqref{def:ExpectLemma}. To this end we need to analyze how orbits visit the
different strips $I_\ga^j\times\TT$
 in the interval 
 \[J\in[J_-,J_+]=[ I_--\de,
 I_++\de]
 \]
 (see Proposition \ref{prop:Cylindermaps:normalform}). This interval  is the union of  $e_0^{-\ga}$ strips. By a translation, assume that the initial action $J_0$ is $J_0=0\in [J_-,J_+]$. To
model the visits to the different strips  by a symmetric random walk, we slightly modify the widths of the strips considered in Section \ref{sec:LocalMartingale}.
We consider endpoints of the strips
\[
 J_j=\tA_j e_0^\ga,\quad j\in\ZZ
\]
with some (to be determined) constants $\tA_j$ independent of $e_0$ to the leading order and satisfying $\tA_0=0$, $\tA_1=\tA>0$
and
$\tA_j<\tA_{j+1}$ for $j>0$ (and similarly for $j<0$).  We consider the strips
\[
 I_\ga^j=[J_j,J_{j+1}]=[\tA_j\,e_0^\ga, \tA_{j+1}\,e_0^\ga].
\]
To analyze the visits to these strips, we consider the lattice of points
$\{J_j\}_{j\in\ZZ}\subset\RR$ and we analyze the ``visits'' to these points. By
visit we mean the existence of an iterate $\OO(e_0)$-close to it.  Lemma~\ref{lemma:exittime:Igamma} implies that
if we start with $J=J_j$ we hit either $J_{j-1}$ or
$J_{j+1}$  with probability one (up to exponentially small terms). This process  can be treated as a random walk
for $j\in\ZZ$ (stopping if the orbit reaches the endpoints of $[J_-,J_+]$). We define
\begin{equation}\label{def:randomwalk:inter}
 S_j=\sum_{i=0}^{j-1}Z_i,
\end{equation}
where $Z_i
$ are Bernoulli variables taking values $\pm 1$. $Z_i$'s are not necessarily
symmetric.  Thus, in the following lemma, we choose
the constants $\tA_j>0$ so that the $Z_i$ are Bernoulli
variables with $p=1/2$.

\begin{lemma}\label{lemma:randomwalkinter}
Fix $\tA>0$. There exist constants  $\tA_\pm>0$,  independent of $e_0$,
and  $\{\tA_j\}_j,\  j\in \left[\lfloor \tA_-e_0^{-\ga}\rfloor, \lfloor
\tA_+e_0^{-\ga}\rfloor \right]$ such that
\begin{itemize}
 \item
$\tA_j=\tA_{j-1}+(\tA_1-\tA_{0})\exp(-\int_0^{J_{j-1}}\frac {
2\bbb(J)}{\bms^2(J)}dJ)+\OO(e_0^{\ga}), \qquad \tA_0=0,\qquad \tA_1=\tA.$

\item $\displaystyle [J_-,J_+]\subset\bigcup_{j=\lfloor
\tA_-e_0^{-\ga}\rfloor}^{\lfloor
\tA_+e_0^{-\ga}\rfloor}[\tA_je_0^{\ga},\tA_{j+1}e_0^{\ga}]$.
\item The random walk process induced by the map \eqref{def:ModelCylinderMapAfterNF} on the
lattice $\displaystyle\{J_j\}_j,\
j\in \left[ \lfloor \tA_-e_0^{-\ga}\rfloor, \lfloor
\tA_+e_0^{-\ga}\rfloor \right]$ is a symmetric random walk.
\end{itemize}
\end{lemma}

\begin{proof}
To compute the probability of hitting (an $e_0$-neighborhood of) either
$J_{j\pm 1}$ from $J_j$, we use Lemma \ref{lemma:expectationlemsubstrip}. Consider $f$ in the kernel of the
infinitesimal generator $A$ of the diffusion process (see
\eqref{eq:diffusion-generator}) and  solve the boundary problem
\[
 \bbb(J)f'(J)+\frac{1}{2}\bms^2(J)f''(J)=0,\qquad f(J_{j-1})=0,\quad f(J_{j+1})=1
\]
The solution gives the probability of hitting $J_{j+1}$ before hitting
$J_{j-1}$  starting at a given
$J\in[J_{j-1},J_{j+1}]$. The unique solution is given by
\[
 f(J)=\frac{\int_{J}^{J_{j+1}}\exp(-\int_0^\iota
\frac{2\bbb(s)}{\bms^2(s)}ds)\ d\iota}{
\int_{J_{j-1}}^{J_{j+1}}\exp(-\int_0^\iota
\frac{2\bbb(s)}{\bms^2(s)}ds)\ d\iota}.
\]
We use $f$ to choose the coefficients $\tA_j$ iteratively (both as $j>0$ increases
and $j<0$ decreases). Assume that $\tA_{j-1}$, $\tA_j$ have been fixed. Then, to
have a symmetric random walk, we have to choose $\tA_{j+1}$ such that
$f(J_j)=1/2$.

Define
\[
  m(\iota)=\exp\left(-\int_0^\iota
\frac{2\bbb(s)}{\bms^2(s)}ds\right)
 \]
 and $\tD_j=\tA_j-\tA_{j-1}$. Then, using the mean value theorem, $f(J_j)=1/2$ can be
 written as
 \[
  \frac{m(\xi_j)D_j}{m(\xi_j)D_j+m(\xi_{j+1})D_{j+1}}=\frac{1}{2},
 \]
 where $\xi_j\in [\tA_{j-1}e_0^\gamma,\tA_je_0^\gamma]$ and $\xi_{j+1}\in [\tA_{j}e_0^\gamma,\tA_{j+1}e_0^\gamma]$. Thus, one
has
 \[
  \tD_{j+1}=\frac{m(\xi_{j})}{m(\xi_{j+1})}\tD_j \quad \text{  which implies }\quad
  \tD_{j+1}=\frac{m(\xi_{1})}{m(\xi_{j+1})}\tD_1.
 \]
Thus the ``normalized length'' $\tD_j=\tA_j-\tA_{j-1}$ of the strip $I_\ga^j=[J_j,J_{j+1}]=[\tA_j e_0^\ga,
\tA_j e_0^\ga]$ is given by
 \[
  \tD_j=\frac{m(\xi_{1})}{m(\xi_{j})}\ \tD_0
 =\tA  \exp\left(-\int_0^{J_{j-1}}
\frac{2\bbb(s)}{\bms^2(s)}ds+\OO(e_0^{\ga})\right).
 \]
%
%
The distortion of the strips does not depend on $e_0$ at first order. Therefore, adjusting $\tA$  one can obtain the intervals
$[J_{j}, J_{j+1}]=[\tA_{j} e_0^\ga, \tA_{j+1} e_0^\ga]$ which  cover
$[J_-,J_+]$ with $J>0$.
Proceeding analogously for $j<0$, one can do the same for $[J_-,J_+]$ with $\{J<0\}$.
\end{proof}

To prove \eqref{def:ExpectLemma}, we need to combine the iterations within each strip $I_\ga^j$ and the random walk evolution among the strips.
%

Define the random variable $j^*$ giving the number of changes of strip until
reaching the final time $n=\lfloor \kk e_0^{-2}\rfloor$ or the boundaries of
$[J_-,J_+]$. We define the
Markov times  $0=n_0<n_1<\dots <n_{j^*-1}<n_{j^*}<n$
for some random $j^*=j^*(\om)$ such that each $n_j$ is the stopping
time in the visit of the $j$-strip. Almost surely $j^*(\om)$ is
finite.  We decompose the above sum as $\eta_f=\sum_{j=0}^{j^*} \eta_j$ with
\begin{equation*}
\eta_j=f(J_{n_{j+1}})-
f(J_{n_{j}})-   e_0^2
\sum_{k=n_j}^{n_{j+1}}\left(\bbb(J_k)f'(J_k)+\frac{\bms^2(J_k)}{2}
f''(r_k)\right).
\end{equation*}
Lemma \ref{lemma:expectationlemsubstrip} implies that for
any $j$,
\begin{equation}\label{def:localExpec}
|\E(\eta_j)|\lesssim e_0^{2\ga+\zeta}
\end{equation}
for some $\zeta>0$. Define $\Delta_j=n_{j+1}-n_j$. We split
$\E(\eta_f)$ as
\begin{equation}\label{def:SplitGlobalExpec}
 \begin{split}
\E(\eta_f)= & \E\left(\left.\sum_{j=0}^{j^*}\eta_j\right|
e_0^{-2(1-\ga)+\de}\leq \Delta_j\leq
e_0^{-2(1-\ga)-\de}\,\,\forall j\right)\\
&\qquad\times\Prob\left(e_0^{-2(1-\ga)+\de}\leq
\Delta_j\leq
e_0^{-2(1-\ga)-\de}\,\,\forall j\right)\\
&+\E\left(\left.\sum_{j=0}^{j^*}\eta_j\right| \,\,\exists j \,\text{ s.\
t. }
\Delta_j<e_0^{-2(1-\ga)+\delta} \text{ or }
\Delta_j>e_0^{-2(1-\ga)-\delta}\right)\\
&\qquad\times \Prob\left(\exists
j \,\text{ s.\ t. }
\Delta_j< e_0^{-2(1-\ga)+\delta} \text{ or }
\Delta_j> e_0^{-2(1-\ga)-\delta}\right),
\end{split}
\end{equation}
where $j$ satisfies $0\leq j\leq j^*-1$.

We first bound the second term in the sum. We need
to estimate how many strips the iterates may
visit for $n\le \kk e_0^{-2}$. Since we have $|J_n-J_{n-1}|\lesssim e_0$,
there exists a
constant $c>0$ such that
\[
 |n_{j+1}-n_j|\geq c e_0^{\ga-1}\quad\text{ for }j=0,\ldots, j^*-1.
\]
Therefore, since $n\le \kk e_0^{-2}$,
\begin{equation}\label{def:maxvisitedstrips:final}
j^*\lesssim e_0^{-1-\ga}.
\end{equation}
Then, by  Lemma \ref{lemma:exittime:Igamma}, for any small
$\de>0$,
\[
 \Prob\left(\exists
k \,\text{ s. t. }
\Delta_j< e_0^{-2(1-\ga)+\delta} \text{ or }
\Delta_j>e_0^{-2(1-\ga)-\delta}\right)\leq
e_0^{-1-\ga}e^{-Ce_0^{-\de}}.
\]
This implies,
\[
\begin{split}
\Bigg| \E\Bigg(\sum_{j=0}^{j^*}\eta_j\Bigg| &\,\,\exists j \,\text{ s.
t. }
\Delta_j<e_0^{-2(1-\ga)+\delta} \text{ or }
\Delta_j>e_0^{-2(1-\ga)-\delta}\Bigg)\Bigg|\\
&\times \Prob\left(\exists
j \,\text{ s. t. }
\Delta_j<e_0^{-2(1-\ga)+\delta} \text{ or }
\Delta_j>e_0^{-2(1-\ga)-\delta}\right)\\
&\leq
e_0^{-1-\ga}\cdot e_0^{2\ga+\zeta}\cdot e_0^{-1-\ga}e^{
-Ce_0^{-\de}}.
\end{split}
\]
Now we bound the first term in \eqref{def:SplitGlobalExpec}. The assumptions on the exit times $\Delta_j$ when conditioning the probability imply
\begin{equation}\label{eq:j*:final}
 e_0^{-2\ga+\de}\leq j^*\leq e_0^{-2\ga-\de}.
\end{equation}
Now we are ready to prove that the first term in \eqref{def:SplitGlobalExpec}
tends to zero with $e_0$. We bound the probability by one. To prove
that
the conditioned expectation in the first line tends to zero with $e_0$,
it is enough to take into account \eqref{def:localExpec} and
\eqref{eq:j*:final}, to obtain
\[
\left|\E\left(\left.\sum_{j=0}^{j^*}\eta_j\right|
e_0^{-2(1-\ga)+\de}\leq \Delta_j\leq
e_0^{-2(1-\ga)-\de}\,\,\forall j\right)\right|
\lesssim
e_0^{2\ga+\zeta}\cdot e_0^{-2\ga-\de}
\leq e_0^{\zeta-\de}.
\]
Therefore, taking $\de>0$ small enough we obtain \eqref{def:ExpectLemma}.
In conclusion, we obtain  that the $J$ behaves as a diffusion process
\[
 d\JJ_\kk=\bbb(\JJ)d\kk+\bms^2(\JJ)dB_\kk,
\]
where the drift $\bbb$ and the variance $\bms^2$ are those defined in Lemma \ref{lemma:DriftVar}.

Since the change of coordinates considered in Proposition \ref{prop:Cylindermaps:normalform}
is $e_0$-close to the identity, we obtain that the original variable $I$
behaves following the same diffusion process. Finally it is enough to recall that the Jacobi constant $\tJ$ in \eqref{eq:Jacobi} satisfies $\tJ=-I+\OO(e_0)$ (see \eqref{def:HamDelaunayRot}). This completes the proof of Theorem \ref{thm:main-thm}.

\subsection{Stochastic evolution of the eccentricity (Theorem \ref{thm:main-thm-bis})}\label{sec:maintheorem2}

To deduce Theorem \ref{thm:main-thm-bis} from Theorem \ref{thm:main-thm}, we proceed as follows. 
First, note that one can deduce the angular momentum $G$ from the other variables and the Jacobi constant. Indeed,
\[
\tJ=-\frac{1}{2L^2}-G-\Delta H_\ccirc+\OO(\mu e_0),
\]
which, implies that one can write the angular momentum $G$ as
$G=\bG(L,\ell, g,\tJ, t)$, 
where $\bG$ is a function satisfying
\[
 \bG(L,\ell, g,\tJ, t)=-\tJ-\frac{1}{2L^2}+\OO(\mu).
\]
Then, if we define 
\[
\bE(\tJ,L)=\sqrt{1-\frac{(\tJ+\frac{1}{2L^2})^2}{L^2}},
\]
(see \eqref{def:JtoE}), 
we can write the eccentricity 
\[
 \ecc=\bE(\tJ,L_0)+E,\qquad L_0=3^{-1/3}
\]
where
\[
E=\sqrt{1-\frac{ \bG(L,\ell, g,\tJ, t)^2}{L^2}}-\bE(\tJ,L_0),
\]
which is $\OO(\mu)$. Now, to prove Theorem \ref{thm:main-thm-bis}, it only remains to: \begin{itemize}
\item[$(a)$] Show that if $\tJ_s$ is a diffusion process with drift $\bbb(\tJ_s)$ and variance $\bms^2(\tJ_s)$ (as stated in Theorem \ref{thm:main-thm}), then  $\tilde \ecc_s=\bE(\tJ_s,L_0)$ is also a diffusion process for suitable $\wt \bbb(\tilde\ecc_s)$ and $\wt\bms^2(\tilde\ecc_s)$.
\item[$(b)$] Show that $E$ can be split into the two correction terms provided by Theorem \ref{thm:main-thm-bis}.
\end{itemize}	
Item $(a)$ is a direct consequence of It\^o Lemma. Indeed, if we denote by $\tJ=g(\wt\ecc,L_0)$ the inverse of the function $\tJ\mapsto \tilde\ecc=\bE(\tJ,L_0$), it ensures that $\tilde\ecc_s$ also satisfies a diffusion process with drift and variance given by
\[
\begin{split}
	\wt \bbb(\tilde\ecc_s)&=\pa_\tJ \bE(g(\tilde \ecc_s,L_0),L_0)\bbb (g( \tilde\ecc_s,L_0))+\frac{1}{2}\pa^2_\tJ \bE(g(\tilde\ecc_s,L_0),L_0)\bms^2 (g( \tilde\ecc_s,L_0)),\\
	\wt \bms(\tilde\ecc_s)&=\pa_\tJ \bE(g(\tilde\ecc_s,L_0),L_0)\bms (g(\tilde\ecc_s,L_0)).
\end{split}
\]

Now we prove Item $(b)$. To this end, let us consider a section $\Pi=\{q=\rho\}$, which is transverse to the unstable manifold of the normally hyperbolic invariant cylinder ($q$ is the coordinate introduced in Lemma \ref{lemma:MoserNF}). Note that the orbits in the normally hyperbolic invariant lamination keep crossing this section at each loop (until hitting the stoping time in the boundary). For each trajectory $\wt \Phi_t(X)$ with $X$ in the support of the measure (i.e in the normally hyperbolic invariant lamination), we denote by $t_j$, $j\geq 1$, the sequence of times these orbits hit the section $\Pi$. Since the lamination is $\rr$-close to the normally hyperbolic invariant cylinder, these times satisfy $t_{j+1}-t_{j}\sim|\log\rr|$ for any $X\in \mathrm{supp}\nu_{\tJ^*, e_0}$. 

Now, for $e_0>0$ small enough, each of the points $\wt \Phi_t(X)$ are $\rr$-close to a homoclinic orbit of the RPC3BP ($e_0=0$) at the same Jacobi constant as $\wt \Phi_t(X)$ (provided by Ansatz \ref{ans:NHIMCircular:bis}). Let us denote  by $(L_j,\ell_j, g_j)(t,\tJ)$, $j=1,2$, the time parameterization of these homoclinics. Then, one can define the two corrections as follows. Note that we have to analyze
\[
\EE(X,t)=E(\wt \Phi_t(X)),
\]
which we split it as  $\EE=\EE_1+\EE_2$ as 
\[
\begin{split}
	\EE_1(X,t)&=\bE(\tJ_t(X),L_t(X))-\bE(\tJ_t(X),L_0)\\
	\EE_2(X,t)&=E_2(\wt \Phi_t(X)),
\end{split}
\]
where $z_t(X)=\pi_z\wt \Phi_t(X)$, for $z=L,\ell, g,\tJ, t$, and 
\[
E_2(X)=\sqrt{1-\frac{ \bG(L,\ell, g,\tJ, t)^2}{L^2}}-\bE(\tJ,L).
\]
We analyze each of them separately for $t\in (t_j,t_{j+1})$ and we denote $\tJ=\tJ_{t_j}(X)$. We write the first one as $\EE_1=\EE_1^1+\EE_1^2$ with 
\[
\begin{split}
	\EE_1^1(X,t)&=\bE(\tJ,L_j(t,\tJ))-\bE(\tJ,L_0)\\
	\EE_1^2(X,t)&=\left[\bE(\tJ_t(X),L_t(X))-\bE(\tJ,L_j(t,\tJ))\right]+\left[\bE(\tJ,L_0)-\bE(\tJ_t(X),L_0)\right].
\end{split}
\]
Note that since we are considering the regime $0<e_0\ll \rr\ll 1$,  $\EE_1^2$ satisfies $|\EE_1^2|\lesssim\rr^{\zeta^*}$ since the orbits in the lamination are $\rr^{\zeta^*}$-close to the homoclinic orbit (see Theorem \ref{thm:NHILElliptic}). Therefore this term is of the form of the correction $E$ in Theorem \ref{thm:main-thm-bis}. The term $\EE_1^1$ only depends on the choice of homoclinic orbit, the  Jacobi constant at time $t_j$ and on the time $t\in (t_j,t_{j+1})$, and therefore it is of the form of the correction $\de\ecc$. Moreover, by Ansatz \ref{ans:NHIMCircular:bis},
\[
|\EE_1^1(X,t)|\leq 126 \mu.
\]
Moreover, one can easily check that, by Ansatz \ref{ans:NHIMCircular:bis} and Theorem \ref{thm:NHILElliptic}, $E_2$ satisfies
\[
|E_2(X)|\leq 126\mu.
\]
We split $\EE_2=\EE_2^1+\EE_2^2$  as
\[
\begin{split}
	\EE_2^1(t,\tJ)&=E_2(L_j(t,\tJ),\ell_j(t,\tJ), g_j(t,\tJ), \tJ, t)\\
	\EE_2^2(X,t)&=E_2(\wt \Phi_t(X))-E_2(L_j(t,\tJ),\ell_j(t,\tJ), g_j(t,\tJ), \tJ, t).
\end{split}
\]
Then, proceeding analogously as for $\EE_1$, one has that $|\EE_2^2(X,t)|\lesssim\rr^{\zeta^*}$ and contributes to the correction term $E$ in Theorem \ref{thm:main-thm-bis} and that $\EE_2^1$ is of the form of the correction term $\de\ecc$ in the theorem. 

%

\section{Compact Lie group extensions of hyperbolic maps}\label{app:LiegroupExtensions}
We devote this section to prove Theorems \ref{thm:Dima} and \ref{thm:Dima:FixedAngle} and Lemma \ref{lemma:tripledecay}. Theorem \ref{thm:Dima} is a direct consequence of \cite{FieldMT03}. This is explained in Section \ref{app:proofdecaycorrLeb}. The proofs of Theorem  \ref{thm:Dima:FixedAngle} and Lemma \ref{lemma:tripledecay} require more work, which deeply relies on the tools developed in  \cite{FieldMT03} and also on a Central Limit Theorem proven in \cite{IbragimovL71}. This is explained in Section \ref{app:proofdecaycorrfixedtheta}.

\subsection{Decay of correlations and a CLT for type 2 measures (Theorem \ref{thm:Dima})}\label{app:proofdecaycorrLeb}
We start by explaining the setting of  \cite{FieldMT03} and the results that we use. Consider $\Sigma$ an aperiodic subshift of finite type and a compact Lie group $G$.  Assume that $G$ acts (orthogonally) on $\RR^n$ by identifying $G$ with a subgroup of $O(n)$ (for some $n$).

Fix $\vartheta\in (0,1)$ and consider a map $h:\Sigma\to G$ with $\|h\|_\vartheta<\infty$ (see \eqref{def:HolderNorm}). We define the skew product $\sigma_h:\Sigma\times G\to \Sigma\times G$ as
\begin{equation}\label{def:sigmahapp}
	\sigma_h(\omega,g)=\left(\sigma\omega, gh(\omega)\right),
\end{equation}
where the second component is just the (right) Lie group product by $h$. Consider the $\sigma$-equilibrium state $\mu$ and the Haar measure $\nu$ on $G$. Then, the product measure $\tm=\mu\times\nu$ is invariant under $\sigma_h$.

We consider $G$-equivariant   observables $\phi:\Sigma\times G\to \RR^n$, that is observables $\phi$ satisfying 
\begin{equation}\label{def:EquivObsApp}
	\phi(\omega, ag)=a\phi(\omega,g)\quad \text{ for all }\quad a\in G. 
\end{equation}
Equivalently, $\phi(\omega,g)=g V(\omega)$ where $V:\Sigma\to\RR^n$. We assume that $V$ is H\"older and satisfies $\|V\|_\vartheta<\infty$ (see \eqref{def:HolderNorm}). We consider the space of those observables
\[
\FF_\vartheta^G=\left\{\phi:\Sigma\times G\to\RR^n: \phi(\omega,g)=gV(\omega), \|V\|_\vartheta<\infty\right\}.
\]
Consider the space $L^2(\Sigma\times G,\RR^n)$. Note that $ \FF_\vartheta^G\subset L^2(\Sigma\times G,\RR^n)$.  It is easy to see that one can scale the norm $\|\cdot\|_\vartheta$ so that it satisfies $\|\phi\|_{L^2}\leq \|\phi\|_\vartheta$.

We also define the Koopman operator 
\begin{equation}\label{def:transferoperator}
	U:L^2(\Sigma\times G,\RR^n)\to L^2(\Sigma\times G,\RR^n),\qquad  U\phi=\phi\circ\sigma_h.
\end{equation}
Assume that
\begin{equation}\label{def:equivWM}
	\begin{gathered}	
		\text{The only solutions of \, $U\phi=\alpha\phi$ \, of the form \, $\phi(\omega,g)=gV(\omega)$}\\\text{are $\alpha=1$ and  $V=\text{constant}$.}
	\end{gathered}
\end{equation} 
Note that this is implied by  weak mixing. However, it is a weaker statement since only rules out the existence of equivariant eigenfunctions of the Koopman operator $U$.

Under the condition \eqref{def:equivWM}, in \cite{FieldMT03}, it is shown that 
there exists  $r\in (0,1)$ and $C>0$ such that for any $j\geq 1$ and $\phi,\psi\in \FF_\vartheta^G$,
\begin{equation}\label{def:DecayCorrExpec}
	\left|\int \phi\circ\sigma_h^j\psi^Td\tm-\int \phi\, d\tm\int\psi^Td\tm\right|\leq C\|\phi\|_{\vartheta^{1/2}} \|\psi\|_{\vartheta^{1/2}}\, r^j.
\end{equation}
The paper \cite{FieldMT03} also provides a Central Limit Theorem under assumption \eqref{def:equivWM}. To state it, we first need to introduce the covariance matrix. For a $G$-equivariant observable $\phi:\Sigma\times G\to\RR^n$, we define
\[
\phi_N=\sum_{j=0}^{N-1}\phi\circ\sigma_h^j.
\]
Assume that $\phi$ has  zero mean, that is $\int\phi d\tm=0$. Then, one can easily check that \eqref{def:DecayCorrExpec} implies that  the limit $\lim_{N\to\infty}N^{-1}\int \phi_N\phi_N^Td\tm$ exists, which allows to define the covariance matrix
\[
\Sigma_\phi=\lim_{N\to\infty}\frac{1}{N}\int \phi_N\phi_N^T\,d\tm.
\]
\begin{lemma}[\cite{FieldMT03}]\label{lemma:CLT}
	Assume that the matrix $\Sigma_\phi$ is non-singular. Then, for each $c\in\RR^n$, the sequence $\frac{1}{\sqrt{N}}c^T\phi_N$, $N\geq 1$, converges in distribution to a normal distribution with zero mean  and variance $\sigma^2=c^T\Sigma_\phi c$.
\end{lemma}

Note that the setting of Theorem \ref{thm:Dima} fits the one in \cite{FieldMT03}. Indeed, it is enough to consider 
the Lie Groups
\begin{equation}\label{def:LieGroup}
	G_m=\left\{g_\al=\begin{pmatrix}\cos(m\al) &\sin(m\al)\\ -\sin(m\al) &\cos(m\al) \end{pmatrix}: \al \in \RR/(2\pi/m)\ZZ\right\}\qquad \text{with}\qquad m=1,2.
\end{equation}
Then, for a given   $\vartheta$-H\"older function $\beta:\Sigma\to \RR$, one can  define $g_\beta\in  \FF_\vartheta (\Sigma,G_m)$ and the associated map
\[
\sigma_\beta(\omega,g_\al)=g_{\al+\beta(\omega)}V(\sigma\omega).
\]
For this map, the absence of non-trivial eigenvalues of the associated  operator $U$ in the unit circle is given by assumption \eqref{def:WeakMixing}. Then, we can consider the associated $G_m$-equivariant observables,
\begin{equation}\label{def:observablephi}
	\phi(\omega,\al)=\begin{pmatrix}\cos(m\al) &\sin(m\al)\\ -\sin(m\al) &\cos(m\al) \end{pmatrix}\begin{pmatrix}P_1(\omega)\\ P_2(\omega) \end{pmatrix}
\end{equation}
for $\vartheta$-H\"older functions $P_1,P_2:\Sigma\to \RR$. Note that all equivariant observables over $G_m$ have zero mean with respect to the measure $\tm$.

In this setting, the decay of correlations statement \eqref{def:DecayCorrExpec} and the Central Limit Theorem (Lemma \ref{lemma:CLT}) imply Theorem \ref{thm:Dima}.

\subsection{Decay of correlations and a  CLT for type 1 measures (Theorem \ref{thm:Dima:FixedAngle})}\label{app:proofdecaycorrfixedtheta}
The statements on decay of correlations and Central Limit Theorem in Theorem \ref{thm:Dima} consider as measure the Gibbs measure on $\Sigma$ times the Haar measure on the Lie group $G$. Now we prove Theorem \ref{thm:Dima:FixedAngle}, which gives the same statements  as  Theorem \ref{thm:Dima} but for a different measure, which is not invariant: fixing  any initial condition  $s_0\in\TT$ and considering the Gibbs measure  on $\Sigma$, that is $\tm_{s_0}=\mu\times\delta_{s_0}$.

We start by proving that the first item of the theorem  is a direct consequence of Corollary~\ref{coro:decaycorrelations} and the second item of Theorem \ref{thm:Dima:FixedAngle}. Below, in Sections \ref{app:decay:onesided} and \ref{app:decay:twosided} we prove the second item of the theorem.

If we write $F$ and $G$ as 
\[
F(\omega,s)=F^+(\omega)e^{\im ms}+\ol{F^+}(\omega)e^{-\im ms},\quad G(\omega,s)=G^+(\omega)e^{\im m s}+\ol{G^+}(\omega)e^{-\im ms},\quad m=1,2,
\]
we have that, for $n>k$,
\[
\begin{split}
	\E_{\omega}\left(G(\sigma^n\omega, s_n)F(\sigma^k\omega,
	s_k)\right)=&2\Re\left(F^+(\sigma^n\omega)\ol{G^+}(\sigma^k\omega)e^{\im m\sum_{j=k}^{n-1}\beta^{\sigma^j\omega}}\right)\\
	&+2\Re\left(F^+(\sigma^n\omega)G^+(\sigma^k\omega)e^{\im m\sum_{j=0}^{n-1}\beta^{\sigma^j\omega}}e^{\im m\sum_{j=0}^{k-1}\beta^{\sigma^j\omega}}\right).
\end{split}
\]
Then, the term in the first line can be bounded applying Corollary \ref{coro:decaycorrelations}.  The term in the second line is estimated by the second item of Theorem \ref{thm:Dima:FixedAngle}. 

We prove now the second item of Theorem \ref{thm:Dima:FixedAngle}. We strongly rely on the tools developed in \cite{FieldMT03}. This is done in two steps. First, in Section \ref{app:decay:onesided} we prove it under the assumption that  $\beta$, $F$ and $G$ only depend on future symbols $\omega_k$, $k\geq 0$. That is, we prove decay of correlations for skew-products over the one-sided shift. Then, in Section \ref{app:decay:twosided}, we explain how to adapt the proof to the general case.

\subsubsection{Decay of correlations and Lemma \ref{lemma:tripledecay} for one-sided shifts}\label{app:decay:onesided}
We first prove Item 2 of Theorem \ref{thm:Dima:FixedAngle}  for skew-products over the one-sided shift. To simplify the notation we consider $m=1$. The same proof applies to $m=2$.

Let us define $\Sigma^+=\{0,1\}^\NN$, call $\sigma$ the shift on $\Sigma^+$ and $\mu$ the Gibbs measure on $\Sigma^+$. For a function $P:\Sigma^+\to\CC$, we can define the Koopman and transfer operators associated to $\sigma$ as 
\[
\UU:L^\infty(\Sigma^+)\to L^\infty(\Sigma^+)\qquad \text{and}\qquad \LL:L^1(\Sigma^+)\to L^1(\Sigma^+),
\]
by
\[
\UU P=P\circ\sigma\qquad\text{and}\qquad \int_{\Sigma^+}\LL P\, Q\,d\mu=\int_{\Sigma^+}P\, \UU Q\,d\mu.
\]
Now, we consider the Lie group extension of the shift $\sigma$,  $\sigma_\beta:\Sigma^+\times\TT\to\Sigma^+\times\TT$ defined as
\begin{equation}\label{def:mapSigmaplus}
	\sigma_\beta(\omega,s)=(\sigma \omega, s+\beta(\omega))
\end{equation}
where $\beta:\Sigma^+\to\RR$ is a $\vartheta$-H\"older function (see \eqref{def:HolderNorm}) and, following \cite{FieldMT03}, we consider equivariant observables
\begin{equation}\label{def:observables:equiv}
	F(\omega,s)=P(\omega)e^{\im s},\qquad G(\omega,s)=Q(\omega)e^{\im s},
\end{equation}
where $P,Q:\Sigma^+\to\CC$ are also $\vartheta$-H\"older functions. We denote by $F_\vartheta(\Sigma^+)$ the space of $\vartheta$-H\"older functions $P:\Sigma^+\to\CC$.

Associated to this skew-shift and the equivariant observables, we define the twisted transfer operator
\[
\LL_\beta P=\LL(e^{\im \beta}P),
\]
which satisfies that, for $n\geq 1$,
\begin{equation}\label{def:decay:adjointequiv}
	\int_{\Sigma^+}\LL_\beta^n P\, Qd\mu=\int_{\Sigma^+}e^{\im \beta_n}P\, \UU^n Qd\mu\qquad \text{where}\qquad \beta_n=\sum_{j=0}^{n-1}\beta\circ\sigma^j.
\end{equation}
Note that the right hand side above can be seen as a Koopman operator for equivariant observables \eqref{def:observables:equiv},
\[
F\circ\sigma_\beta=e^{\im \beta}\UU P e^{\im s}.
\]
The following proposition is proven in \cite[Corollary 4.1]{FieldMT03}.
\begin{proposition}\label{prop:decayonesided}
	Fix $\vartheta\in (0,1)$. Consider the map $\sigma_\beta$ in \eqref{def:mapSigmaplus} with $\beta:\Sigma^+\to\RR$ $\vartheta$-H\"older and assume  that the equation $e^{\im \beta(\omega)} A(\sigma\omega)=\nu A(\omega) $ has at most one solution in $F_\vartheta(\Sigma_+)$: the trivial solution $\nu=1$, $A_\omega=\text{constant}$. Then, there exists $C>0$ and $r\in (0,1)$ such that 
	\[
	\left\| \LL_\beta^n P\right\|_{\vartheta}\leq Cr^n\|P\|_\vartheta\qquad \text{for all}\quad P\in F_\vartheta (\Sigma_+).
	\]
\end{proposition}
Note that the assumption \eqref{def:WeakMixing} implies that there exists $C>0$ and $r\in (0,1)$ such that $\LL_{\beta}$ and  $\LL_{2\beta}$ satisfy the proposition.

Assumption \eqref{def:WeakMixing} and this proposition will imply Item 2 of Theorem \ref{thm:Dima:FixedAngle} for the one-sided shift, that is 
\[
\left|\E_\omega \left( P\left(\sigma^n\omega\right)e^{\im \beta_n(\omega)} Q\left(\sigma^k\omega\right)e^{\im \beta_k(\omega)}\right)\right|\leq 
C\,\|P\|_{\vartheta}\,\|Q\|_{\vartheta}\ r^{n},
\]
for any $n> k$ and $P, Q\in F_\vartheta(\Sigma_+)$.

Note that the left hand side can be written as 
\[
J_{n,k}=\E_\omega \left( P\left(\sigma^n\omega\right)e^{\im \beta_n(\omega)} Q\left(\sigma^k\omega\right)e^{\im \beta_k(\omega)}\right)=\int_{\Sigma^+}(\UU^n P)e^{i\beta_n}(\UU^k Q)e^{i\beta_k}d\mu.
\]
Now, since $\beta_n=\beta_k+\UU^k\beta_{n-k}$, one has that 
\begin{equation}\label{def:betashifts}
	e^{i\beta_n}=e^{i\beta_k}\UU^ke^{i\beta_{n-k}}. 
\end{equation}
Then, using also \eqref{def:decay:adjointequiv},
\[
\begin{split}
	J_{n,k}&=\int_{\Sigma^+}e^{2\im\beta_k}\UU^k \left(e^{i\beta_{n-k}}Q\,\UU^{n-k} P\right)d\mu=\int_{\Sigma^+} Q \LL_{2\beta}^k 1 e^{i\beta_{n-k}}\UU^{n-k} Pd\mu\\
	&=\int_{\Sigma^+} \LL_{\beta}^{n-k}\left(Q \LL_{2\beta}^k 1 \right) Pd\mu.
\end{split}
\]
Therefore, denoting by $|\cdot |_\infty$ the norm associated to $L^\infty(\Sigma_+)$ and recalling that $|\cdot|_\infty\leq \|\cdot\|_{\vartheta}$, Proposition \ref{prop:decayonesided} and \eqref{def:HolderAlgebra} imply 
\[
\begin{split}
	|J_{n,k}|&\leq \left|\LL_{\beta}^{n-k}\left(Q \LL_{2\beta}^k 1 \right)\right|_\infty |P|_\infty\leq Cr^{n-k}\|Q \LL_{2\beta}^k 1\|_\vartheta |P|_\infty\\
	&\leq Cr^{n-k}\|Q\|_\vartheta \| \LL_{2\beta}^k 1\|_\vartheta |P|_\infty\leq C^2r^{n}\|Q\|_\vartheta \|  1\|_\vartheta |P|_\infty\leq C^2r^{n}\|Q\|_\vartheta  |P|_\infty.
\end{split}
\]
This completes the proof of Item 2 of Theorem \ref{thm:Dima:FixedAngle} for one-sided shifts. 

Actually,  to generalize this result to two-sided shifts we need a more involved decay of correlations statement which involves more observables. This is stated in the following lemma, which will be also used to prove Lemma \ref{lemma:tripledecay}.



\begin{lemma}\label{lemma:tripledecay:onesided}
	Fix $\vartheta\in (0,1)$. Then, there exist constants $r\in (0,1)$ and
	$C>0$ such that, for any $n, k,i\in\NN$, $n>k>i$ and  any  $\vartheta$-H\"older functions $P,Q,S, T:\Sigma^+\to \CC$, the following is satisfied,
	\[
	\begin{split}
		\left|\E_\omega \left( e^{\im \beta_n(\omega)} P\left(\sigma^n\omega\right) e^{\im \beta_k(\omega)} Q\left(\sigma^k\omega\right) S\left(\omega\right)\right)\right|&\leq C\,|P|_{\infty}\,\|Q\|_{\vartheta}\,\|S\|_{\vartheta}\ r^{n}\\
		\left|\E_\omega \left( e^{\im \beta_n(\omega)} P\left(\sigma^n\omega\right) Q\left(\sigma^k\omega\right) e^{\im \beta_i(\omega)} S\left(\sigma^i\omega\right)T(\omega)\right)\right|&\leq C\,|P|_{\infty}\,\|Q\|_{\vartheta}\,\|S\|_{\vartheta}\,\|T\|_{\vartheta}\ r^{n}.	
	\end{split}
	\]
\end{lemma}
\begin{proof}
	Note that the first statement in Lemma \ref{lemma:tripledecay:onesided} is implied by the second one just taking $Q=1$. Therefore, it is enough to prove the second one.

	We denote
	\[
	J_{nki}=\E_\omega \left( e^{\im \beta_n(\omega)} P\left(\sigma^n\omega\right) Q\left(\sigma^k\omega\right) e^{\im \beta_i(\omega)} S\left(\sigma^i\omega\right)T(\omega)\right)\\
	=\int_{\Sigma^+}e^{\im \beta_n}\UU^nP\,\UU^kQ \, e^{\im \beta_i} \UU^iS\, Td\mu.
	\]
	Then, using \eqref{def:betashifts},
	\[
	\begin{split}
		J_{nki}=&\int_{\Sigma^+}e^{\im \beta_k}\UU^k\left(Q e^{\im \beta_{n-k}}\UU^{n-k}P\right) e^{\im \beta_i}\UU^iS\, Td\mu\\
		=&\int_{\Sigma^+}e^{2\im \beta_i}\UU^i\left[S e^{\im \beta_{k-i}}\UU^{k-i}\left(Q e^{\im \beta_{n-k}}\UU^{n-k}P\right)\right] Td\mu
	\end{split}
	\]
	and,  by \eqref{def:decay:adjointequiv}, 
	\[
	\begin{split}
		J_{nki}=&\int_{\Sigma^+} e^{\im \beta_{k-i}}\UU^{k-i}\left(Q e^{\im \beta_{n-k}}\UU^{n-k}P\right)S\LL_{2\beta}^iT d\mu\\
		=&\int_{\Sigma^+}Q e^{\im \beta_{n-k}}\UU^{n-k}P\LL_{\beta}^{k-i}\left(S\LL_{2\beta}^iT\right) d\mu\\
		=&\int_{\Sigma^+}P\LL_{\beta}^{n-k}\left[Q\LL_{\beta}^{k-i}\left(S\LL_{2\beta}^iT\right) \right]d\mu.
	\end{split}
	\]
	Then, Proposition \ref{prop:decayonesided} and \eqref{def:HolderAlgebra} imply 
	\[
	\begin{split}
		|J_{nki}|\leq &|P|_\infty\left|\LL_{\beta}^{n-k}\left[Q\LL_{\beta}^{k-i}\left(S\LL_{2\beta}^iT\right) \right]\right|_\infty \leq C r^{n-k}|P|_\infty\left\|Q\right\|_\vartheta\left\|\LL_{\beta}^{k-i}\left(S\LL_{2\beta}^iT\right)\right\|_\vartheta\\
		\leq &C^2 r^{n-i} |P|_\infty\left\|Q\right\|_\vartheta\left\|S\right\|_\vartheta\left\|\LL_{2\beta}^iT\right\|_\vartheta\leq C^3 r^{n}|P|_\infty\left\|Q\right\|_\vartheta\left\|S\right\|_\vartheta\|T\|_\vartheta.
	\end{split}
	\]
\end{proof}

\subsubsection{Decay of correlations and Lemma \ref{lemma:tripledecay} for two-sided shifts}\label{app:decay:twosided}
To go from the one-sided shift to the two-sided shift, we  proceed in two steps. We  first  consider observables defined in $\Sigma=\{0,1\}^\ZZ$ but $\beta$ still defined in $\Sigma^+$ (i.e depending only on future symbols). Then, in the second step we consider also a general $\beta:\Sigma\to \RR$.

\begin{lemma}\label{lemma:tripledecay:twosided}
	Fix $\vartheta\in (0,1)$ and consider the map $\sigma_\beta$ in~\eqref {def:mapSigmaplus} where $\beta:\Sigma^+\to\RR$ is a $\vartheta$-H\"older function. Assume that $\sigma_\beta$ satisfies \eqref{def:WeakMixing}. Then, there exist constants $r\in (0,1)$ and
	$C>0$ such that, for any  $n, k,i\in\NN$, $n>k>i$ and  any  $\vartheta$-H\"older observables $P,Q,S, T:\Sigma\to \CC$, the following is satisfied,
	\[
	\begin{split}
		\left|\E_\omega \left( e^{\im \beta_n(\omega)} P\left(\sigma^n\omega\right) e^{\im \beta_k(\omega)} Q\left(\sigma^k\omega\right) S\left(\omega\right)\right)\right|&\leq C\,\|P\|_{\vartheta}\,\|Q\|_{\vartheta}\,\|S\|_{\vartheta}\ r^{n}\\
		\left|\E_\omega \left( e^{\im \beta_n(\omega)} P\left(\sigma^n\omega\right) Q\left(\sigma^k\omega\right) e^{\im \beta_i(\omega)} S\left(\sigma^i\omega\right)T(\omega)\right)\right|&\leq C\,\|P\|_{\vartheta}\,\|Q\|_{\vartheta}\,\|S\|_{\vartheta}\,\|T\|_{\vartheta}\ r^{n}.	
	\end{split}
	\]
\end{lemma}
\begin{proof}
	As already mentioned in the proof of Lemma \ref{lemma:tripledecay:onesided}, the first statement is a direct consequence of the second one taking $Q=1$. Therefore, we just have to prove the second one.
	
	To prove the second statement, we proceed as in the proof of Lemma 4.9 in \cite{FieldMT03} (see also \cite{ParryP90}). The first step is to consider slightly modified observables $P_j$, $Q_j$, $S_j$, $T_j$	which depend only on coordinates $\omega_l$ with $|l|\leq j$ and satisfy
	\begin{equation}\label{def:aproxobservables}
		|P_j|_\infty\leq |P|_\infty\qquad \text{and}\qquad |P-P_j|_\infty\leq \|P\|_\vartheta \,\vartheta^k
	\end{equation}
	(and analogously for the other observables, its construction is explained in  \cite{FieldMT03}).
	
	Then, we write 
	\begin{equation}\label{def:Jkni}
		J_{nki}=\E_\omega \left( e^{\im \beta_n(\omega)} P\left(\sigma^n\omega\right) Q\left(\sigma^k\omega\right) e^{\im \beta_i(\omega)} S\left(\sigma^i\omega\right)T(\omega)\right)=\int_{\Sigma}e^{\im \beta_n} P\circ\sigma^n Q\circ\sigma^k e^{\im \beta_i} S\circ\sigma^iTd\mu
	\end{equation}
	as
	\[
	\begin{split}
		J_{nki}=&\int_{\Sigma}e^{\im \beta_n}( P-P_j)\circ\sigma^n Q\circ\sigma^k e^{\im \beta_i} S\circ\sigma^iTd\mu\\
		&+\int_{\Sigma}e^{\im \beta_n} P_j\circ\sigma^n (Q-Q_j)\circ\sigma^k e^{\im \beta_i} S\circ\sigma^iTd\mu\\
		&+\int_{\Sigma}e^{\im \beta_n} P_j\circ\sigma^n Q_j\circ\sigma^k e^{\im \beta_i} (S-S_j)\circ\sigma^iTd\mu\\
		&+\int_{\Sigma}e^{\im \beta_n} P_j\circ\sigma^n Q_j\circ\sigma^k e^{\im \beta_i} S_j\circ\sigma^i(T-T_j)d\mu\\
		&+\int_{\Sigma}e^{\im \beta_n} P_j\circ\sigma^n Q_j\circ\sigma^k e^{\im \beta_i} S_j\circ\sigma^iT_jd\mu.
	\end{split}
	\]
	Taking into account \eqref{def:aproxobservables}, the term in the first line can be estimated as
	\[
	\begin{split}
		\left|\int_{\Sigma}e^{\im \beta_n}( P-P_j)\circ\sigma^n Q\circ\sigma^k e^{\im \beta_i} S\circ\sigma^iTd\mu\right|&\leq \left|( P-P_j)\circ\sigma^n\right|_\infty  \left|Q\circ\sigma^k \right|_\infty  \left| S\circ\sigma^i\right|_\infty  \left|T\right|_\infty\\
		&\leq \left\| P\right\|_\vartheta  \left|Q\right|_\infty  \left| S\right|_\infty  \left|T\right|_\infty \vartheta^j\\
		&\leq \left\| P\right\|_\vartheta \left\| Q\right\|_\vartheta\left\| S\right\|_\vartheta\left\| T\right\|_\vartheta\vartheta^j.
	\end{split}
	\]
	Proceeding analogously, one can estimate also the terms in the second, third and fourth lines and one obtains the same upper bound. Therefore
	\begin{equation}\label{def:Jkni:estimate}
		\left|J_{nki}\right|\leq 4 \left\| P\right\|_\vartheta \left\| Q\right\|_\vartheta\left\| S\right\|_\vartheta\left\| T\right\|_\vartheta\vartheta^j+\left|\wh J_{nki}\right|,
	\end{equation}
	where 
	\begin{equation*}
		\wh J_{nki}=\int_{\Sigma}e^{\im \beta_n} P_j\circ\sigma^n Q_j\circ\sigma^k e^{\im \beta_i} S_j\circ\sigma^iT_jd\mu.
	\end{equation*}
	We now estimate this term. Since the measure $\mu$ is $\sigma$-invariant, 
	\[
	\wh J_{nki}=\int_{\Sigma}e^{\im \beta_n\circ \sigma^j} P_j\circ\sigma^{n+j} Q_j\circ\sigma^{k+j} e^{\im \beta_i\circ \sigma^j} S_j\circ\sigma^{i+j}T_j\circ \sigma^jd\mu.
	\]
	Note that 
	\[
	e^{\im\beta_{n+j}}=e^{\im\beta_{n}\circ\sigma^j}e^{\im\beta_{j}}=e^{\im\beta_{j}\circ\sigma^n}e^{\im\beta_{n}},
	\]
	which implies that 
	\[
	e^{\im\beta_{n}\circ\sigma^j}=e^{\im\beta_{j}\circ\sigma^n}e^{\im\beta_{n}}e^{-\im\beta_{j}}.
	\]
	We use this identity to rewrite $\wh J_{nki}$ as 
	\[
	\wh J_{nki}=\int_{\Sigma}e^{\im \beta_n} \wh P_j\circ\sigma^{n}\wh Q_j\circ\sigma^{k} e^{\im \beta_i}\wh S_j\circ\sigma^{i}\wh T_jd\mu,
	\]
	where
	\[
	\wh P_j= e^{\im \beta_j}P_j\circ\sigma^j,\quad  \wh Q_j=Q_j\circ\sigma^j,\quad  \wh S_j= e^{\im \beta_j}S_j\circ\sigma^j,\quad  \wh T_j= e^{-2\im \beta_j}T_j\circ\sigma^j.
	\]
	Note that, by construction, these observables only depend on future variables. That is, they are defined on $\Sigma^+$. Then, we can apply Lemma \ref{lemma:tripledecay:onesided}, which implies 
	\[
	\left|\wh J_{nki}\right|\leq C\,|\wh P_j|_{\infty}\,\|\wh Q_j\|_{\vartheta}\,\|\wh S_j\|_{\vartheta}\,\|\wh T_j\|_{\vartheta}\ r^{n}\leq C\,|P_j|_{\infty}\,\|\wh Q_j\|_{\vartheta}\,\|\wh S_j\|_{\vartheta}\,\|\wh T_j\|_{\vartheta}\ r^{n}.
	\]
	Now we bound each term in the right hand side. We have that $|P_j|_\infty\leq |P|_\infty$ by \eqref{def:aproxobservables}. For the other terms, proceeding as in \cite{FieldMT03}, one can see that 
	\[
	\|\wh Q_j\|_{\vartheta}\leq C_2 |Q|_{\infty}\vartheta^{-2j}
	\]
	for some $C_2\geq 1$ which only depends on $\vartheta$ and $\beta$ (but is independent of the observable). 
	
	Then, one obtains
	\[
	\left|\wh J_{nki}\right|\leq CC_2^3\,|P|_{\infty}\,| Q|_{\infty}\,| S|_{\infty}\,| T|_{\infty}\ r^{n}\vartheta^{-6j}.
	\]
	This estimate and \eqref{def:Jkni:estimate} lead to the following bound for $J_{nki}$ (see \eqref{def:Jkni}),
	\[
	\left|J_{nki}\right|\leq C  \left\| P\right\|_\vartheta \left\| Q\right\|_\vartheta\left\| S\right\|_\vartheta\left\| T\right\|_\vartheta\left(\vartheta^j+r^{n}\vartheta^{-6j}\right)
	\]
	for some $C>0$ independent of the observables. Therefore to complete the proof it is enough to define $\nu>1$ such that $r^\nu=\vartheta$ and choose $j=\lfloor n/7\nu\rfloor$.
\end{proof}	

Now we are ready to prove the second statement in Theorem \ref{thm:Dima:FixedAngle}. We now  consider the map $\sigma_\beta$ in \eqref{def:mapSigmaplus} satisfying \eqref{def:WeakMixing}, where now $\beta:\Sigma\to\RR$ is a $\vartheta$-H\"older function (see \eqref{def:HolderNorm}) defined on $\Sigma=\{0,1\}^\ZZ$. The first step is to perform a change of coordinates so that $\beta$ depends only on future symbols (that is, defined on $\Sigma^+$). The following lemma is proven in \cite{ParryP90}.

\begin{lemma}\label{lemma:FromTwoTo1Sided}
	Fix $\vartheta\in (0,1)$ and consider a $\vartheta$-H\"older function $\beta:\Sigma\to\RR$. Then, there exists $\vartheta^{1/2}$-H\"older functions $\wt\beta:\Sigma^+\to\RR$ and $\chi:\Sigma\to\RR$ such that
	\[
	\beta=\wt \beta+\chi-\chi\circ\sigma.
	\]
\end{lemma}

Note that, by construction $\sigma_\beta$ satisfies \eqref{def:WeakMixing} if and only if $\sigma_{\wt\beta}$ satisfies it.

We use this lemma to rewrite the left hand side of \eqref{def:DecayCorrelations:1} in Theorem \ref{thm:Dima:FixedAngle} as 
\[
\begin{split}
	\II_{nk}&=\E_\omega \left( e^{\im \beta_n(\omega)}P\left(\sigma^n\omega\right) e^{\im \beta_k(\omega)}Q\left(\sigma^k\omega\right)\right)\\
	&=\E_\omega \left( e^{\im \wt \beta_n(\omega)} \wt P\left(\sigma^n\omega\right)e^{\im \wt \beta_k(\omega)}\wt Q\left(\sigma^k\omega\right)\wt S(\omega)\right),
\end{split}
\]
where 
\[
\wt P=e^{-\im\chi }P,\qquad \wt Q=e^{-\im\chi }Q\qquad \text{and}\qquad \wt S=e^{2\im\chi }
\]
and $\wt\beta$ and $\chi$ are given by Lemma \ref{lemma:FromTwoTo1Sided}.

Now, the term obtained in the expectation fits the framework of the first statement in Lemma~\ref{lemma:tripledecay:twosided} for $\vartheta^{1/2}$-H\"older observables. Therefore, using also \eqref{def:HolderAlgebra},
\[
\begin{split}
	\left|\II_{nk}\right|&\leq C\,\|\wt P\|_{\vartheta^{1/2}}\,\|\wt Q\|_{\vartheta^{1/2}}\,\|\wt S\|_{\vartheta^{1/2}}\ r^{n}\\
	&\leq C\,\| P\|_{\vartheta^{1/2}}\,\| Q\|_{\vartheta^{1/2}}\,\| e^{-\im\chi }\|^2_{\vartheta^{1/2}}\,\| e^{2\im\chi }\|_{\vartheta^{1/2}}\ r^{n}.
\end{split}
\]
Note that the terms $\| e^{2\im\chi }\|_{\vartheta^{1/2}}$ and $\| e^{-\im\chi }\|_{\vartheta^{1/2}}$ do not depend on the observables and therefore  \eqref{def:DecayCorrelations:3}  is satisfied for a suitable constant $C>1$ depending on $\vartheta$ and $\beta$ but independent of the observables. This completes the proof of the second statement of Theorem \ref{thm:Dima:FixedAngle}. The proof of Lemma \ref{lemma:tripledecay} from Lemma \ref{lemma:tripledecay:twosided} follows exactly the same lines.

\subsubsection{Central Limit Theorem}\label{sec:fixedtheta}
To prove Item 3 of Theorem \ref{thm:Dima:FixedAngle}, we  rely again on the  Lie Group extension setting of \cite{FieldMT03} and also on the approach in \cite{IbragimovL71} to prove Central Limit Theorems for ``weakly dependent'' random variables. 

For  $\beta:\Sigma\to\RR$, we consider the associated  map 
\begin{equation}\label{def:mapsigmah_CLT}
	\sigma_\beta(\omega,g)=(\sigma\omega, gh(\omega)),
\end{equation}
where $h:\Sigma\to G$ is 
\begin{equation}\label{def:h}
	h(\omega)=
	\begin{pmatrix}
		\cos\beta(\omega)&	\sin\beta(\omega)\\
		-\sin\beta(\omega)&	\cos\beta(\omega)
	\end{pmatrix},
\end{equation}
and  $G=G_1$ is the Lie group introduced in \eqref{def:LieGroup}.

We consider equivariant observables in $F_\theta(\Sigma\times \TT,\RR^n)$, that is observables of the form 
\[
\phi(\omega, g)=g V(\omega).
\]
We fix $c\in\RR^n$ and consider the sequence of random variables
\begin{equation}\label{def:XjsCLT}
	X_j=c^Tg h_jV\circ\sigma^j,\qquad h_j(\omega)=h(\omega)h(\sigma\omega)\ldots h(\sigma^{j-1}\omega).
\end{equation}
Without loss of generality, we can take $c=(1,0)^T$. 

\begin{proposition}\label{prop:CLT}
	Assume that the map $\sigma_\beta$ in \eqref{def:mapsigmah_CLT} satisfies \eqref{def:WeakMixing}. Assume also that, for any $g\in G$, the random variables $X_j$ satisfy that
	\[
	\sum_{j=1}^{+\infty}\E_\omega X_j=0.
	\]
	Then,
	\begin{itemize}  
		\item The variance  
		\begin{equation}\label{def:SumVariance:0}
			\begin{split}
				\bms^2=&\frac{1}{2}\E_\omega\left(V_1^2(\omega)+V_2^2(\omega)\right)\\
				&+\frac{1}{2}\Re\E_\omega\left((V_1(\sigma^j\omega)-\im V_2(\sigma^j\omega)(V_1(\omega)+\im V_2(\omega)))e^{\im\sum_{k=0}^{j-1}\beta_{\sigma^k\omega}}\right)
			\end{split}
		\end{equation}
		is well defined. 
		
		\item Assume that $\bms^2\neq 0$. Then, for any fixed $g\in G$, the sequence
		\[
		\frac{1}{\sqrt{N}}\sum_{j=1}^NX_j
		\]
		converges weakly to a normal distribution $\NNN(0,\bms^2)$ as $N\to+\infty$.
	\end{itemize}
	
\end{proposition}


Note that the first item of the proposition is a direct consequence of Item 2 of Theorem \ref{thm:Dima:FixedAngle}. We devote the rest of this section to prove the second item of  
the proposition. The first step is to apply a transformation so that the observable and the matrix $h$ only depend on the past symbols. To this end, we define
\[
\Sigma^-=\{0,1\}^{\ZZ^-}.
\]
\begin{lemma}\label{lemma:FromTwoTo1Sided:v2}
	Fix $\vartheta\in (0,1)$ and consider $\vartheta$-H\"older functions $h:\Sigma\to G$ of the form \eqref{def:h} and $V:\Sigma\to \RR^2$ and an equivariant observable $\phi(\omega,g)=gV(\omega)$. Then, there exists $\vartheta^{1/2}$-H\"older functions $\wt h, \wh M:\Sigma^-\to G$ (with $\wt h$ given by a function $\wt\beta:\Sigma^-\to \RR$ as in \eqref{def:h}) and $W,\chi:\Sigma\to\RR$ such that
	\[
	h=\wh M \wt h (\wh M\circ\sigma)^{-1}\qquad\text{and}\qquad {\wh M}^{-1} V=W+\chi-\wt h\chi\circ\sigma.
	\]
	Moreover, $\sigma_\beta$ satisfies \eqref{def:WeakMixing} if and only if $\sigma_{\wt \beta}$ satisfies \eqref{def:WeakMixing}.
\end{lemma}
The proof of this lemma can be found in \cite{ParryP90}.

If we define the observable
\[
\psi(\omega, g)=g W(\omega),
\]
it can be easily seen (see \cite{FieldMT03}) that the Birkhoff sums
\[
\phi_N=\sum_{j=0}^{N-1}\phi\circ\sigma_h^j,\qquad \psi_N=\sum_{j=0}^{N-1}\psi\circ\sigma_{\wt h}^j,
\]
satisfy
\[
\phi_N=M\psi_N +\OO(1)\qquad \text{where}\quad M(\omega,g)=g\wh M(\omega)g^{-1}.
\]
Moreover, since $M\in G$, the covariant matrices 
\[
\Sigma_\phi=\lim_{N\to+\infty}\frac{1}{N}\E_\omega(\phi_N\phi_N^T),\qquad \Sigma_\psi=\lim_{N\to+\infty}\frac{1}{N}\E_\omega(\psi_N\phi_N^T)
\]
satisfy
\[
\Sigma_\phi=\Sigma_\psi.
\]
Note that the variance $\bms^2$ in \eqref{def:SumVariance:0} is just $\bms^2=c^T\Sigma_\phi c$ with $c=(1,0)^T$.

In \cite{FieldMT03} it is also proven the following lemma.
\begin{lemma}
	The  sequences $\{M\phi_N,N\geq 1\}$ and $\{\phi_N,N\geq 1\}$ have the same joint probability distribution.
\end{lemma}

The results above imply that proving a central limit theorem for  the sequence $X_j=c^T\phi_j$ in \eqref{def:XjsCLT} with convergence to a normal distribution $\mathcal{N}(0,\bms^2)$ with the variance in \eqref{def:SumVariance:0} is equivalent to proving it  for the sequence $\wt X_j=c^T\psi_j$ with the same limiting distribution.

From now on, we just deal with $\{\wt X_j\}_{j\geq 1}$. Note that the random variables $\wt X_j$ only depend on the symbols $\{\omega_k\}_{k\leq j}$. To simplify the notation, we drop the tildes in $X_j$.

To prove the central limit theorem for this sequence, we split the sum $S_n=X_1+\ldots+ X_n$ as follows. For fixed $p,q\in\NN$, there exists $k\in\NN$ such that $S_n$ can be written 
as
\[
S_n=\sum_{j=0}^k \xi_j+\sum_{j=0}^k \eta_j,
\]
where
\begin{align}\label{def:xieta}
	\xi_j&=\sum_{m=jp+jq+1}^{(j+1)p+jq}X_m, &\text{for}&\,\,0\leq j\leq k-1\\
	\eta_j&=\sum_{m=(j+1)p+jq+1}^{(j+1)p+(j+1)q}X_m, &\text{for}&\,\, 0\leq j\leq k-1\\
	\eta_k&=\sum_{m=kp+kq+1}^{n}X_m.&&
\end{align}
That is, any two random variables $\xi_i,\xi_j,i\neq j$, are ``separated'' by 
one variable $\eta_j$ containing $q$ terms. 
We choose the parameters $p$ and $q$ as
\begin{equation}\label{def:choicepq}
	p\sim n^\tau\qquad \text{and}\qquad q\sim \log^2 n,
\end{equation}
for $\tau\in (0,1)$ to be chosen later. Note that $k$ satisfies
\[
k\sim\frac{n}{p}\sim n^{1-\tau}.
\]
We first compute the variance of the random variable $S_n$. Recall that $\E_\omega  
S_n=0$. Then, 
\begin{equation}\label{def:SumVariance}
	\bms_n^2=V(S_n)=\E_\omega(S_n^2)=\sum_{i=1}^n \E_\omega(X_i^2)+2\sum_{i<j}\E_\omega(X_iX_j).
\end{equation}
Note that, by Theorem \ref{thm:Dima:FixedAngle}, there exists $0<\bms_-<\bms_+$ such that 
\[
\bms_-^2 n\leq V(S_n)\leq \bms_+^2 n
\]
Moreover,
\[
\lim_{n\to\infty}V(S_n)=\bms^2,
\]
where $\bms^2$ is the variance introduced in \eqref{def:SumVariance:0}.

We prove the Central Limit Theorem by the characteristic functions method. To this end, we have to show that 
\[
\lim_{n\to +\infty }\E_\omega e^{\frac{\im tS_n}{\bms_n}}=e^{-\frac{t^2}{2}}.
\]
We prove that in several steps. We first state several technical lemmas (Lemmas \ref{lemma:CLT:RemoveEtas} and \ref{lemma:CLT:CloseToIndep}) analyzing $\E_\omega e^{\frac{\im tS_n}{\bms_n}}$. To this end, we define
\[
\wt S_k=\sum_{j=0}^k \xi_j.
\]
Note that $ \wt S_k$ is the sum of some of the terms in $S_n$.

\begin{lemma}\label{lemma:CLT:RemoveEtas} Fix $t\in\RR$. Then,
	\[
	\left|\E_\omega e^{\frac{\im S_n}{\bms_n}}-\E_\omega
	e^{\frac{\im t\wt S_k}{\bms_n}}\right|=o(1) \text{ as }n\to+\infty.
	\]
\end{lemma}
\begin{proof}
	Since $k$ satisfies $k\sim n^{1-\tau}$ and $\eta_j$ satisfies $|\eta_j|\lesssim 
	\log^2n$, one has that 
	\[
	\left|e^{\frac{\im tS_n}{\bms_n}}-e^{\frac{\im t\wt S_k}{\bms_n}}\right|
	\leq\left|e^{\frac{it\wt 
			S_k}{\bms_n}}\left(e^{\frac{\im t}{\bms_n}\sum_{j=1}^k\eta_j}-1\right)\right|
	\lesssim \frac{\log^2n}{n^\tau},
	\]
	which implies the claim.
\end{proof}

\begin{lemma}\label{lemma:CLT:CloseToIndep} Fix $t\in\RR$. Then,
	\[
	\left|\E_\omega \left(e^{\frac{\im t\wt S_k}{\bms_n}}\right)-\prod_{j=1}^{k}\E_\omega \left(e^{\frac{\im t\xi_j}{\bms_n}}\right)\right|=o(1) \text{ as }n\to+\infty.
	\]
\end{lemma}

\begin{proof}
	We split 
	\[
	\II=\E_\omega e^{\frac{\im t\wt S_k}{\bms_n}}-\prod_{j=1}^{k}\E_\omega
	e^{\frac{\im t\xi_j}{\bms_n}}
	\]
	as $\II=\sum_{j=2}^{k}\II_j$ with 
	\[
	\II_j=\E_\omega e^{\frac{\im t\wt S_{j}}{\bms_n}}\prod_{\ell=j+1}^{k}\E_\omega
	e^{\frac{\im t\xi_\ell}{\bms_n}}-\E_\omega e^{\frac{\im t\wt
			S_{j-1}}{\bms_n}}\prod_{\ell=j}^{k}\E_\omega 
	e^{\frac{\im t\xi_\ell}{\bms_n}}.
	\]
	Note that 
	\begin{equation}\label{def:PartialSumCLT}
		|\II_j|\leq \left|\E_\omega e^{\frac{\im t\wt
				S_{j}}{\bms_n}}-\E_\omega e^{\frac{\im t\wt
				S_{j-1}}{\bms_n}}\E_\omega 
		e^{\frac{\im t\xi_j}{\bms_n}}\right|.
	\end{equation}
	We estimate the right hand side  by using conditioned 
	expectations. We denote by $\FF_j$ the $\sigma$-algebra generated by  $\{\omega_k\}_{k\leq jp+jq}$. By Lemma \ref{lemma:FromTwoTo1Sided:v2}, both $\wt h$ and the considered observables only depend on the past symbols. Therefore,  the iterates 
	$X_1,\ldots X_{jp+jq}$ are functions of $\{\omega_k\}_{k\leq jp+jq}$.
	This implies 
	\[
	\E_\omega (f(\xi_\ell)|\FF_{j-1})=f(\xi_\ell)\qquad \text{for}\qquad \ell\leq j-1.
	\]
	Then, recalling $\wt S_{j}=\wt S_{j-1}+\xi_j$, we rewrite the right hand side 
	in \eqref{def:PartialSumCLT} as 
	\[
	\begin{split}
		\E_\omega e^{\frac{\im t\wt
				S_{j}}{\bms_n}}-\E_\omega e^{\frac{\im t\wt
				S_{j-1}}{\bms_n}}\E_\omega 
		e^{\frac{\im t\xi_j}{\bms_n}}=&\E_\omega\left(\E_\omega
		\left(\left.e^{\frac{\im t\wt
				S_{j}}{\bms_n}}\right|\FF_{j-1}\right)\right)-\E_\omega \left(e^{\frac{\im t\wt
				S_{j-1}}{\bms_n}}\right)\E_\omega \left(
		e^{\frac{\im t\xi_j}{\bms_n}}\right)\\
		= &\E_\omega\left(e^{\frac{\im t\wt
				S_{j-1}}{\bms_n}}\JJ_j\right),
	\end{split}
	\]
	with
	\[
	\JJ_j=\E_\omega\left(\left.
	e^{\frac{\im t\xi_j}{\bms_n}}\right|\FF_{j-1}\right)-\E_\omega\left(
	e^{\frac{\im t\xi_j}{\bms_n}}\right).
	\]
	To estimate $\JJ_j$ we expand it. Using that $|\xi_j|\lesssim p\lesssim n^\tau$ with $\tau\in (0,1)$,
	\[
	\JJ_j=\frac{\im t}{\bms_n}\left(\E_\omega\left(\left.
	\xi_j\right|\FF_{j-1}\right)-\E_\omega\left(
	\xi_j\right)\right)-\frac{t^2}{2\bms_n^2}
	\left( \E_\omega\left(\left.
	\xi_j^2\right|\FF_{j-1}\right)-\E_\omega\left(
	\xi_j^2\right)\right)
	+\OO\left(\frac{p^3}{n^{3/2}}\right).
	\]
	Next lemma estimates these terms. The lemma is proven at the end of the section.
	\begin{lemma}\label{lemma:DecayConditionned}
		The following estimates are satisfied,
		\[
		\begin{split}
			\left|\E_\omega\left(\left.
			\xi_j\right|\FF_{j-1}\right)-\E_\omega\left(
			\xi_j\right)\right|\lesssim r^{q}\\
			\left| \E_\omega\left(\left.
			\xi_j^2\right|\FF_{j-1}\right)-\E_\omega\left(
			\xi_j^2\right)\right|\lesssim r^{q}.
		\end{split}
		\]
	\end{lemma}
	
	Using this lemma and taking into account \eqref{def:choicepq},
	\[
	|\II|\lesssim n^{1-\tau} r^{\log^2n}+n^{2\tau-\frac{1}{2}}.
	\]
	Then, taking $\tau<1/4$, one obtains that $ |\II|=o(1)$ as $n\to \infty$.
\end{proof}

Lemmas \ref{lemma:CLT:RemoveEtas} and \ref{lemma:CLT:CloseToIndep} imply that 
\[
\lim_{n\to +\infty }\E_\omega e^{\frac{\im tS_n}{\bms_n}}=\lim_{n\to +\infty }
\prod_{j=1}^{k}\E_\omega 
e^{\frac{\im t\xi_j}{\bms_n}}.
\]
Therefore, it only remains to prove that 
\[
\lim_{n\to +\infty }
\prod_{j=1}^{k}\E_\omega 
e^{\frac{\im t\xi_j}{\bms_n}} =  e^{-\frac{t^2}{2}}.
\]
To this end, we define the function 
\[
L_n(t)=\log\left[\prod_{j=1}^{k}\E_\omega 
e^{\frac{\im t\xi_j}{\bms_n}}\right]=\sum_{j=1}^k
\log\left[\E_\omega 
e^{\frac{\im t\xi_j}{\bms_n}}\right].
\]
Then, taking into account that $\E_\omega(\xi_j)=0$ and that $k\sim n/p\sim n^{1-\tau}$,
\[
\begin{split}
	L_n(t)=&\sum_{j=1}^k 
	\log\left(1+\frac{\im t}{\bms_n}\E_\omega(\xi_j)-\frac{t^2}{\bms^2_n}
	\E_\omega(\xi_j^2)+\OO\left(\frac{p^3}{n^{3/2}}\right) \right)\\
	&=-\frac{t^2}{\bms^2_n} \sum_{j=1}^k 
	\E_\omega(\xi_j^2)+\OO\left(\frac{p^2}{n^{1/2}}\right).
\end{split}
\]
Since, $p=n^\tau$ and $\tau\in (0,1/4)$, 
\[
\lim_{n\to+\infty} L_n(t)= -\frac{t^2}{2}  
\lim_{n\to+\infty}\left(\frac{\sum_{j=1}^k 
	\E_\omega(\xi_j^2)}{\bms_n^2}\right).
\]
It only remains to analyze the limit on the right hand side. 

\begin{lemma}\label{lemma:LimitVariance}
	The following limit is satisfied,
	\[
	\lim_{n\to+\infty}\left(\frac{\sum_{j=1}^k 
		\E_\omega(\xi_j^2)}{\bms_n^2}\right)=1.
	\]
\end{lemma}

\begin{proof}
	By the  definition of $\bms_n$ in \eqref{def:SumVariance} and recalling that $\E_\omega(S_n)=0$, one has 
	\[
	\begin{split}
		\bms_n^2&=\E_\omega(S_n)\\
		&=\sum_{j=1}^k\left(\E_\omega(\xi_j^2)+\E_\omega(\eta_j^2)\right)+2\sum_{j<i}\left(\E_\omega(\xi_i\xi_j)+\E_\omega(\eta_i\eta_j)\right)+2\sum_{i,j=1}^k\E_\omega(\eta_i\xi_j).
	\end{split} 
	\]
	To prove the lemma, it is enough to show that 
	\begin{equation}\label{def:ExtraTermsVar}
		\sum_{j=1}^k\E_\omega(\eta_j^2)+2\sum_{j<i}\left(\E_\omega(\xi_i\xi_j)+\E_\omega(\eta_i\eta_j)\right)+2\sum_{i,j=1}^k\E_\omega(\eta_i\xi_j)=o(n).
	\end{equation}
	Indeed, the first term can be estimated as 
	\[
	\sum_{j=1}^k  \left|\E_\omega(\eta_j^2) \right|\lesssim n^{1-\tau}\log^4n=o(n).
	\]
	For the second term, we use the decay of correlations estimate in Item 1 of Theorem \ref{thm:Dima:FixedAngle} and the fact that the $X_i's$ in different $\xi_j$'s (or $\eta_j's$) are $\gtrsim\log^2n$ separated. Therefore
	\[
	\left|\E_\omega(\xi_i\xi_j)\right|+\left|\E_\omega(\eta_i\eta_j)\right|\lesssim n^{2\tau}r^{\log^2n}, 
	\]
	wich implies 
	\begin{equation}\label{def:ExtraTermsVar2}
		\sum_{j<i}\left( \left|\E_\omega(\xi_i\xi_j)\right|+ \left|\E_\omega(\eta_i\eta_j)\right|\right)\lesssim n^{2(1-\tau)}n^{2\tau}r^{\log^2n}=o(n).
	\end{equation}
	For the last term in \eqref{def:ExtraTermsVar}, we consider two cases. If $i\neq j-1,1$ one can proceed as for the second term to obtain an estimate analogous to \eqref{def:ExtraTermsVar2}. When $i=j-1,j$ on has to proceed more carefully. We analyze only $i=j$ since $i=j-1$ can be estimated analgously. Note that 
	\[
	\xi_j\eta_j=\sum_{\ell_1=1}^p\sum_{\ell_2=1}^q X_{jp+j+q+\ell_1}X_{(j+1)p+jq+\ell_2}.
	\]
	Then,
	\[
	\begin{split}
		\left|\E_\omega( \xi_j\eta_j)\right|\leq&\,\sum_{\ell_1=1}^p\sum_{\ell_2=1}^q \left|\E_\omega(X_{jp+jq+\ell_1}X_{(j+1)p+jq+\ell_2})\right|\\
		\lesssim&\,\sum_{\ell_1=1}^p\sum_{\ell_2=1}^q r^{p+\ell_2-\ell_1}\lesssim 1, 
	\end{split}
	\]
	which implies 
	\[
	\sum_{k=1}^j\left| \E_\omega( \xi_j\eta_j)\right|\lesssim n^{1-\tau}.
	\]
	This gives the estimate for the third term in \eqref{def:ExtraTermsVar} and therefore completes the proof of the lemma.
\end{proof}

Now it only remains to prove lemma \ref{lemma:DecayConditionned}.
\begin{proof}[Proof of Lemma \ref{lemma:DecayConditionned}]
	We first obtain the first estimate. Note that, by hypothesis, $\E_\omega(\xi_j)=0$.
	Then, note that, by \eqref{def:xieta},
	\[
	\E_\omega(\xi_j|\FF_{j-1})-\E_\omega(\xi_j)=\sum_{m=jp+jq+1}^{(j+1)p+jq}\E_\omega(X_m|\FF_{j-1}).
	\]
	The random variables $X_m$ only have harmonics $\pm 1$. Conditioning by $\FF_{j-1}$ corresponds to fixing the values of the symbols $\{\omega_0,\ldots, \omega_{jp+(j-1)q}\}$. Therefore, the first estimate in Lemma \ref{lemma:DecayConditionned} is a direct consequence of the second statement of Theorem \ref{thm:Dima:FixedAngle}.
	
	Now we prove the second estimate. Note that 
	\begin{equation}\label{def:xideg2}
		\xi_j^2=\sum_{m=jp+jq+1}^{(j+1)p+jq}X_m^2+2\sum_{m,k=jp+jq+1, m<k}^{(j+1)p+jq}X_mX_k.
	\end{equation}
	Each of these terms have harmonics $0$ and $\pm 2$. Let us start by the harmonic $0$. It contains terms of the form 
	\[
	\PP_0= P(\sigma^k\omega)\ol P(\sigma^m\omega)e^{\im \sum_{i=m}^{k-1}\beta(\sigma^i\omega)}.
	\]
	We need to estimate the term
	\[
	\E_\omega\left(\left. \PP_0\right|\FF_{j-1}\right)-\E_\omega\left( \PP_0\right).
	\]
	Now, we write $\E_\omega=\E_{\omega_\leq}\E_{\omega_>}$, where $\E_{\omega_>}$ refers to the expectation with respect to the symbols $\{\omega_k\}_{k\geq jp+(j-1)q+1}$ and $\E_{\omega_\leq}$ refers to the expectation with respect to $\{\omega_0,\ldots, \omega_{jp+(j-1)q}\}$. Then, we can write 
	\[
	\begin{split}
		\E_\omega\left(\left. \PP_0\right|\FF_{j-1}\right)-\E_\omega\left( \PP_0\right)
		&=  \E_{\omega,>}(\PP_0)-\E_{\omega,>}\E_{\omega,\leq}(\PP_0)\\
		&=\E_{\omega,>}\left[\PP_0-\E_{\omega,\leq}(\PP_0)\right].
	\end{split}
	\]
	Then, since $\PP_0$ is H\"older with respect to $\omega$, one has that there exists $r>0$ such that 
	\[
	\left|\PP_0-\E_{\omega,\leq}(\PP_0)\right|\lesssim r^{\log^2n},
	\]
	which implies 
	\[
	\left|\E_\omega(\PP_0|\FF_{j-1})-\E_\omega(\PP_0)\right|\lesssim r^{\log^2n}.
	\]
	Now we consider the harmonic 2 in \eqref{def:xideg2} (the harmonic -2 can be estimated analogously). It contains terms of the form 
	\[
	\PP_2= P(\sigma^k\omega)\ol P(\sigma^m\omega)e^{\im \sum_{i=m}^{k-1}\beta(\sigma^i\omega)}e^{2\im \sum_{i=0}^{m-1}\beta(\sigma^i\omega)}.
	\]
	We consider two cases depending on $|k-m|$. 
	
	Note that we have two parameters: $r$ given by the decay of correlations (see Item 2 in Theorem \ref{thm:Dima:FixedAngle}) and $\vartheta$ coming from the $\vartheta$-H\"older dependence of $P$ and $\beta$ on $\omega$. We choose $\kk>0$ such that 
	\[
	\vartheta^{-\kk}r\leq \wt r<1.
	\]
	Then, we consider first $k-m\leq \kk q$ and then $k-m\geq \kk q$. 
	
	For the first case, we apply Item 2 of Theorem \ref{thm:Dima:FixedAngle} to the functions 
	\[
	\phi(\al,\omega)=Q(\omega)e^{2\im \al}+ \ol{Q}(\omega)e^{-2\im \al},\quad \psi(\al,\omega)=e^{2\im \al}+e^{-2\im \al},
	\]
	where
	\[
	Q(\omega)=P(\sigma^{k-m}\omega)P(\omega)e^{2\im \sum_{i=0}^{k-m-1}\beta(\sigma^i\omega)}.
	\]
	Then,
	\[
	|\E_\omega(\PP_2|\FF_{j-1})|\leq \|Q\|_\vartheta r^q.
	\]
	Since $\|Q\|_\vartheta\leq \vartheta^{-|k-m|}$ and $k-m\leq \kk q$, we have
	\[
	|\E_\omega(\PP_2|\FF_{j-1})|\leq \vartheta^{-\kk q}r^q \leq \wt r^{\log^2n}.
	\]
	For the second case, $k-m\geq \kk q$, we further condition the expectation to 
	\[
	\E_\omega(\PP_2|\FF_{j-1})=\E_\omega\left(\E_\omega(\PP_2|\{\omega_k\}_{k\leq m-1})|\FF_{j-1}\right).
	\]
	Note that,
	\[
	\E_\omega(\PP_2|\{\omega_k\}_{k\leq m-1})= e^{\im \sum_{i=0}^{m-1}\beta(\sigma^i\omega)}\E_\omega\left(\left. P(\sigma^k\omega)\ol P(\sigma^m\omega)e^{2\im \sum_{i=m}^{k-1}\beta(\sigma^i\omega)}\right|\{\omega_k\}_{k\leq m-1}\right).
	\]
	Now, for any fixed values for $\{\omega_k\}_{k\leq m-1}$, one can apply Item 2 of Theorem \ref{thm:Dima:FixedAngle} to obtain 
	\[
	\left|\E_\omega\left(\left. P(\sigma^k\omega)\ol P(\sigma^m\omega)e^{2\im \sum_{i=m}^{k-1}\beta(\sigma^i\omega)}\right|\{\omega_k\}_{k\leq m-1}\right)\right|\leq r^{k-m}.
	\]
	Since, by hypothesis $k-m\geq \kk q$, we obtain 
	\[
	|\E_\omega(\PP_2|\FF_{j-1})|\leq r^{\log^2 n}.
	\]
	Note that one can obtain analogous estimates for $\E_\omega(\PP)$ just writing
	\[
	\E_\omega\left( \E_\omega(\PP_2|\FF_{j-1})\right).
	\]
\end{proof}

\appendix

\section{Wisdom mechanism}\label{wisdom}

The first explanation of formation of the Kirkwood gaps was
proposed by Wisdom \cite{Wisdom82} and
Neishtadt \cite{Neishtadt87} (see also \cite{NeishtadtS04}). Suppose that  the mass ratio between Jupiter and the Sun  and
eccentricity of Jupiter satisfy
\[
\frac{\sqrt \mu}{e_0}\ \ll\  1,
\]
then the dynamics in the mean motion resonances of the restricted
planar 3-body problem can be described by a system with
three time scales: fast, intermediate and slow. According to Wisdom
in the main approximation, the secular evolution of the asteroid motion
near the 3:1 mean motion resonance with Jupiter is described by the canonical equations
\[
\dot \varphi =\{\varphi,H_P\},\qquad
\dot \Phi =\{\Phi,H_P\}, \qquad\dot x =\{y,H_P\}, \qquad\dot y =\{x,H_P\}.
\]
Here $\varphi$ is the critical angle (three times the mean longitude
of Jupiter minus the mean longitude of the asteroid), $\Phi$ is some
function of the asteroid’s semimajor axis $a$, $x$ and $y$ are
proportional to the Laplace vector components
\[
x\approx a^{1/4} \left( \frac{\ecc}{e_0} \right) \cos \hat g,  \qquad
x\approx a^{1/4} \left( \frac{\ecc}{e_0} \right) \sin \hat g,
\]
where $\hat g$ is the argument of the asteroid perihelion.

The Poisson brackets $\{\cdot,\cdot\}$ in the above Hamiltonian
equations are defined in a way such that they split the components into slow and fast,
\[
\{\varphi,\Phi\}=1, \qquad \{x,y\}=\varepsilon,\qquad
\{\varphi,x\}=\{\varphi,y\}=\{\Phi,x\}=\{\Phi,y\}=0.
\]
The averaged Hamiltonian $H_P$, given by Wisdom \cite{Wisdom82},
can be rewritten as
\[
H_P(\Phi,\varphi,x,y)=\frac 12 \alpha \Phi^2+
A(x,y) \cos \varphi+
B(x,y)\sin \varphi+C(x,y),
\]
where $A(x,y), B(x,y), C(x,y)$ are quadratic polynomials in $x,y$ (
see \cite{NeishtadtS04}).

Observe that the Hamiltonian $H_P$ has two fast components
$(\Phi,\varphi)$ and two slow components $(x,y)$. If the fast
components are on a level set $E_f$ not passing through its
equilibrium there is a second averaged system, giving a slow
Hamiltonian $H_{\text{slow}}(x,y,E_f)$ and describing an evolution
of the slow component.

It is, however, possible that the fast components $(\Phi,\varphi)$,
which are a pendulum, approach the critical level set and, as a result, approach 
the saddle equilibrium. The set of values of $(x,y)$ which correspond to this 
phenomenon forms the so-called {\it uncertainty curve $\Gamma$} (see the curve 
in red in Figure \ref{fig:wisdom-slow-3/1}). Each time the slow component crosses 
$\Gamma$, there is a redistribution of slow and fast energy. See Figure \ref{fig:wisdom-3/1}
for the associate evolution of the eccentricity. 

  \begin{figure}[h]
   \begin{center}
   \includegraphics[width=6.25cm]{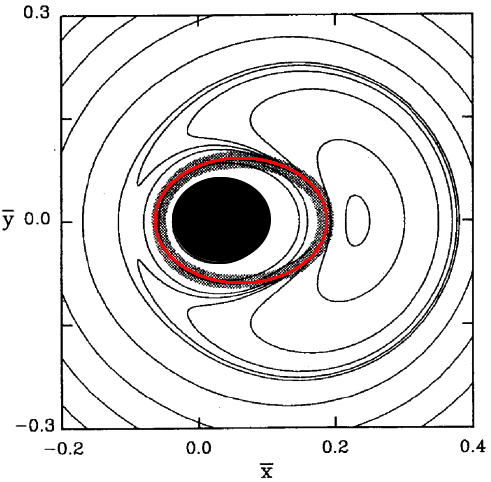}
   \end{center}
   \caption{Slow evolution and uncertainty curve (see e.g. \cite{Morb}).}
   \label{fig:wisdom-slow-3/1}
 \end{figure}

On the plane of the slow component in Figure \ref{fig:wisdom-slow-3/1}, by definition,
the radial component corresponds to the rescaled eccentricity
of the asteroid.  Crossing the ``red'' uncertainty curve can
lead to the slow component  getting outside of the ``central''
island and evolve in the exterior of it. This leads to a drastic
increase in eccentricity and to crossings with the orbit of Mars.

Even assuming that $\sqrt \mu/e_0$ is small, numerical simulations indicate that the resonances 4:1, 5:2 and 7:3  are have mixed behavior, which includes chaotic regions (see \cite{Morb, Moons199533}). In the chaotic regions,  the asteroids in these mean motion resonances can reach
very large eccentricities on a ``short” timescale:
$0.6\cdot 10^6$ and $19\cdot 10^6$ for the 5:2 and 7:3 resonances respectively (see also \cite{Gladman1997197}). However, it is hard to agree that these resonances are dominantly chaotic
as 7:3 has "plenty" of asteroids left (see Figure \ref{fig:distribution}). It is believed that in different gaps there are different
mechanisms of instabilities of the eccentricity (see \cite{Moons1996175}).

\section{Fast drift along the resonance}\label{app:fastdrift}

Theorem \ref{thm:main-thm} provides stochastic behavior in the energy drift at time scales $\sim e_0^{-2}$. Relying on the same analysis, one can construct orbits that drift at much faster pace.

\begin{theorem}\label{thm:fastdrift}
Assume Ansatz \ref{ans:NHIMCircular:bis} and \ref{ansatz:Melnikov:1} hold and consider the interval $[\tJ_-,\tJ_+]$ introduced in Remark \ref{rmk:intervals}.
Then, there exists $e_0^*>0$ and $C>0$ such that for all $e_0\in (0,e_0^*)$,  the Hamiltonian $\HH_\mu$ given by \eqref{eq:RPE} with $\mu=0.95387536\times10^{-3}$ has an orbit $(q(t), p(t))$ such that 
\[
\HH_\mu(q(0), p(0),0)<\tJ_-  \qquad \text{and}\qquad \HH_\mu(q(T), p(T),T)>\tJ_+
\]
for some  
\[\frac{1}{Ce_0}\leq T \leq\frac{C}{e_0}.\]
Moreover, along this orbit the osculating eccentricity satisfies
\[
\ecc(0)<0.676  \qquad \text{and}\qquad \ecc(T)>0.77.
\]
\end{theorem}

Note that the drift in energy/eccentricty was already obtained in \cite{FejozGKR15}. However in that paper no time estimates were provided. These estimates can be deduced from the analysis of the separatrix map done in Theorem  \ref{thm:FormulasSMii} (see also \cite{paperNHIL23}) following the ideas developed by Piftankin in \cite{Piftankin06}. Indeed, both our setting and that of Piftankin, who deals with the so-called Mather Problem, fall into what is usually called the \emph{a priori chaotic} setting and present exactly the same features. In \cite{Piftankin06} Piftankin, relying on the formulas he obtains for the separatrix map (analogous to those in Theorem \ref{thm:FormulasSMii}), constructs fast diffusive orbits. The same proof can be applied to prove Theorem \ref{thm:fastdrift}. Since the goal of the present paper is to obtain diffusive behavior rather than fast drift, we do not reproduce the proof here.

\section{Large deviation estimates on the exit time (Proposition \ref{lemma:exittime:Igamma})} \label{app:exittime}
We first prove the first statement.
Note that $\Prob\{n_\ga<e_0^{-(1-\ga)}/C\}=0$ since
	$|J_{k+1}-J_k|\leq Ce_0$ for some $C>0$ independent of $e_0$. Therefore,
	one needs at least
	$\lceil e_0^{-(1-\ga)}/C\rceil$ iterations to exit the strip. Thus, we only need to analyze
	$\Prob\{e_0^{-(1-\ga)}/C\leq n_\ga<e_0^{-2(1-\ga)+\delta}\}$, which is
	equivalent to
	\[
	\Prob\{\exists\, n\in [e_0^{-(1-\ga)}/C,e_0^{-2(1-\ga)+\delta}):
	|J_n-J_0|\geq e_0^\ga\}.
	\]
	We have that, for $0<n\leq n_\gamma$, 
	\[J_{n}=J_{0}+e_0
	\sum_{k=0}^{n-1}
	A_{\sigma^{k}\omega}(\theta_{k},J_{k})+\mathcal{O}( n e_0^2).\]
	Taking also into account that
	$\theta_{k}=\theta_{0}+\sum_{j=0}^{k-1}\beta_{\sigma^{j}\omega}(J_{0})+\mathcal{O}
	( n^2e_0)$ for $0\leq k\leq  n$, we can write
		\begin{equation}\label{eqHni}
		J_{n}=J_{0}+e_0
		\sum_{k=0}^{ n-1}
		A_{\sigma^{k}\omega}\left(\theta_{0}+\sum_{j=0}^{k-1}\beta_{\sigma^{j}\omega}(J_{0}),J_{0}
		\right)+\mathcal{O}( n^3e_0^2).
	\end{equation}
	Define
	\begin{equation}\label{def:xi}
		\xi_n=\frac{1}{\sqrt{n}}\sum_{k=0}^{n-1}
		A_{\sigma^{k}\omega}\left(\theta_{0}+\sum_{j=0}^{k-1}\beta_{\sigma^{j}\omega}(J_{0}),J_{0}
		\right).
	\end{equation}
	Then, for $e_0>0$ small enough and  $n\in [e_0^{-(1-\ga)}/C,e_0^{-2(1-\ga)+\delta})$,
	\[
	\begin{split}
		\Prob\left\{ |J_{n}-J_0|\geq e_0^\ga\right\}&=
		\Prob\left\{\left|e_0\sum_{k=0}^{n-1}
		A_{\sigma^k \omega}\left(\theta_{n_i}+\sum_{j=0}^{k-1}\beta_{\sigma^{j}\omega}(J_{n_i}),J_{n_i}
		\right)+\OO\left(e_0^2n^3\right)\right|\geq e_0^\ga\right\} \\
		&\leq \Prob\left\{\left|\xi_n+\OO(e_0
		n^{5/2})\right|\geq e_0^{\ga-1}n^{-1/2}\right\}.
	\end{split}
	\]
Now, using
	that $\de\in (0, 1-\ga)$ and $n\in
	[e_0^{-(1-\ga)}/C,e_0^{-2(1-\ga)+\delta})$
	we have that
	\[
	\Prob\left\{\left|\xi_n+\OO\left(e_0
	n^{5/2}\right)\right|\geq e_0^{\ga-1}n^{-1/2}\right\}\leq
	\Prob\left\{|\xi_n|\geq\frac{e_0^{-\de/2}}{2}\right\}.
	\]
	We now  analyze $\xi_{n}$ as $n\to\infty$ (equivalently $e_0\to 0$) by using the
	Central Limit Theorem given by Theorems \ref{thm:Dima} and \ref{thm:Dima:FixedAngle}, which can be applied by   \eqref{ansatz:omegadifferent} and Lemma \ref{lemma:WeakMixing}.
	Then, $\xi_n$ converges weakly to a normal distribution $\NNN(0,\bms^2(J))$ 
	where $\bms^2(J)$ is the variance introduced in \eqref{def:variance}, which satisfies 
	$\bms^2(J)\neq 0$ for  $J\in [I_--\de, I_++\de]$ by Ansatz \ref{ansatz:Melnikov:1} (see also Lemma \ref{lemma:DriftVar}).
	
Thus,
	\[
	\Prob\left\{ |J_{n}-J_0|\geq e_0^\ga\right\} \leq e^{-\frac{C}{e_0^\de}}
	\]
	for some $C>0$ independent of $e_0$. Then, since $\sharp
	[e_0^{-(1-\ga)}/C,e_0^{-2(1-\ga)+\delta})\sim e_0^{-2(1-\ga)+\delta}$,
	\[
	\Prob\{\exists\, n\in [e_0^{-(1-\ga)}/2,e_0^{-2(1-\ga)+\delta}):
	|J_n-J_0|\geq e_0^\ga\}\leq e^{-\frac{C}{e_0^\de}},
	\]
	taking a smaller $C>0$.

Now we prove the second statement of the proposition.  Let $\wt n_\ga=[e_0^{-2(1-\ga)}],$
$n_\delta=[e_0^{-\delta}]$, and
$n_i=i\wt n_\ga$. Then,
\begin{equation}\label{probN-D3}
		\Prob\left\{n_\ga>e_0^{-2(1-\ga)-\delta}\right\}\leq\,\Prob\left\{|J_{n_{i+1}}
		-J_ { n_i } |\leq
		e_0^\ga\,\textrm{ for all }i=0,\dots, n_\delta-1\right\}
\end{equation}
Proceeding as for the first statement of the proposition, we have that 
	\begin{equation}\label{eqHni:2}
	J_{n_{i+1}}=J_{n_i}+e_0 M_{i}
	+\mathcal{O}(\wt n_\ga^3e_0^2).
\end{equation}
where
\[
M_{i}=\sum_{k=0}^{\wt n_\ga-1}
A_{\sigma^{n_i+k}\omega}\left(\theta_{n_i}+\sum_{j=0}^{k-1}\beta_{\sigma^{n_i+j}\omega}(J_{n_i}),J_{n_i}
\right).
\]
Therefore we have that 
\begin{equation}\label{def:prob:exittime:ind}
	\Prob\left\{n_\ga>e_0^{-2(1-\ga)-\delta}\right\}\leq\,\Prob\left\{\left|M_{i}\right|\leq
	2e_0^{\ga-1}\,	\textrm{ for all }i=0,\dots, n_\delta-1\right\}.
\end{equation}
Note that the random variables $M_{i}$, which depend on $\omega$, are not independent. We use conditionned probabilities to estimate the right hand side above.
The first observation is that, proceeding as in Section~\ref{sec:fixedtheta} (see Lemma~\ref{lemma:FromTwoTo1Sided:v2}), we can assume that $M_{i}$ only depends on $\{\omega_k\}_{k\leq n_i+\wt n_\ga-1}$.

We estimate the probability in the right hand side of \eqref{def:prob:exittime:ind} by induction. Thus, for $j=0,\ldots, n_\de-1$, we define
\[
P_j=\Prob\left\{\left|M_{i}\right|\leq
e_0^{\ga-1}\,	\textrm{ for all }i=0,\dots, j\right\}
\]
and we prove that there exists a constant $C\in (0,1)$ such that, for any $j=0,\ldots, n_\de-1$ and $e_0$ small enough,
\begin{equation}\label{def:exittimeinduction}
P_j\leq C^{j+1}.
\end{equation}
To this end, we define the $\frac{1}{2}$-Bernoulli measure $b_k$ for the symbol $\omega_k\in \{0,1\}$ and the Bernoulli product measure
\[
\nu_i=\prod_{k\leq n_i+\wt n_\ga-1} b_k
\]
defined on the space of sequences 
\[
X_i=\left\{\omega_i^\leq=\{\omega_k\}_{k\leq n_i+\wt n_\ga-1}, \omega_k\in \{0,1\}\right\}.
\]
Note that 
\[
\nu_{i+1}=b_{n_{i+1}+\wt n_\ga-1}\times\ldots b_{n_i+\wt n_\ga}\times \nu_i.
\]

We first prove the inductive statement for $j=0$. Note that 
\[
P_0=\Prob\left\{\left|M_{0}\right|\leq
e_0^{\ga-1}\right\}, 
\]
which, relying on the random variable $\xi_n$ introduced in \eqref{def:xi}, satisfies
\[
P_0=\Prob\left\{\left|\xi_{\wt n_\ga}\right|\leq 3\right\}.
\]
Then, proceeding as in the proof of the first statement of Proposition  \ref{lemma:exittime:Igamma} and using the Central Limit Theorems provided in Theorems \ref{thm:Dima} and \ref{thm:Dima:FixedAngle}, there exists a constant $C\in (0,1)$ such that, for $e_0>0$ small enough,
\[
P_0\leq C.
\]
Now we assume that we have proven \eqref{def:exittimeinduction} for some  $j=0,\ldots, n_\gamma-2$ and we prove it for $j+1$.

Then, by conditioning the probability, 
\[
P_{j+1}=\int_{X_i} \Prob\left\{\left.\left|M_{i}\right|\leq
e_0^{\ga-1}\,	\textrm{ for all }i=0,\dots, j+1\right|\omega_i^\leq\right\} d\nu_j.
\]
Now, since $M_{i}$ for $i=0,\dots, j$ are independent of $\omega_k\in [n_i+\wt n_\ga, n_{i+1}+\wt n_\ga-1]$, this probability can be written as 
\[
\begin{split}
P_{j+1}=\int_{X_i}&\Bigg[ \Prob\left\{\left.\left|M_{j+1}\right|\leq
e_0^{\ga-1}\,	\textrm{ for all }i=0,\dots, j+1\right|\omega_i^\leq\right\}\\ 
&\times \Prob\left\{\left.\left|M_{i}\right|\leq
e_0^{\ga-1}\,	\textrm{ for all }i=0,\dots, j\right|\omega_i^\leq\right\}  \Bigg] d\nu_j.
\end{split}
\]
Now, proceeding as in the case $j=0$, 
\[
\Prob\left\{\left.\left|M_{j+1}\right|\leq
e_0^{\ga-1}\,	\textrm{ for all }i=0,\dots, j+1|\right|\omega_i^\leq\right\}\leq C
\]
and therefore
\[
P_{j+1}=C P_j,
\]
which completes the proof of the inductive statement \eqref{def:exittimeinduction}. 

Now, since $n_\delta=[e_0^{-\delta}]$, \eqref{def:exittimeinduction} implies the second statement of Proposition \ref{lemma:exittime:Igamma}.

\bibliography{references}
\bibliographystyle{alpha}
\end{document}